\documentclass[11pt,reqno]{amsart}
\usepackage[a4paper, hmargin={2.7cm,2.7cm},vmargin={3.3cm,3.3cm}]{geometry}
\usepackage{amsmath,amssymb,amsthm,amsfonts,amsopn,todonotes,mathtools}
\usepackage[nospace,noadjust]{cite}
\usepackage{dsfont}
\usepackage{enumitem}
\usepackage[utf8]{inputenc}
\usepackage{mathrsfs}
\usepackage{comment}

\usepackage{color,url,setspace,wasysym}
\usepackage{lscape}
\usepackage{hyperref}

\theoremstyle{plain}
\newtheorem{theorem}{Theorem}[section]
\newtheorem{proposition}[theorem]{Proposition}
\newtheorem{lemma}[theorem]{Lemma}
\newtheorem{corollary}[theorem]{Corollary}
\theoremstyle{definition}
\newtheorem{definition}[theorem]{Definition}

\theoremstyle{remark}
\newtheorem{remark}[theorem]{Remark}

\theoremstyle{remark}
\newtheorem*{rem*}{Remark}

\makeatletter
\DeclareFontFamily{OMX}{MnSymbolE}{}
\DeclareSymbolFont{MnLargeSymbols}{OMX}{MnSymbolE}{m}{n}
\SetSymbolFont{MnLargeSymbols}{bold}{OMX}{MnSymbolE}{b}{n}
\DeclareFontShape{OMX}{MnSymbolE}{m}{n}{
    <-6>  MnSymbolE5
   <6-7>  MnSymbolE6
   <7-8>  MnSymbolE7
   <8-9>  MnSymbolE8
   <9-10> MnSymbolE9
  <10-12> MnSymbolE10
  <12->   MnSymbolE12
}{}
\DeclareFontShape{OMX}{MnSymbolE}{b}{n}{
    <-6>  MnSymbolE-Bold5
   <6-7>  MnSymbolE-Bold6
   <7-8>  MnSymbolE-Bold7
   <8-9>  MnSymbolE-Bold8
   <9-10> MnSymbolE-Bold9
  <10-12> MnSymbolE-Bold10
  <12->   MnSymbolE-Bold12
}{}

\let\llangle\@undefined
\let\rrangle\@undefined
\DeclareMathDelimiter{\llangle}{\mathopen}{MnLargeSymbols}{'164}{MnLargeSymbols}{'164}
\DeclareMathDelimiter{\rrangle}{\mathclose}{MnLargeSymbols}{'171}{MnLargeSymbols}{'171}
\makeatother

\numberwithin{equation}{section}

\newcommand{\RR}{\mathbb R}
\newcommand{\R}{\mathbb R}
\newcommand{\RHat}{\widehat{\mathbb{R}}}
\newcommand{\ZZ}{\mathbb Z}
\newcommand{\Z}{\mathbb Z}
\newcommand{\NN}{\mathbb N}
\newcommand{\N}{\mathbb N}
\newcommand{\CC}{\mathbb C}
\newcommand{\eps}{\varepsilon}

\newcommand{\localized}{\natural}
\newcommand{\normalizedBAPU}{\flat}
\newcommand{\normalizedFactorisation}{\diamond}

\newcommand{\newA}{\mathrm{A}}
\newcommand{\newB}{\mathrm{B}}

\newcommand{\analysis}{\mathscr{C}}
\newcommand{\synthesis}{\mathscr{D}}

\def\esssup{\mathop{\operatorname{ess~sup}}}

\def\supp{\mathop{\operatorname{supp}}}

\makeatletter
\newcommand{\LeftEqNo}{\let\veqno\@@leqno}

\DeclareFontFamily{U}{mathx}{\hyphenchar\font45}
\DeclareFontShape{U}{mathx}{m}{n}{
      <5> <6> <7> <8> <9> <10>
      <10.95> <12> <14.4> <17.28> <20.74> <24.88>
      mathx10
      }{}
\DeclareSymbolFont{mathx}{U}{mathx}{m}{n}
\DeclareFontSubstitution{U}{mathx}{m}{n}
\DeclareMathAccent{\widecheckBlub}{0}{mathx}{"71}
\newcommand{\widecheck}[1]{\widecheckBlub{#1}}
\DeclareMathAccent{\wideparen}{0}{mathx}{"75}

\newcommand{\Indicator}{{\mathds{1}}}

\newcommand{\Fourier}{\mathcal{F}}
\newcommand{\DecompSp}{\mathcal{D}}
\newcommand{\StandardDecompSp}{\DecompSp(\CalQ,L^p,\ell_w^q)}

\newcommand{\DenseSpace}{\Schwartz_{\CalO}}
\newcommand{\CalQ}{\mathcal{Q}}

\newcommand{\CalO}{\mathcal{O}}
\newcommand{\CalD}{\mathcal{D}}
\newcommand{\CalB}{\mathcal{B}}
\newcommand{\GL}{\mathrm{GL}}
\newcommand{\DistributionSpace}{\mathcal{D}'}
\newcommand{\TestFunctionSpace}{C_c^\infty}
\newcommand{\Schwartz}{\mathcal{S}}

\newcommand{\identity}{\mathrm{id}}

\newcommand{\with}{\,:\,}

\newcommand{\mybullet}{\cdot}
\newcommand{\vertiii}[1]{{\left\vert \kern-0.25ex
                            \left\vert \kern-0.25ex
                              \left\vert #1\right\vert\kern-0.25ex
                            \right\vert \kern-0.25ex
                          \right\vert}}

\newcommand{\FirstN}[1]{\underline{#1}}
\newcommand{\Translation}[1]{T_{#1} \, }
\newcommand{\Modulation}[1]{M_{#1} \, }
\newcommand{\StandardGSI}{(\Translation{\gamma}g_j)_{j\in J,\gamma\in C_j\ZZ^d}}

\newcommand{\schur}{\mathrm{Schur}}

\usepackage{fancyhdr}
\usepackage{lastpage}

\newcommand\numberthis{\addtocounter{equation}{1}\tag{\theequation}}
\newcommand{\op}{\mathrm{op}}

\let\emptyset\varnothing

\begin{document}
\author{Jos\'e Luis Romero}
\address{Faculty of Mathematics,
University of Vienna,
Oskar-Morgenstern-Platz 1,
A-1090 Vienna, Austria\\and
Acoustics Research Institute, Austrian Academy of Sciences,
Wohllebengasse 12-14 A-1040, Vienna, Austria}
\email{jose.luis.romero@univie.ac.at}

\author{Jordy Timo van Velthoven}
\address{Faculty of Mathematics,
University of Vienna,
Oskar-Morgenstern-Platz 1,
A-1090 Vienna, Austria \\ and
Delft University of Technology,
Mekelweg 4, Building 36,
2628 CD Delft, The Netherlands}
\email{jordy-timo.van-velthoven@univie.ac.at, j.t.vanvelthoven@tudelft.nl}

\author{Felix Voigtlaender}
\address{Katholische Universität Eichstätt-Ingolstadt,
Lehrstuhl Reliable Machine Learning, Ostenstraße 26, 85072 Eichstätt}
\email{felix.voigtlaender@ku.de, felix@voigtlaender.xyz}

\thanks{\emph{Acknowledgments.}
J.~L.~R.~gratefully acknowledges support from the Austrian Science Fund (FWF): Y 1199 and P 29462-N35, and from the WWTF grant INSIGHT (MA16-053).
J.~v.~V.~acknowledges support from the Austrian Science Fund (FWF): P 29462-N35.
J.~v.~V.~ is grateful for the hospitality and support
of the Katholische Universit\"at Eichst\"att-Ingolstadt during his visit.}

\title[Invertibility of frame operators on Besov-type decomposition spaces]%
      {Invertibility of frame operators on Besov-type decomposition spaces}



\begin{abstract}
We derive an extension of the Walnut-Daubechies criterion for the invertibility of frame operators.
The criterion concerns general reproducing systems and Besov-type spaces.
As an application, we conclude that $L^2$ frame expansions associated with smooth and fast-decaying
reproducing systems on sufficiently fine lattices extend to Besov-type spaces.
This simplifies and improves recent results on the existence of atomic decompositions,
which only provide a particular dual reproducing system with suitable properties.
In contrast, we conclude that the $L^2$ canonical frame expansions
extend to many other function spaces, and, therefore,
operations such as analyzing using the frame, thresholding the resulting coefficients,
and then synthesizing using the canonical dual
frame are bounded on these spaces.
\end{abstract}

\maketitle

\section{Introduction}
Given a countable collection $(g_j)_{j \in J}$ of functions $g_j: \mathbb{R}^d \to \mathbb{C}$
and a collection $(C_j)_{j \in J}$ of matrices $C_j \in \GL(d, \RR)$,
we consider the structured function system
\begin{align}
  \StandardGSI
  = \big( g_j (\cdot - \gamma) \big)_{j \in J, \gamma \in C_j \ZZ^d},
  \label{eq:GSI_intro}
\end{align}
and aim to represent a function or distribution $f$ as a linear combination
\begin{align}
  f = \sum_{j \in J}
        \sum_{\gamma \in C_j \ZZ^d}
          c_{j, \gamma} \, \Translation{\gamma} g_j.
  \label{eq:series_intro}
\end{align}
In many important examples of this formalism,
the functions $g_j$ are obtained through affine transforms (in the Fourier domain)
of a single function $g$.
For instance, in dimension $d=1$, the well-known \emph{wavelet} \cite{DaubechiesTenLectures}
and \emph{Gabor} systems \cite{GroechenigTimeFrequencyAnalysis} are obtained as
\begin{alignat}{5}
  g_j(x) &:= 2^{j/2} \, g(2^j x), && \qquad j \in \ZZ, && \qquad C_j = 2^j,
  \label{eq_wavelets}
  \\
  g_j(x) &:= e^{2 \pi i j x} \, g(x), && \qquad j \in \beta \ZZ, && \qquad C_j = \alpha.
  \label{eq_gabor}
\end{alignat}
For $d>1$, \emph{anisotropic wavelet systems} provide additional
important examples, see e.g., \cite{MR1982689,MR2082156,MR3189276}.

We are interested in the ability of \eqref{eq:GSI_intro} to reproduce all functions or distributions
$f$ in various function spaces by a suitably convergent series \eqref{eq:series_intro}.
For the Hilbert space $L^2(\RR^d)$ this task is significantly easier:
 it amounts to establishing the \emph{frame inequalities}
\begin{align} \label{eq:frame_normequiv}
  \| f \|_{L^2}^2
  \asymp \sum_{j \in J}
           \sum_{\gamma \in C_j \mathbb{Z}^d}
             \big| \big\langle f \,\mid\, \Translation{\gamma} g_j \big\rangle \big|^2
             \qquad \forall \, f \in L^2 (\RR^d).
\end{align}
Indeed, the norm equivalence \eqref{eq:frame_normequiv} means that the \emph{frame operator}
$S : L^2 (\RR^d) \to L^2 (\RR^d)$,
\begin{align*}
S f := \sum_{j \in J} \sum_{\gamma \in C_j \ZZ^d}
\big\langle f \,\mid\, \Translation{\gamma} g_j \big\rangle
\,
\Translation{\gamma} g_j
\end{align*}
is bounded and invertible on $L^2(\RR^d)$, and consequently \eqref{eq:series_intro} holds with
$c_{j, \gamma} = \big\langle S^{-1} f \,\mid\,  \Translation{\gamma} g_j \big\rangle$.

The validity of the frame inequalities is closely related to the covering properties
of the Fourier transforms of the generating functions $\widehat{g_j}$,
which is encoded in the \emph{Calder\'on condition}:
\begin{align}
   \sum_{j \in J} \frac{1}{|\det C_j|} \big| \widehat{g_j} \big|^2 \asymp 1,
   \qquad \mbox {a.e.}
   \label{eq:GSI_assumption_introduction}
\end{align}
This connection is most apparent in the so-called \emph{painless case}, in which the
supports of the functions $\widehat{g_j}$ are compact. Under this assumption,
the expansion \eqref{eq:series_intro} is a \emph{local Fourier expansion}
\begin{align}
  \widehat{f}(\xi)
  = \sum_{j \in J}
      \sum_{\gamma \in C_j \ZZ^d}
        c_{j, \gamma} \, e^{-2 \pi i \gamma \xi} \, \widehat{g_j} (\xi).
  \label{eq:series_2}
\end{align}
In many important cases, the functions $g_j$ are not bandlimited, but
have a well concentrated frequency profile, such as a Gaussian.
Then \eqref{eq:series_2} is an almost-local Fourier expansion,
that one still expects to be governed by \eqref{eq:GSI_assumption_introduction}---and,
indeed, under mild conditions, \eqref{eq:GSI_assumption_introduction} is necessary
for \eqref{eq:frame_normequiv} to hold \cite{christensen2017explicit, Fuehr2019System}.

The formal analysis of non-painless expansions with a reproducing system \eqref{eq:GSI_intro}
relies on a remarkable representation of the frame operator
in the Fourier domain, namely
\begin{align}
\label{eq_intro_rep}
\widehat{S f}(\xi) = \sum_{\alpha \in \Lambda} t_\alpha(\xi-\alpha) \widehat{f}(\xi-\alpha),
\end{align}
where
\(
  t_{\alpha}(\xi)=
     \sum_{j \in \kappa (\alpha)}
                  \frac{1}{|\det C_j|} \,\,
                  \overline{\widehat{g}_j (\xi)}
                  \, \widehat{g_j} (\xi + \alpha)
\);
here, the translation nodes $\Lambda \subseteq \RR^d$ and indices
${\kappa(\alpha) \subseteq J}$ are determined by the matrices $C_j$
(see \eqref{eq:TAlphaDefinition} below).
For Gabor expansions, the representation \eqref{eq_intro_rep} is known under the name of
Walnut's representation \cite{MR1155734} while for wavelets it is attributed to Daubechies
and Tchamitchian \cite[Chapter 3]{DaubechiesTenLectures}.
The theory of \emph{generalized shift-invariant systems} \cite{hernandez2002unified,ron2005generalized}
establishes the general form of \eqref{eq_intro_rep} and exploits its many consequences.
For example, tight frames---that is, systems for which equality holds
in \eqref{eq:frame_normequiv}---are characterized by a set of algebraic relations
involving the functions $t_{\alpha}$; see \cite{hernandez2002unified}.

\subsection{The Walnut-Daubechies criterion}

The multiplier $t_0$ associated with $\alpha=0$ in \eqref{eq_intro_rep} is precisely
the Calder\'on sum appearing in \eqref{eq:GSI_assumption_introduction}; that is,
\begin{align*}
  t_0(\xi) = \sum_{j \in J} \frac{1}{|\det C_j|} | \widehat{g_j} (\xi) |^2.
\end{align*}
A powerful frame criterion arises by comparing the representation of $S$
given in \eqref{eq_intro_rep} to the diagonal term $\Fourier^{-1} (t_0 \cdot \widehat{f} \, )$,
and by estimating the corresponding discrepancy.
In the model cases of Gabor and wavelets systems, these criteria are again attached to the names
of Walnut and Daubechies, and are particularly useful for studying Gaussian wave-packets,
which have fast-decaying frequency tails, but do not yield tight frames.
A general version of the Walnut-Daubechies criterion also holds
for generalized shift-invariant systems under mild assumptions
\cite{Lemvig2018Criteria, Christensen2008frame}; this criterion is greatly useful
in the construction of anisotropic time-scale decompositions---see e.g. \cite{MR3225503}.

The price to pay for the flexibility of the Walnut-Daubechies criterion is that it does not produce
an explicit dual system implementing the coefficient functionals $f \mapsto c_{j, \gamma}$
in \eqref{eq:series_intro}.
Rather, it only yields an $L^2$ norm estimate which is sufficient
to establish \eqref{eq:frame_normequiv} but does not imply the convergence of \eqref{eq_intro_rep}
in other norms.
In contrast, explicit constructions of \emph{frame pairs}, that is,
frames where the coefficient functionals are given by
\begin{align*}
  c_{j, \gamma} = \big\langle f \,\mid\, \Translation{\gamma} h_j \big\rangle
\end{align*}
for another reproducing system $\{h_j: j \in J\}$, naturally extend
to many other Banach spaces besides $L^2(\RR^d)$.
These spaces are determined by the concentration of the Fourier support of the generators $g_j$,
and are generically called \emph{Besov-type spaces}
\cite[Chapter 2]{TriebelFourierAnalysisAndFunctionSpaces} \cite{tr78}.
The model case is given by \eqref{eq_wavelets}, where the functions $\widehat{g_j}$
form a so-called Littlewood-Paley decomposition.

The goal of this article is to derive a variant of the Walnut-Daubechies criterion
which implies that the frame operator is invertible in such Besov-type spaces.

\subsection{Besov-type decomposition spaces}
For the informal definition of Besov-type spaces, fix a cover $\CalQ = (Q_i)_{i \in I}$
of a full measure open subset in the Fourier domain $\RHat^d$.
We impose a mild admissibility condition by limiting the number of overlaps
between different elements of $\CalQ$---see Section~\ref{sec:decomposition} for the precise condition.
Given a suitable partition of unity $(\varphi_i)_{i \in I}$ subordinate to $\CalQ$,
together with a suitable (so-called $\CalQ$-moderate) weight function $w : I \to (0,\infty)$,
the space $\StandardDecompSp$, for $p, q \in [1,\infty]$, is defined as the space
of distributions $f$ satisfying
\begin{align}
  \| f \|_{\mathcal{D} (\CalQ, L^p, \ell^q_w)}
  := \Big\|
       \big(
         \big\| \Fourier^{-1} (\varphi_i \cdot \widehat{f} \,) \|_{L^p}
       \big)_{i \in I}
     \Big\|_{\ell^q_w}
   = \Big\|
       \big(
         w_i \cdot \big\| \Fourier^{-1} (\varphi_i \cdot \widehat{f} \,) \|_{L^p}
       \big)_{i \in I}
     \Big\|_{\ell^q}
  < \infty,
  \label{eq:decomposition_norm_intro}
\end{align}
where $\Fourier^{-1}$ denotes the inverse Fourier transform.
Provided that an adequate notion of distribution is used in the definitions,
the spaces $\StandardDecompSp$ form Banach spaces and are independent of the
particular (sufficiently regular) partition of unity used to define them.

The construction of Besov-type spaces follows the so-called
\emph{decomposition method} \cite[Chapter 2]{TriebelFourierAnalysisAndFunctionSpaces},
\cite[Section 1.2]{tr78}, yielding an instance of the so-called
\emph{spaces defined by decomposition methods} \cite{sttr79},
or \emph{decomposition spaces} \cite{tr77,DecompositionSpaces1} in more abstract settings.
This is why we also use the term \emph{Besov-type decomposition spaces}.
Uniform Besov-type spaces, associated with
the cover $\CalQ$ consisting of integer translates of a cube,
are known as \emph{modulation spaces} \cite{feichtinger1989atomic},
while a dyadic frequency cover yields the usual
\emph{Besov spaces} \cite{FrazierBesov,PeetreNewThoughtsOnBesovSpaces}
---see also \cite[Section 2.2]{TriebelFourierAnalysisAndFunctionSpaces}.
When the cover is generated by powers of an expansive matrix, one obtains anisotropic Besov spaces
\cite{TriebelFourierAnalysisAndFunctionSpaces, MR1982689,MR2179611,MR2838283}.
We remark that the range of spaces defined by \eqref{eq:decomposition_norm_intro}
does \emph{not} include Triebel-Lizorkin spaces \cite{FrazierTriebel}.

\subsection{Overview of the results}
We state a simplified version of our main results for systems of the form \eqref{eq:GSI_intro}
with generating functions $g_j \in L^1 (\RR^d) \cap L^2 (\RR^d)$
with $\widehat{g} \in C^{\infty} (\RHat^d)$, given by
\begin{equation}
  g_j
  = |\det A_j |^{-1/2} \cdot \Fourier^{-1} (\widehat{g} \circ S_j^{-1})
  = |\det A_j|^{1/2} \cdot e^{2 \pi i \langle b_j, \cdot \rangle} \cdot (g \circ A_j^t) \, ,
  \label{eq:structured_intro}
\end{equation}
for (invertible) affine maps $S_j = A_j (\cdot) + b_j$
and translation matrices $C_j = \delta A_j^{-t}$ with $\delta > 0$.
The parameter $\delta>0$ is a \emph{resolution parameter}
that controls the density of the translation nodes in \eqref{eq:GSI_intro}.

In order to define Besov-type spaces adapted to the frequency concentration
of the system $(g_j)_{j \in J}$, we also consider an \emph{affinely generated cover}
$\CalQ = (Q_j)_{j \in J}$ of the form $Q_j = A_j Q + b_j$.
If $\widehat{g}$ is mostly concentrated inside the basic set $Q$, then
\eqref{eq:structured_intro} implies that $\widehat{g_j}$ is localized around $Q_j$.
Under these assumptions, the Calder\'on condition reads
\begin{align}
  0 < \newA \leq
  \sum_{j \in J} \big|\widehat{g}(S_j^{-1} \xi) \big|^2 \leq \newB < \infty,
  \qquad \mbox{a.e. },
  \label{eq_intro_cal2}
\end{align}
which means that $(\widehat{g_j})_{j \in J}$
is approximately a partition of unity adapted to $\CalQ$.

The following is our main result, proved in Section~\ref{sec_proof_of_thm:theorem_intro}.

\begin{theorem} \label{thm:theorem_intro}
For each affinely generated cover
$\CalQ = (A_j Q + b_j)_{j \in J} = (S_j Q)_{j \in J}$ of an open,
co-null set $\CalO \subset \RHat^d$, and each $\CalQ$-moderate weight $w = (w_j)_{j \in J}$,
there exists a constant $C_{d,\CalQ,w}$ with the following property:
Suppose that $(g_j)_{j \in J}$ is compatible with $\CalQ$ in the sense of \eqref{eq:structured_intro}
and that the Calder\'on condition \eqref{eq_intro_cal2} holds.
Moreover, suppose that
\[
 M_0 := \sup_{i \in J}
    \sum_{j \in J}
      \max \{1, \|A_j^{-1} A_i\|^{d+1} \}
      \bigg(
        \int_Q \max_{|\alpha| \leq d+1}
          |(\partial^{\alpha} \widehat{g}) (S_j^{-1} (S_i \xi))|^{2(d+1)}
        \; d\xi
      \bigg)^{\frac{1}{d+1}}
  < \infty \,
\]
and that
\(
  M_1 := \max \big\{
                \sup_{i \in J} \sum_{j \in J} M_{i,j},\; \sup_{j \in J} \sum_{i \in J} M_{i,j}
              \big\}
       < \infty
\),
where
\[
  M_{i,j}   := L_{i,j} \cdot
               \int_{Q}
                 (1 + |S_j^{-1} (S_i \xi)|)^{2d+2}
                 \max_{|\alpha| \leq d+1}
                   |(\partial^\alpha \widehat{g})(S_j^{-1} (S_i \xi))|
               \, d \xi
  \]
  and
  \(
    L_{i,j} := \max \big\{ \frac{w_i}{w_j}, \frac{w_j}{w_i} \big\} \cdot
               \big(
                 \max \{ 1, \|A_i^{-1} A_j\|^2 \}
                 \, \max \{ 1, \|A_j^{-1} A_i\|^3 \}
               \big)^{d+1}
  \)
  for $i,j \in J$.
  Choose $\delta > 0$ such that
  \[
    C_{d,\CalQ,w} \, M_0^{\frac{d+1}{d+2}} \, M_1^{\frac{2}{d+2}} \, \delta < \newA.
  \]
  Then the frame operator associated to
  $(\Translation{\delta A_j^{-t} k} g_j)_{j \in J, k \in \ZZ^d}$
  is well-defined, bounded, and invertible on $\StandardDecompSp$ for all $p,q \in [1,\infty]$.
  The value of the constant $C_{d,\CalQ,w}$ is given in Theorem~\ref{thm:MainSummary} below.
\end{theorem}

The quantities $M_0$ and $M_1$ in Theorem~\ref{thm:theorem_intro} control the interaction
between the generators $g_j$ and the elements of the cover $\CalQ$.
In contrast to the classical $L^2$ Walnut-Daubechies criterion,
the derivatives of $\widehat{g}$ are now involved.
We also prove a more technical version of Theorem~\ref{thm:theorem_intro} in which the generators
need not exactly be affine images (in the Fourier domain) of a single function,
but only approximately so.
This is important, for example, to describe non-homogeneous time-scale systems,
which contain a low-pass and a high-pass window.
We refer the reader to \cite{StructuredBanachFrames} for a detailed discussion of concrete examples
and calculations that can be used also in our framework.

Although the constant $C_{d,\CalQ, w}$ in Theorem~\ref{thm:theorem_intro} is explicit,
it is too large to be used as a guide for concrete numerical implementations.
We also derive a version of the criterion with more favorable constants,
but which only provides expansions on $L^2$-based Besov-type spaces;
see Section \ref{sub:WeightedL2CaseSimplified}.

A result closely related to Theorem~\ref{thm:theorem_intro} was recently obtained
by the third named author in \cite{StructuredBanachFrames}---see the discussion below.
While our techniques are significantly different from those in \cite{StructuredBanachFrames}---and,
indeed, we regard the simplicity of the present methods a main contribution---we remark that
we make use of several auxiliary results obtained in \cite{StructuredBanachFrames}.

Under the conditions of Theorem~\ref{thm:theorem_intro},
the \emph{coefficient} and \emph{reconstruction operators}
\begin{align}
  \analysis : f \mapsto \Big(
                          \big\langle f \,\mid\, \Translation{\gamma} g_j \big\rangle
                        \Big)_{j \in J, \gamma \in C_j \ZZ^d}
  \quad \text{and} \quad
  \synthesis : c = (c_{j,\gamma})_{j \in J, \gamma \in C_j \Z^d}
               \mapsto \sum_{j \in J}
                         \sum_{\gamma \in C_j \ZZ^d}
                           c_{j, \gamma} \, \Translation{\gamma} g_j
  \label{eq_intro_cd}
\end{align}
define bounded operators between the Besov-type space $\StandardDecompSp$
and suitable sequence spaces (see Section~\ref{sec:FrameOperatorBounded}).
As a consequence, the invertibility of the frame operator on the spaces $\StandardDecompSp$
implies that the $L^2$-convergent \emph{canonical frame expansions}
\begin{equation}
  \begin{aligned}
    f  = \sum_{j \in J}
           \sum_{\gamma \in C_j \ZZ^d}
             \big\langle S^{-1} f \,\mid\, \Translation{\gamma} g_j \big\rangle
             \Translation{\gamma} g_j
    = \sum_{j \in J}
        \sum_{\gamma \in C_j \ZZ^d}
          \big\langle f \,\mid\, \Translation{\gamma} g_j \big\rangle
          S^{-1} \Translation{\gamma} g_j
  \end{aligned}
  \label{eq_intro_ops}
\end{equation}
extend to series convergent in Besov-type norms (or weak-$*$-convergent for $p=\infty$ or $q=\infty$).
In more technical terms, the \emph{canonical Hilbert-space dual frame}
$\{ S^{-1} \Translation{\gamma}  g_j: j \in J, \gamma \in C_j \ZZ^d\}$
provides a \emph{Banach frame} and an \emph{atomic decomposition}
for the Besov-type spaces $\StandardDecompSp$.
This is a novel feature of Theorem~\ref{thm:theorem_intro}:
other results on the existence of series expansions, based on so-called \emph{oscillation estimates},
show that the coefficient and reconstruction maps \eqref{eq_intro_cd} are respectively
left and right invertible on the Besov-type spaces, but do not yield consequences
for the Hilbert space pseudo-inverses $\analysis^\dagger= S^{-1} \synthesis$
and $\synthesis^\dagger=\analysis S^{-1}$
\cite{FeichtingerCoorbit1,GroechenigDescribingFunctions, StructuredBanachFrames}.
In contrast, Theorem~\ref{thm:theorem_intro} concerns $\analysis^\dagger, \synthesis^\dagger$---%
see Corollary~\ref{cor:BanachFramesAtomicDecompositions}---and implies that operations
\emph{on the canonical frame expansions} \eqref{eq_intro_ops} that decrease
the magnitude of the coefficients, such as thresholding, are uniformly bounded in Besov-type norms.
More precisely, if for each $j \in J$ and $\gamma \in C_j \Z^d$,
we are given a function $\Phi_{j,\gamma} : \CC \to \CC$ satisfying
$|\Phi_{j,\gamma} (x)| \leq C \, |x|$, then the maps
\[
  f \mapsto \sum_{j \in J}
              \sum_{\gamma \in C_j \Z^d}
                \Phi_{j,\gamma} \big( \big\langle S^{-1} f \,\mid\, \Translation{\gamma} g_j \big\rangle \big) \,
                \Translation{\gamma} g_j
                \]
and
\[
  f \mapsto \sum_{j \in J}
              \sum_{\gamma \in C_j \Z^d}
                \Phi_{j, \gamma} \big( \big\langle f \,\mid\, \Translation{\gamma} g_j \big\rangle \big) \,
                S^{-1} \Translation{\gamma} g_j
\]
are bounded (possibly non-linear) operators on all of the spaces $\StandardDecompSp$.
In particular, \emph{frame multipliers} with bounded symbols---see e.g.~\cite{MR2273547}---define
bounded operators on Besov-type spaces.

\subsection{Related work}
\mbox{}
\medskip

\noindent {\em The theory of localized frames}.
The uniform frequency cover $\{(-1,1)^d + k: k \in \ZZ^d\}$---which gives rise to Gabor systems
\eqref{eq_gabor}---is special in that every reproducing system \eqref{eq:GSI_intro}
satisfying the frame inequalities \eqref{eq:frame_normequiv},
and mild smoothness and decay conditions, provides also expansions for other Banach spaces
(the precise range of spaces being determined by the particular smoothness and decay of the generators).
Indeed, the theory of localized frames \cite{MR2224392,MR2235170,MR2054304} implies that
the frame operator is invertible on modulation spaces.
Similar results hold for $L^p$ spaces \cite{MR2450070,MR3247038}.
Thus, in these cases, the classical Walnut-Daubechies criterion has consequences
for Banach spaces besides $L^2$---without having to adjust the density $\delta$---and
Theorem~\ref{thm:theorem_intro} does not add anything interesting.

The key tool of the theory of localized frames is the spectral invariance of certain matrix algebras.
Such tools are not applicable to general admissible covers as considered in this article.
Indeed, it is known that the frame operator associated with certain smooth and fast-decaying wavelets
with several vanishing moments fails to be invertible on $L^p$-spaces
\cite[Chapter 4]{MeyerWaveletsAndOperators}.
In connection to this point, we mention that the Mexican hat wavelet satisfies Daubechies criterion,
but the validity of the corresponding $L^p$ expansions was established only recently
with significant ad-hoc work \cite{MR3034769}.

\medskip

\noindent{\em Almost painless generators and homogeneous covers}.
There is a well-developed literature related to the so-called painless expansions
on decomposition spaces.
The first construction of Banach frames for general decomposition spaces was given
by Borup and Nielsen \cite{BorupNielsenDecomposition} using bandlimited generators.
This construction was then complemented with a delicate perturbation argument to produce compactly
supported frames \cite{nielsen2012compactly}---see also \cite{MR1862566, MR2989978}.
The constructions in \cite{nielsen2012compactly}
for Besov-type spaces are restricted to so-called \emph{homogeneous covers}, which are generated
by applying integer powers of a matrix to a given set.
This restriction rules out some important examples such as inhomogeneous dyadic covers
and many popular wavepacket systems.

\medskip

\noindent{\em Invertibility of the frame operator versus existence of left and right inverses}.
The first construction of time-scale decompositions proceeded by discretizing
Calder\'on's reproducing formula through Riemann-like sums \cite{MR1107300}.
A similar approach works for the voice transform associated with any integrable
unitary representation and is the basis of the so-called coorbit theory \cite{FeichtingerCoorbit1}.
To some extent, those techniques extend to any integral transform,
provided that one can control its modulus of continuity \cite{MR1025485}.
Such an approach was used by the third named author to construct compactly supported
Banach frames and atomic decompositions in Besov-type spaces \cite{StructuredBanachFrames}.
The main result of \cite{StructuredBanachFrames} is qualitatively similar
to Theorem~\ref{thm:theorem_intro}, but only concludes the existence of left and right inverses
for the coefficient and synthesis maps, acting on respective Banach spaces.
In contrast, we show that the Hilbert space frame operator is simultaneously invertible
on all the relevant Banach spaces.
The advantage of the present approach is that we are able to show that
the Hilbert spaces series---which are defined by minimizing the $\ell^2$ norm
of the coefficients in \eqref{eq:series_intro}---extend to series convergent in Besov-type spaces, and thus many operations on the canonical frame expansion are
also shown to be bounded in Besov-type spaces.
On the other hand, there are situations in which there exists a left inverse for the coefficient operator (or a right inverse for the reconstruction operator), but the frame operator is not invertible.
For example, a wavelet system generated by a smooth mother wavelet without vanishing moments can generate an atomic decomposition for the Besov spaces $B^s_{p,q}(\R^d)$ of strictly positive smoothness $s > 0$ without yielding a frame \cite[Proposition 8.4]{StructuredBanachFrames}. Such examples are not covered by our results.

\medskip

\noindent{\em Quasi-Banach spaces}.
We do not treat the quasi-Banach range $p,q \in (0,\infty]$,
which is treated in \cite{StructuredBanachFrames}.
We expect the tools developed in \cite{StructuredBanachFrames} for treating the quasi-Banach range
to be also applicable to the present setting, and to yield an extension
of our main results to the quasi-Banach range.

\subsection{Technical overview and organization}
Our approach is as follows:
we consider the Walnut-Daubechies representation \eqref{eq_intro_rep} of the frame operator
and bound the discrepancy between $Sf$ and the diagonal term
$\mathcal{F}^{-1} \big(t_0 \cdot \widehat{f} \, \big)$ in a Besov-type norm.
To this end, we estimate each Fourier multiplier $t_\alpha$ with a Sobolev embedding,
and control the inverse Fourier multiplier ${1}/{t_0}$ by directly bounding the terms
in Fa\`a di Bruno's formula.

The main estimates are derived in decreasing level of generality.
We first consider very general covers $\CalQ = (Q_i)_{i \in I}$ and an abstract notion of molecule,
which models the interaction between the generators $g_j$ of the system $\StandardGSI$
and the elements $Q_i$ of the cover $\CalQ$.
Here, the associated index sets $I$ and $J$ do not need to coincide.
We then provide simplified estimates for affinely generated covers.
The limiting cases $p,q = \infty$ involve delicate approximation arguments
that may be of independent interest.

The paper is organized as follows:
Section~\ref{sec:Notation} introduces notation and preliminaries.
Besov-type spaces are introduced in Section~\ref{sec:decomposition}.
Section~\ref{sec:FrameOperatorBounded} treats the boundedness of the
coefficient, synthesis and frame operators on suitable spaces.
Section~\ref{sec:FrameOperatorInvertible} is concerned with the invertibility
of the frame operator and provides estimates for the abstract Walnut-Daubechies criterion.
These estimates are further simplified in Sections~\ref{sec:SimplifiedEstimates}
and \ref{sec:Structured_Compatible} for affinely generated covers
and suitably adapted generating functions.
Several technical results are deferred to the appendices.

\section{Notation and preliminaries}
\label{sec:Notation}

\subsection{General notation}

We let $\NN := \{1,2,3,\dots\}$, and $\NN_0 := \NN \cup \{0\}$.
For $n \in \NN_0$, we write ${\underline{n} := \{1, ..., n\}}$;
in particular, $\underline{0} = \emptyset$.
For a multi-index $\beta \in \NN_0^d$, its length is
${|\beta| = \sum_{i = 1}^d |\beta_i|}$.

The conjugate exponent $p'$ of $p \in (1,\infty)$ is defined
as $p' := \frac{p}{p-1}$.
We let $1' := \infty$ and $\infty' := 1$.

Given two functions $f, g : X \to [0,\infty)$, we write $f \lesssim g$ provided that
there exists a constant $C > 0$ such that $f(x) \leq C g(x)$ for all $x \in X$.
We write $f \asymp g$ for $f \lesssim g$ and $g \lesssim f$.

The dot product of $x, y \in \RR^d$ is written
$x \cdot y := \sum_{i = 1}^d x_i \, y_i$.
The Euclidean norm of a vector $x \in \RR^d$ is denoted by $|x| := \sqrt{x \cdot x}$.
The open Euclidean ball, with radius $r > 0$ and center $x \in \RR^d$,
is denoted by $B_r (x)$, and the corresponding closed ball is denoted by $\overline{B_r}(x)$.
More generally, the closure of a set $M \subseteq \RR^d$ is denoted by $\overline{M}$.

The cardinality of a set $X$ will be denoted by $|X| \in \NN_0 \cup \{\infty\}$.
The Lebesgue measure of a Borel measurable set $E \subset \RR^d$ will be denoted
by $\lambda (E)$.
Given a subset $M \subset X$, we define its \emph{indicator function}
$\Indicator_M : X \to \{0,1\}$ by requiring $\Indicator_M (x) = 1$ if $x \in M$
and $\Indicator_M (x) = 0$ otherwise.

\smallskip{}

For a matrix $M \in \mathbb{C}^{I \times J}$, its \emph{Schur norm} is defined as
\[
  \| M \|_{\schur}
  := \max \bigg\{
            \sup_{i \in I}
              \sum_{j \in J}
                |M_{i,j}|
            , \;
            \sup_{j \in J}
              \sum_{i \in I}
                |M_{i,j}|
          \bigg\} \in [0,\infty] \, .
\]
A matrix $M \in \mathbb{C}^{I \times J}$ satisfying $\| M \|_{\schur} < \infty$
is said to be of \emph{Schur-type}.
A Schur-type matrix $M \in \CC^{I \times J}$ induces a bounded linear operator
\(
  \mathbf{M} :
  \ell^p (J) \to \ell^p (I), \;
  (c_j)_{j \in J} \mapsto \big(\sum_{j \in J} M_{i,j} c_j \big)_{i \in I}
\),
with $\|\mathbf{M}\|_{\ell^p \to \ell^p} \leq \| M \|_{\schur}$ for all $p \in [1,\infty]$;
this is called \emph{Schur's test}.
For a proof of a (weighted) version of Schur's test, cf.  \cite[Lemma~4]{Grochenig2000nonlinear}.

\subsection{Fourier analysis}
\label{sub:FourierAnalysis}

The translate of $f : \RR^d \to \CC$ by  $y \in \RR^d$ is denoted by
$\Translation{y} f (x) = f(x - y)$.
We denote by $\RHat^d$ the Fourier domain of $\RR^d$.
Modulation of $f : \RR^d \to \CC$ by $\xi \in \RHat^d$ is denoted by
$M_{\xi} f (x) := e^{2 \pi i \xi \cdot x} f(x) $.
The Fourier transform
$\Fourier : L^1 (\RR^d) \to C_0 (\RHat^d), \; f \mapsto \widehat{f}$
is normalized as
\[
  \widehat{f} (\xi)
  = \int_{\RR^d} f(x) \, e^{-2 \pi i x \cdot \xi} \; dx
\]
for $\xi \in \RHat^d$.
Similarly normalized, we define
$\Fourier : L^1 (\RHat^d) \to C_0(\RR^d)$.
The inverse Fourier transform $\Fourier^{-1}f := \widehat{f}(- \cdot) \in C_0 (\RR^d)$
of $f \in L^1 (\RHat^d)$ will occasionally also be denoted by $\widecheck{f}$.
Similar notation will be used for the (unitary) Fourier-Plancherel transform
$\mathcal{F} : L^2 (\RR^d) \to L^2 (\RHat^d)$.

The test space of compactly supported, smooth functions on an open set
$\CalO \subset \RR^d$ will be denoted by $C_c^{\infty} (\CalO)$.
The topology on $C_c^{\infty} (\CalO)$ is taken to be the usual topology
defined through the inductive limit of Fr\'echet spaces;
see \cite[Section~6.2]{RudinFunctionalAnalysis} for the details.
The sesquilinear dual pairing between $\CalD(\CalO) := C_c^\infty (\CalO)$
and its dual $\CalD'(\CalO)$ is given by
\(
  \langle f \mid g \rangle_{\CalD', \CalD}
  := f (\overline{g})
\)
for $f \in \CalD'(\CalO)$ and $g \in C_c^{\infty} (\CalO)$.

The Schwartz space is denoted by
$\Schwartz (\RR^d)$ and its topological dual will be denoted by $\Schwartz' (\RR^d)$.
The canonical extension of the Fourier transform to $\Schwartz'(\RR^d)$
is denoted by $\Fourier : \Schwartz' (\RR^d) \to \Schwartz' (\RHat^d)$, that is,
$\langle \Fourier f ,\, g \rangle_{\Schwartz', \Schwartz}
 = \langle f ,\, \Fourier g \rangle_{\Schwartz',\Schwartz}$ for
$f \in \Schwartz'(\RR^d)$ and $g \in \Schwartz(\RHat^d)$.
We denote \emph{bilinear} dual pairings by $\langle \cdot,\cdot \rangle$,
while $\langle \cdot \mid \cdot \rangle$ denotes a \emph{sesquilinear} dual pairing,
which is anti-linear in the second component.

\smallskip{}

Lastly, for $p \in [1,\infty]$ we define
$\Fourier L^p (\RR^d) := \{ \widehat{f} \colon f \in L^p (\RR^d) \} \subset \Schwartz' (\RHat^d)$,
equipped with the norm ${\|f\|_{\Fourier L^p} := \|\Fourier^{-1} f\|_{L^p}}$.
Here, note that $\|f \cdot g\|_{\Fourier L^p} \leq \|f\|_{\Fourier L^1} \cdot \|g\|_{\Fourier L^p}$,
where the exact nature of the product $f \cdot g$ is explained in more detail
in Definition~\ref{def:SpecialMultiplication}.
Furthermore, for  any invertible affine-linear map $S : \RHat^d \to \RHat^d$, one has
$\|f \circ S\|_{\Fourier L^1} = \|f\|_{\Fourier L^1}$.

\subsection{Amalgam spaces}\label{sub:AmalgamSpaces}
Let $U \subset \R^d$ be a bounded Borel set with non-empty interior.
The \emph{Amalgam space} $W_U (L^{\infty}, L^1)$ is the space of
all $f \in L^{\infty} (\RR^d)$ satisfying
\[
  \| f \|_{W_U (L^{\infty}, L^1)} :=  \int_{\RR^d} \| f \|_{L^{\infty} (U + x)} \; dx < \infty.
\]
The (closed) subspace of $W_U (L^{\infty} \!,\! L^1)$ consisting of \emph{continuous}
functions is denoted $W_U(C_0, L^1)$.

The space $W (L^\infty, L^1) := W_U (L^\infty, L^1)$ is independent of the choice of $U$,
with equivalent norms for different choices.
In particular, if $A \in \GL(\RR^d)$, then
\begin{equation}
  \|f\|_{W_{AU}(L^\infty, L^1)}
  = |\det A| \cdot \| f \circ A \|_{W_U (L^\infty, L^1)} \, ,
  \label{eq:WienerAmalgamLinearTransformation}
\end{equation}
an identity that will be used repeatedly.
It is readily seen that the space $W_U (L^{\infty}, L^1)$ is an $L^1$-convolution module;
that is, if $f \in L^1 (\RR^d)$ and $g \in W_U (L^{\infty}, L^1)$,
then the product $f \ast g \in W_U (L^{\infty}, L^1)$, with
$\| f \ast g \|_{W_U (L^{\infty}, L^1)} \leq \|f\|_{L^1} \|g\|_{W_U(L^{\infty}, L^1)}$,
simply because of
$\|f \ast g\|_{L^\infty(U + x)} \leq \big( |f| \ast [y \mapsto \|g\|_{L^\infty(U + y)}] \big) (x)$.

Lastly, there is an equivalent \emph{discrete} norm on $W(L^{\infty}, L^1)$, namely
\[
  \| f \|_{W(L^{\infty}, \ell^1)}
  := \sum_{n \in \ZZ^d}
       \| \mathds{1}_{n + [0,1]^d} \cdot f \|_{L^{\infty}}.
\]
The global component in this norm is denoted by $\ell^1$ rather than $L^1$
in order to distinguish it from $\| \cdot \|_{W_U (L^{\infty}, L^1)}$.
The norm $\| \cdot \|_{W(C_0, \ell^1)}$ is simply the restriction of
$\| \cdot \|_{W(L^{\infty}, \ell^1)}$ to $W_U(C_0, L^1)$.

The reader is referred to \cite{ho75, fost85} for background on amalgam spaces
and to \cite{FeichtingerWienerSpaces} for a far-reaching generalization
that includes the combination of smoothness and decay conditions.

\section{Besov-type spaces}
\label{sec:decomposition}

This section introduces decomposition spaces, and related notions such as covers,
weights and bounded admissible partitions of unity (BAPUs).

\subsection{Covers and BAPUs}
\label{sub:Coverings}

\begin{definition}\label{def:AdmissibleCovering}
Let $\CalO \neq \emptyset$ be an open subset of $\RHat^d$.
A family $\CalQ = (Q_i)_{i \in I}$ of subsets $Q_i \subset \CalO$ is
called an \emph{admissible cover} of $\CalO$ if
\begin{enumerate}
  \item[(i)] $\CalQ$ is a cover of $\CalO$, that is,
             $\CalO = \bigcup_{i \in I} Q_i$;

  \item[(ii)] $Q_i \neq \emptyset$ for all $i \in I$;

  \item[(iii)] $N_{\CalQ} := \sup_{i \in I} |i^*| < \infty$,
               where $i^* := \{ \ell \in I \; : \; Q_{\ell} \cap Q_i \neq \emptyset \}$
               for $i \in I$.
\end{enumerate}
A sequence $w = (w_i)_{i \in I}$ in $(0,\infty)$ is called a
\emph{$\CalQ$-moderate weight} if
\[
  C_{w, \CalQ}
  := \sup_{i \in I} \, \sup_{\ell \in i^*} \, \frac{w_i}{w_{\ell}}
  < \infty.
\]
\end{definition}
For a \emph{weight} $w = (w_i)_{i \in I}$ in $(0,\infty)$ and an exponent
$q \in [1, \infty]$, we define
\[
  \ell_w^q (I)
  := \left\{
       c = (c_i)_{i \in I} \in \CC^I
       \; : \;
       \| c \|_{\ell^q_w} := \| (w_i \cdot c_i)_{i \in I} \|_{\ell^q} < \infty
     \right\}.
\]
The significance of a $\CalQ$-moderate weight is that the associated
\emph{$\CalQ$-clustering map} is well-defined and bounded.
The precise statement is as follows; see~\cite[Lemma~4.13]{DecompositionEmbedding}.

\begin{lemma}\label{lem:clustering_map}
Let $q \in [1,\infty]$. Suppose that $\CalQ = (Q_i)_{i \in I}$ is an admissible
cover of an open subset $\CalO \subset \RHat^d$
and that the weight $w = (w_i)_{i \in I}$ is
$\CalQ$-moderate. Then the \emph{$\CalQ$-clustering map}
\[
  \Gamma_\CalQ : \ell_w^q (I) \to \ell_w^q (I), \quad
                 (c_i)_{i \in I} \mapsto (c_i^{\ast})_{i \in I},
\]
where
$
  c_i^\ast := \sum_{\ell \in i^\ast} c_\ell \, ,
$
is well-defined and bounded, with
$\|\Gamma_{\CalQ}\|_{\ell^q_w \to \ell^q_w} \leq C_{w,\CalQ} \cdot N_{\CalQ}$.
\end{lemma}

The next definition clarifies our assumptions regarding the partitions of unity
that are suitable for defining the decomposition space norm.

\begin{definition}\label{def:BAPU}
  Let $\CalQ = (Q_i)_{i \in I}$ be an admissible cover of
  an open subset $\emptyset \neq \mathcal{O} \subset \RHat^d$.
  A family ${\Phi = (\varphi_i )_{i \in I}}$ is called a
  \emph{bounded admissible partition of unity} (BAPU), subordinate to $\CalQ$, if
  \begin{enumerate}[label=(\roman*)]
    \item  \label{enu:BAPU_Smoothness}
           $\varphi_i \in C_c^{\infty} (\CalO) \subset \Schwartz (\RHat^d)$ for all $i \in I$;

    \item \label{enu:BAPU_Partition}
          $\sum_{i \in I} \varphi_i (\xi) = 1$ for all
          $\xi \in \CalO$;

    \item \label{enu:BAPU_Support}
          $\varphi_i (\xi) = 0$ for all $\xi \in \CalO \setminus Q_i$
          and all $i \in I$;

    \item \label{enu:BAPU_Boundedness}
          $C_{\Phi}
           := \sup_{i \in I} \| \Fourier^{-1} \varphi_i \|_{L^1} < \infty$.
  \end{enumerate}
  The cover $\CalQ$ is called a
  \emph{decomposition cover}  if there exists a BAPU subordinate to $\CalQ$.
\end{definition}

Given a decomposition cover $\CalQ = (Q_i)_{i \in I}$ of
an open set $\emptyset \neq \CalO \subset \RHat^d$, it will be assumed throughout this article
that a BAPU $\Phi = (\varphi_{i})_{i \in I}$ for $\CalQ = (Q_i)_{i \in I}$ is fixed.

\begin{definition}\label{def:SemiStructuredCovering}
  Let $\mathcal{O} \neq \emptyset$ be an open subset of $\RHat^d$.
  A family $\CalQ = (Q_i)_{i \in I}$ of subsets $Q_i \subset \CalO$ is
  called an \emph{affinely generated cover} of $\CalO$ if, for each
  $i \in I$, there are $A_i \in \mathrm{GL}(d, \RR)$ and
  $b_i \in \RHat^d$ and an open subset $Q'_i \subset \RHat^d$ with
  $Q_i = A_i \, (Q'_i) + b_i$ satisfying the following:
  \begin{enumerate}
    \item[(i)] $\CalQ$ is an admissible cover of $\CalO$;

    \item[(ii)] the sets $(Q'_i)_{i \in I}$ are uniformly bounded, that is,
                \[
                  R_{\CalQ} := \sup_{i \in I} \sup_{\xi \in Q'_i} |\xi| < \infty \, ;
                \]

    \item[(iii)] for indices $i, \ell \in I$ with
                 $Q_i \cap Q_\ell \neq \emptyset$, the transformations
                 $A_i (\cdot) + b_i$ and $A_{\ell} (\cdot) + b_{\ell}$
                 are uniformly compatible, that is,
                 \[
                   C_{\CalQ} :=
                   \sup_{i \in I}
                     \sup_{\ell \in i^*}
                       \| A_i^{-1} A_{\ell} \|
                   < \infty  ;
                 \]
  \end{enumerate}
  and moreover,
  for each $i \in I$, there is an open set $Q''_i \subset \RHat^d$ such that
  \begin{enumerate}
    \item[(iv)] the closure $\overline{Q''_i} \subset Q'_i$ for all $i \in I$;
          \vspace{0.1cm}

    \item[(v)] the family $(A_i ( Q''_i) + b_i)_{i \in I}$ covers $\CalO$; and
          \vspace{0.1cm}

    \item[(vi)] the sets $\{Q'_i \with i \in I \}$ and $\{Q''_i \with i \in I\}$
          are finite.
  \end{enumerate}
\end{definition}

\begin{remark}
An affinely generated cover is also called an \emph{(almost) structured cover}
in the literature, see for instance \cite{DecompositionEmbedding} and \cite{BorupNielsenDecomposition}
for similar notions.
\end{remark}

In the sequel, the map $S_i : \RHat^d \to \RHat^d\vphantom{\sum_j}$
will always denote an affine linear mapping $\xi \mapsto A_i \, \xi + b_i$
for some $A_i \in \mathrm{GL}(d,\RR)$ and $b_i \in \RHat^d$.

\begin{definition}\label{def:RegularBAPUNormalizedVersion}
  Let $\CalQ = \big(S_i (Q_{ i}')\big)_{i \in I}$ be an affinely generated cover of $\CalO$,
  and let $\Phi = (\varphi_i)_{i \in I}$ be a smooth partition of unity subordinate to $\CalQ$.
  For $i \in I$, define the \emph{normalization} of $\varphi_i$ by
  $\varphi^{\normalizedBAPU}_i := \varphi_i \circ S_i$.
  The family $\Phi = (\varphi_i)_{i \in I}$ is called
  a \emph{regular partition of unity}, subordinate to $\CalQ$, if
  \begin{equation}
    C_{\CalQ, \Phi, \alpha}
    := \sup_{i \in I}
         \| \partial^{\alpha} \varphi_i^{\normalizedBAPU} \|_{L^\infty}
    < \infty
    \label{eq:RegularPartitionOfUnity}
  \end{equation}
  for all multi-indices $\alpha \in \NN_0^d$.
\end{definition}

The following result shows that every affinely generated cover is a decomposition cover.

\begin{proposition}\label{prop:RegularBAPUs}
  (\cite[Corollary~2.7 and Theorem~2.8]{DecompositionIntoSobolev})

  Let $\CalQ = \big(S_i (Q_{ i}')\big)_{i \in I}$ be an affinely generated cover of $\CalO$.
  Then the following hold:

  \begin{enumerate}
    \item Every regular partition of unity $\Phi$ subordinate to $\CalQ$
          is also a BAPU subordinate to $\CalQ$.

    \item There exists a regular partition of unity $\Phi = (\varphi_i)_{i \in I}$
          subordinate to $\CalQ$.
  \end{enumerate}
\end{proposition}

\subsection{Besov-type spaces}
\label{sub:DecompositionSpaces}

We introduce Besov-type spaces following the approach
in \cite{TriebelFourierAnalysisAndFunctionSpaces}, which relies on
the \emph{space of Fourier distributions}.
Since we only treat the Besov-type scale of spaces, we allow for rather general covers.
More restrictions would be necessary to include the Triebel-Lizorkin scale,
because the corresponding theory relies on inequalities for maximal functions;
see \cite[Section 3.6]{sttr79}, \cite[Section 2.4.3]{TriebelFourierAnalysisAndFunctionSpaces},
and also \cite{MR3189276}.

\begin{definition}\label{def:Reservoir}
  Let $\mathcal{O} \neq \emptyset$ be open in $\widehat{\RR}^d$. The space
  \(
    Z (\mathcal{O})
    := \Fourier (C_c^{\infty}  (\mathcal{O}))
  \)
  is called the \emph{Fourier test function space} on $\mathcal{O}$.
  The space $Z (\mathcal{O})$ is endowed with the unique topology making the
  Fourier transform $\Fourier : C_c^{\infty} (\mathcal{O}) \to Z(\CalO)$
  into a homeomorphism.

  The topological dual space $(Z (\mathcal{O}))'$ of
  $Z (\mathcal{O})$ is denoted by $Z' (\mathcal{O})$ and is called the
  space of \emph{Fourier distributions}.
  The (bilinear) dual pairing between $Z'(\mathcal{O})$ and $Z(\CalO)$
  will be denoted by
  \(
    \langle \phi, f \rangle_{Z', Z}
    := \langle \phi, f \rangle_{Z'}
    := \langle \phi, f \rangle
    := \phi(f)
  \)
  for $\phi \in Z'(\CalO)$ and $f \in Z(\CalO)$.

  The \emph{Fourier transform} $\phi \in \DistributionSpace(\CalO)$ of a Fourier distribution
  $\phi \in Z'(\mathcal{O})$ is defined by duality; i.e.,
  \[
    \Fourier : Z'(\mathcal{O}) \to \DistributionSpace(\mathcal{O}),
    \quad
               \phi \mapsto \Fourier \phi
                            := \widehat{\phi}
                            := \phi \circ \Fourier,
  \]
  which entails
  \(
    \langle \Fourier \phi, f \rangle_{\DistributionSpace, \CalD}
    = \langle \phi, \Fourier f \rangle_{Z', Z}
  \)
  for $ \phi \in Z'(\CalO$) and $f \in C_c^{\infty} (\CalO)$.
\end{definition}

Using the Fourier distributions as a reservoir, a decomposition space is defined as
follows:

\begin{definition}\label{def:DecompositionSpace}
  Let $p, q \in [1, \infty]$. Let $\CalQ = (Q_i)_{i \in I}$ be a
  decomposition cover of an open set $\emptyset \neq \CalO \subset \widehat{\RR}^d$ with
  associated BAPU $(\varphi_i)_{i \in I}$.
  Let $w = (w_i)_{i \in I}$ be $\CalQ$-moderate.
  For $f \in Z' (\CalO)$, set
  \begin{align} \label{eq:decomp_norm}
    \| f \|_{\DecompSp (\CalQ, L^p, \ell^q_w) }
    := \Big\|
          (
           \| \Fourier^{-1} (\varphi_i \cdot \widehat{f} \, ) \|_{L^p}
          )_{i \in I}
       \Big\|_{\ell^q_w}
    \in [0,\infty] \, ,
  \end{align}
  and define the associated \emph{decomposition space}
  $\DecompSp(\CalQ,L^p,\ell^q_w)$ as
  \[
    \DecompSp (\CalQ, L^p, \ell_w^q)
    := \bigg\{
          f \in Z' (\mathcal{O})
          \with
          \|f \|_{\DecompSp (\CalQ, L^p, \ell^q_w) } < \infty
       \bigg\}.
  \]
\end{definition}

\begin{remark}
  The norm \eqref{eq:decomp_norm} is well-defined: If $f \in Z'(\mathcal{O})$,
  then $\widehat{f} \in \DistributionSpace(\mathcal{O})$, whence
  $\varphi_i \cdot \widehat{f}$ is a (tempered) distribution with compact
  support.
  By the Paley-Wiener theorem \cite[Theorem~7.23]{RudinFunctionalAnalysis},
  it follows therefore that $\Fourier^{-1}(\varphi_i \cdot \widehat{f} \,)$
  is given by a smooth function.
  In addition, $\DecompSp (\CalQ, L^p, \ell^q_w)$ is a Banach space
  and independent of the choice of the BAPU $(\varphi_i)_{i \in I}$,
  with equivalent norms for different choices;
  see \cite[Corollary~3.18 and Theorem~3.21]{DecompositionEmbedding}.
\end{remark}

\begin{remark}
\label{rem_ds}
Our presentation follows \cite{DecompositionEmbedding,StructuredBanachFrames}
and relies on the original approach of \cite{TriebelFourierAnalysisAndFunctionSpaces,tr78},
specially in the use of Fourier distributions, which is essential for the more technical aspects
of our results.
More abstract versions of Besov-type spaces replace the Fourier transform
by an adequate symmetric operator \cite{tr77} or use a more general Banach space of functions
on a locally compact space in lieu of the Fourier image of $L^p$ \cite{DecompositionSpaces1}.
This latter (far reaching) generalization is particularly useful
to model signal processing applications, such as sampling.
\end{remark}

In the sequel, we will often prove our results on the subspace
 \(
   \Schwartz_{\CalO} (\RR^d)
   := \Fourier^{-1} (\TestFunctionSpace (\CalO))
   \subset \Schwartz (\RR^d)
 \)
 of the space $\DecompSp(\CalQ, L^p, \ell_w^q)$, and then extend
 to all of $\mathcal{D} (\CalQ, L^p, \ell^q_w)$ by a suitable density argument.
 These density arguments rely on the following concept.

 \begin{definition}\label{def:DominatedFunction}
   Let $I$ be an index set, and let $w = (w_i)_{i \in I}$ be a weight.
   For a sequence $F = (F_i)_{i \in I}$ of functions $F_i \in L^p (\RR^d)$, we write
   \(
     \|F\|_{\ell_w^q(I; L^p)}
     := \big\| (\|F_i\|_{L^p})_{i \in I} \big\|_{\ell_w^q}
     \in [0,\infty]
   \),
   and set
   \[
     \ell_w^q (I; L^p)
     := \big\{
          F \in [L^p(\RR^d)]^I
          \colon
          \|F\|_{\ell_w^q(I;L^p)} < \infty
        \big\} \, .
   \]
   Let $\CalQ = (Q_i)_{i \in I}$ be a decomposition cover of an open set
   $\CalO \subset \RHat^d$ with BAPU $\Phi = (\varphi_i)_{i \in I}$,
   and let $F = (F_i)_{i \in I}$ be a family of functions $F_i : \RR^d \to [0,\infty)$.
   A Fourier distribution $f \in Z'(\CalO)$ is said to be \emph{$(F,\Phi)$-dominated} if,
   for all $i \in I$,
   \begin{equation}
     |\Fourier^{-1} (\varphi_i \cdot \widehat{f} \,)| \leq F_i.
     \label{eq:DominatedDefinition}
   \end{equation}
 \end{definition}

We next state our density result; its proof is postponed to Appendix~\ref{sub:DensityResultProof}.

\begin{proposition}\label{prop:Density}
  Let $\CalQ = (Q_i)_{i \in I}$ be a decomposition cover of an open set
  $\emptyset \neq \CalO \subset \RHat^d$ with BAPU $\Phi = (\varphi_i)_{i \in I}$
   and let $w = (w_i)_{i \in I}$ be a $\CalQ$-moderate weight.
  Then
  \begin{enumerate}
    \item[(i)] The inclusion
               $\DenseSpace(\RR^d) \subset \DecompSp(\CalQ, L^p, \ell_w^q)$
               holds for all $p,q \in [1,\infty]$.

    \item[(ii)] If $p,q \in [1,\infty)$, then $\DenseSpace (\RR^d)$ is norm
                dense in $\DecompSp(\CalQ,L^p,\ell_w^q)$.

    \item[(iii)] If $p,q \in [1,\infty]$ and $f \in \DecompSp(\CalQ, L^p, \ell_w^q)$,
                 then there exist $F \in \ell_w^q (I; L^p)$ satisfying
                 \[
                   \|F\|_{\ell_w^q(I; L^p)}
                   \leq C_\Phi \,
                        \|\Gamma_{\CalQ}\|_{\ell_w^q \to \ell_w^q}^2
                        \cdot \|f\|_{\DecompSp(\CalQ, L^p, \ell_w^q)},
                 \]
                 and a sequence $(g_n)_{n \in \NN}$ of $(F, \Phi)$-dominated
                 functions $g_n \in \DenseSpace (\RR^d)$ such that
                 $g_n \to f$, with convergence in $Z' (\CalO)$.
  \end{enumerate}
\end{proposition}

\begin{remark}
The inclusion $\DenseSpace (\RR^d) \subset \StandardDecompSp \subset Z'(\CalO)$
in Proposition~\ref{prop:Density}(i) should be understood
in the following sense:
Clearly $\DenseSpace (\RR^d) \subset \Schwartz(\RR^d) \hookrightarrow \Schwartz'(\RR^d)$,
where as usual a function $f \in \Schwartz(\RR^d)$ is identified
with the distribution $\phi \mapsto \int f \cdot \phi \, dx$.
But since $Z(\CalO) \hookrightarrow \Schwartz(\RR^d)$,
each $f \in \Schwartz' (\RR^d)$
restricts to an element of $Z'(\CalO)$; in particular, each
$f \in \DenseSpace$ can be seen as an element of $Z'(\CalO)$ by virtue of
$\langle f, \phi \rangle_{Z',Z} = \int f \cdot \phi \, dx$.
Under this identification, the Fourier transform
$\Fourier f \in \CalD'(\CalO)$ is just the usual
$\widehat{f} \in \Schwartz(\RHat^d)$, interpreted as a distribution on $\CalO$.
\end{remark}

As a companion to the above density result, the following \emph{Fatou property}
of the decomposition spaces $\DecompSp(\CalQ, L^p, \ell_w^q)$ will be used.
For the proof, see \cite[Lemma~36]{FuehrVoigtlaenderCoorbitSpacesAsDecompositionSpaces}.

\begin{lemma}
\label{lem:DecompositionSpaceFatouProperty}
  Let $\CalQ = (Q_i)_{i \in I}$ be a decomposition cover of an open set
  $\emptyset \neq \CalO \subset \RHat^d$.
  Let ${w = (w_i)_{i \in I}}$ be a $\CalQ$-moderate weight, and let ${p,q \in [1,\infty]}$.
  Suppose that $(f_n)_{n \in \NN}$ is a sequence in $\DecompSp(\CalQ,L^p,\ell_w^q)$ such that
  ${\liminf_{n\to\infty} \|f_n\|_{\DecompSp(\CalQ,L^p,\ell_w^q)} < \infty}$ and
  $f_n \to f \in Z'(\CalO)$, with convergence in $Z'(\CalO)$.
  Then $f \in \DecompSp(\CalQ,L^p,\ell_w^q)$, with associated norm estimate
  $\|f\|_{\DecompSp(\CalQ,L^p,\ell_w^q)}
   \leq \liminf_{n \to \infty} \|f_n\|_{\DecompSp(\CalQ,L^p,\ell_w^q)}$.
\end{lemma}

\subsection{The extended pairing}
\label{sec:pairing}

We will use the following extension of the $L^2$-inner product.

\begin{definition}\label{def:SpecialDualPairing}
  Let $\CalQ = (Q_i)_{i \in I}$ be a decomposition cover of an open set
  $\emptyset \neq \CalO \subset \widehat{\RR}^d$.
  Let $\Phi = (\varphi_i)_{i \in I}\vphantom{\sum_j}$ be a
  BAPU subordinate to $\CalQ$.
  For $f \in Z'(\CalO)$ and $g \in L^1 (\RR^d)$ with
  $\widehat{g} \in C^\infty (\widehat{\RR}^d)$, define the
  \emph{extended inner product} between $f$ and $g$ as
  \begin{equation}
     \langle f \mid g \rangle_{\Phi}
    := \sum_{i \in I}
         \langle
           \widehat{f}
           \mid
           \varphi_i \cdot \widehat{g} \,
         \rangle_{\CalD', \CalD} \,\, ,
    \label{eq:SpecialDualPairing}
  \end{equation}
  provided that the series on the right-hand side converges absolutely.
\end{definition}

\begin{remark} \label{rem:SpecialDualPairing}

\makeatletter
\hyper@anchor{\@currentHref}
\makeatother
\begin{enumerate}[leftmargin=0.7cm]
  \item[(i)] For $f \in L^2 (\RR^d)$ satisfying $\widehat{f} \equiv 0$ almost everywhere on
             $\RHat^d \setminus \CalO$ and for ${g \in L^1 (\RR^d) \cap L^2 (\RR^d)}$
             with $\widehat{g} \in C^\infty (\widehat{\RR}^d)$,
             the extended inner product defined above coincides with the
             standard inner product on $L^2$.
             Indeed, since $|\varphi_i (\xi)| \leq \|\varphi_i\|_{\Fourier L^1} \leq C_\Phi$
             and thus $\sum_{i \in I} |\varphi_i (\xi)| \leq N_\CalQ \, C_\Phi$,
             we can apply the dominated convergence theorem to see that
             \begin{align*}
               \quad \quad
               \langle f \mid g \rangle_{\Phi}
               &= \sum_{i \in I}
                   \langle
                     \widehat{f} \mid \varphi_i \cdot \widehat{g} \,
                   \rangle_{\CalD', \CalD}
               = \sum_{i \in I}
                    \int_{\widehat{\RR}^d}
                      \widehat{f}(\xi) \,\,
                      \overline{\varphi_i (\xi)} \,\,
                      \overline{\widehat{g} (\xi)}
                    \, d \xi \\
               &= \int_{\widehat{\RR}^d}
                   \widehat{f} (\xi) \,\,
                   \overline{\widehat{g} (\xi)} \,\,
                   \overline{\sum_{i \in I} \varphi_i (\xi)}
                 \, d \xi
               = \int_{\CalO}
                    \widehat{f} (\xi) \, \overline{\widehat{g} (\xi)}
                  \, d \xi \\
                &= \langle \widehat{f} \mid \widehat{g} \, \rangle_{L^2}
                = \langle f \mid g \rangle_{L^2} \, .
             \end{align*}

  \item[(ii)] In general, it is not clear whether the extended inner product defined above is
              independent of the chosen BAPU.
              However, as we will show in Lemma~\ref{lem:AdaptednessIndependent},
              the extended pairing is independent of this choice under suitable hypotheses.
\end{enumerate}
\end{remark}

\section{Boundedness of the frame operator} \label{sec:FrameOperatorBounded}

In this section, we present conditions under which the frame operator associated
with a generalized shift-invariant system is well-defined and bounded
on Besov-type decomposition spaces.
These conditions involve the interplay between smoothness and decay of the generators
and the underlying frequency cover.
See also \cite[Section 2]{MR2765595} and \cite{StructuredBanachFrames} for related estimates.

\subsection{Generalized shift-invariant systems}
\begin{definition}\label{def:GSI_System}
  Let $J$ be a countable index set.
  For $j \in J$, let $C_j \in \mathrm{GL}(d, \RR)$
  and $g_j \in L^2(\RR^d)$.
  A \emph{generalized shift-invariant (GSI) system},
  associated with $(g_j)_{j \in J}$ and $(C_j)_{j \in J}$,
  is defined as
  \[
    \big( \Translation{\gamma} g_j \big)_{j \in J, \gamma \in C_j \ZZ^d}
    = \big( g_j(\mybullet - \gamma) \big)_{j \in J, \gamma \in C_j \ZZ^d}.
  \]
\end{definition}

Throughout the paper, we assume the following standing hypotheses on the system.

\medskip{}

\noindent
\textbf{Standing hypotheses}.
The generators $(g_j)_{j \in J}$ of $\StandardGSI$ will be assumed to satisfy
${g_j \in L^1 (\RR^d) \cap L^2 (\RR^d)}$ and $\widehat{g_j} \in C^{\infty} (\RHat^d)$.
Moreover, we will use the function $t_0 := \sum_{j \in J} |\det C_j|^{-1} |\widehat{g_j} |^2$
for which we assume that there exist constants $A, B >0$ such that
\begin{align}
  A \leq \sum_{j \in J} \frac{1}{|\det C_j|} | \widehat{g_j} (\xi) |^2 \leq B
  \qquad \text{for a.e. } \xi \in \RHat^d.
 \label{eq:GSI_assumption}
\end{align}

\begin{remark}
The assumption \eqref{eq:GSI_assumption} is automatically satisfied
for any generalized shift-invariant \emph{frame} $\StandardGSI$ for $L^2 (\RR^d)$,
with frame bounds $A, B > 0$, if it satisfies the so-called
\emph{$\alpha$-local integrability condition} \eqref{eq:LIC} introduced below.
For a proof, see \cite[Theorem~3.13 and Remark~5]{Fuehr2019System}
and \cite[Proposition~4.1]{hernandez2002unified}.
\end{remark}
\medskip{}

Given the GSI system $\StandardGSI$, the associated
\emph{frame operator} is formally defined as
\[
  S : \mathcal{D}(\CalQ, L^p, \ell^q_w) \to \mathcal{D} (\CalQ, L^p, \ell^q_w),
  \quad f \mapsto \sum_{j \in J}
                    \sum_{k \in \mathbb{Z}^d}
                      \langle f \mid \Translation{C_j k} g_j \rangle_{\Phi} \,
                      \Translation{C_j k } g_j \, .
\]
For analyzing the boundedness and well-definedness of the frame operator,
the following terminology will be convenient.

\begin{definition}\label{def:Adaptedness}
Let $\CalQ = (Q_i)_{i \in I}$ be a decomposition cover of an open set $\CalO \subset \RHat^d$
with BAPU $(\varphi_i)_{i \in I}$. Let $w = (w_i)_{i \in I}$
and $v = (v_j)_{j \in J}$ be weights.
The system $\StandardGSI$ is said to be $(w, v, \Phi)$-\emph{adapted}
if the matrix $M \in \CC^{I \times J}$ defined by
\begin{equation}
  M_{i,j}
  := \max \left\{ \frac{w_i}{v_j}, \frac{v_j}{w_i}\right\}
     \cdot | \det C_j|^{\frac{1}{2}}
     \cdot \|
             (\widecheck{\varphi_i} \ast g_j) \circ C_j
           \|_{W(L^{\infty}, \ell^1)} \,
  \label{eq:AnalysisOperatorSchurMatrix}
\end{equation}
is of Schur-type.
\end{definition}

\begin{lemma} \label{lem:AdaptednessIndependent}
Let $\CalQ = (Q_i)_{i \in I}$ be a decomposition cover with BAPU $\Phi$.
Let $w = (w_i)_{i \in I}$ be a $\CalQ$-moderate weight and let the weight $v = (v_j)_{j \in J}$
be arbitrary.
\begin{enumerate}
  \item[(i)] If $\StandardGSI$ is $(w, v, \Phi)$-adapted, then $\StandardGSI$ is $(w, v, \Psi)$-adapted
             for any BAPU $\Psi$ subordinate to $\CalQ$.

  \item[(ii)] If $\StandardGSI$ is $(w,v,\Phi)$-adapted,
              then the extended inner product $\langle f \mid \Translation{C_j k} g_j \rangle_{\Phi}$
              is well-defined and independent of the choice of the BAPU $\Phi$,
              for any $p,q \in [1,\infty]$, any ${f \in \DecompSp(\CalQ, L^p,\ell_w^q)}$,
              and all $j \in J$ and $k \in \ZZ^d$.
\end{enumerate}
\end{lemma}
\begin{proof}
We assume throughout that
$\Phi = (\varphi_i)_{i \in I}$ and $\Psi = (\psi_i)_{i \in I}$ are two BAPUs subordinate to $\CalQ$.

We first show that if $\StandardGSI$ is $(w,v,\Phi)$-adapted,
then $\StandardGSI$ is also $(w,v,\Psi)$-adapted.
For this, note that
$(f \ast g) (C x) = |\det C| \cdot \big( (f \circ C) \ast (g \circ C) \big) (x)$
for any $f \in L^1(\RR^d)$, $g \in L^1 (\RR^d) \cap L^\infty (\RR^d)$, and $C \in \GL(d,\RR)$.
Using this, together with $\psi_i = \varphi_i^{\ast} \, \psi_i$, yields
\begin{align*}
    \big\|
      (\widecheck{\psi_i} \ast g_j) \circ C_j
    \big\|_{W(L^\infty, L^1)}
     &
     \leq \sum_{\ell \in i^\ast}
             \Big\|
               \big[
                 (\Fourier^{-1} \psi_i)
                 \ast (\widecheck{\varphi_\ell} \ast g_j)
               \big] \circ C_j
             \Big\|_{W(L^\infty, L^1)} \\
    & = \sum_{\ell \in i^\ast}
          |\det C_j|
          \cdot \Big\|
                  \big[ \widecheck{\psi_i} \circ C_j \big]
                  \ast \big[ (\widecheck{\varphi_\ell} \ast g_j) \circ C_j \big]
                \Big\|_{W(L^\infty, L^1)} \\
    & \leq \sum_{\ell \in i^\ast}
             |\det C_j|
             \cdot \| \widecheck{\psi_i} \circ C_j \|_{L^1}
             \cdot \|
                     (\widecheck{\varphi_\ell} \ast g_j) \circ C_j
                   \|_{W(L^\infty, L^1)} \\
    & \leq C \cdot C_\Psi
           \cdot \sum_{\ell \in i^\ast}
                   \big\|
       (\widecheck{\varphi_\ell} \ast g_j) \circ C_j
     \big\|_{W(L^\infty, \ell^1)} \, ,
  \numberthis \label{eq:AdaptednessIndependenceMainEstimate}
\end{align*}
where $C \geq 1$ is given by the norm equivalence
$\|\cdot\|_{W(L^\infty, \ell^1)} \asymp \|\cdot\|_{W(L^\infty, L^1)}$.

The matrix entries $M_{i,j}$ in~\eqref{eq:AnalysisOperatorSchurMatrix} satisfy
\[
  M_{i,j}
  = \max \{ \frac{w_i}{v_j}, \frac{v_j}{w_i} \}
    \cdot |\det C_j|^{1/2}
    \cdot \big\|
            (\widecheck{\varphi_i} \ast g_j) \circ C_j
          \big\|_{W(L^\infty, \ell^1)}.
\]
Likewise, let us define
\[
  N_{i,j}
  := \max \{ \frac{w_i}{v_j}, \frac{v_j}{w_i} \}
     \cdot |\det C_j|^{1/2}
     \cdot \big\|
             (\widecheck{\psi_i} \ast g_j) \circ C_j
           \big\|_{W(L^\infty, \ell^1)}.
\]
Using the moderateness of the weight $w$
and the equivalence $\ell \in i^\ast \Longleftrightarrow i \in \ell^\ast$,
we obtain that
\begin{align*}
  \sum_{i \in I}
    N_{i,j}
  & \leq C^2 C_\Psi
          \sum_{i \in I}
                 \sum_{\ell \in i^\ast}
                   \max \Big\{ \frac{w_i}{v_j} , \frac{v_j}{w_i} \Big\}
                    |\det C_j|^{1/2}
                   \big\|
       (\widecheck{\varphi_\ell} \ast g_j) \circ C_j
     \big\|_{W(L^\infty, \ell^1)} \\
  & \leq C^2 C_\Psi C_{w,\CalQ}
         \cdot \sum_{\ell \in I}
                 \sum_{i \in \ell^\ast}
                   \max \Big\{ \frac{w_\ell}{v_j} , \frac{v_j}{w_\ell} \Big\}
                    |\det C_j|^{1/2}
                   \big\|
       (\widecheck{\varphi_\ell} \ast g_j) \circ C_j
     \big\|_{W(L^\infty, \ell^1)} \\
  & \leq C^2 C_\Psi C_{w, \CalQ} N_{\CalQ}
          \sum_{\ell \in I}
                 M_{\ell,j}
    \leq C^2 C_\Psi C_{w, \CalQ} N_{\CalQ}  \|M\|_{\mathrm{Schur}}
    < \infty
\end{align*}
for all $j \in J$.
Similarly,
\begin{align*}
  \sum_{j \in J}
    N_{i,j}
  & \leq C^2 C_\Psi
         \cdot \sum_{j \in J}
                 \sum_{\ell \in i^\ast}
                    \max \Big\{ \frac{w_i}{v_j} , \frac{v_j}{w_i} \Big\}
                     |\det C_j|^{1/2}
                    \big\|
       (\widecheck{\varphi_\ell} \ast g_j) \circ C_j
     \big\|_{W(L^\infty, \ell^1)} \\
  & \leq C^2 C_\Psi C_{w,\CalQ}
         \cdot \sum_{\ell \in i^\ast}
                 \sum_{j \in J}
                    M_{\ell, j}
    \leq C^2 C_\Psi C_{w,\CalQ} N_{\CalQ} \|M\|_{\mathrm{Schur}}
    < \infty
\end{align*}
for all $i \in I$.
In combination, these two estimates show that  $N = (N_{i,j})_{i \in I, j \in J}$
is of Schur-type.

Finally, let $p,q \in [1,\infty]$ and $f \in \DecompSp(\CalQ, L^p, \ell_w^q)$,
as well as $j \in J$ and $k \in \ZZ^d$ be arbitrary; we show that the extended product
$\langle f \mid \Translation{C_j k} g_j \rangle_{\Phi}$ is well-defined and that
\(
  \langle f \mid \Translation{C_j k} g_j \rangle_{\Phi}
  = \langle f \mid \Translation{C_j k} g_j \rangle_{\Psi}
\).
To show this, set
\(
  B_{j,i}
  := |\det C_j|^{1/2}
     \cdot \|
             (\widecheck{\varphi_i} \ast g_j) \circ C_j
           \|_{W(C_0, \ell^1)}
\).
Since $\StandardGSI$ is $(w,v,\Phi)$-adapted, Schur's test shows that
\(
  \mathbf{B} :
  \ell_w^q (I) \to \ell_v^q (J),
  (c_i)_{i \in I} \mapsto \Big( \sum_{i \in I} B_{j,i} \, c_i \Big)_{j \in J}
\)
is well-defined and bounded.
Define $d_i := \| \Fourier^{-1} (\varphi_i \cdot \widehat{f} \,)\|_{L^p}$
and $c_i := \| \Fourier^{-1} (\varphi_i^\ast \cdot \widehat{f} \,)\|_{L^p} \vphantom{\sum_j}$,
and note that
$0 \leq c_i \leq \sum_{\ell \in i^\ast} d_\ell = (\Gamma_{\CalQ} \, d)_i$,
whence $c = (c_i)_{i \in I} \in \ell_w^q(I)$, since $d = (d_i)_{i \in I} \in \ell_w^q(I)$
as $f \in \DecompSp(\CalQ, L^p, \ell_w^q)$.

As the final setup, let $p' \in [1,\infty]$ denote the conjugate exponent to $p$,
and set $g := \Translation{C_j k} g_j$.
Since $\|f\|_{L^{p'}} \leq \|f\|_{W(C_0,\ell^1)}$ for all $f \in W(C_0, \ell^1)$
and since ${\widecheck{\varphi_i} \ast g = \Translation{C_j k} (\widecheck{\varphi_i} \ast g_j)}$,
it follows that
\begin{align*}
  \|
    \Fourier^{-1} (\varphi_i \, \psi_\ell \, \widehat{g} \,)
  \|_{L^{p'}}
  & \leq C_\Psi \cdot \|\widecheck{\varphi_i} \ast g\|_{L^{p'}}
    =    C_\Psi
         \cdot |\det C_j|^{1/p'}
         \cdot \| ( \widecheck{\varphi_i} \ast g_j ) \circ C_j \|_{L^{p'}} \\
  & \leq C_\Psi \cdot |\det C_j|^{1/p'}
         \cdot \| ( \widecheck{\varphi_i} \ast g_j ) \circ C_j \|_{W(C_0, \ell^1)}
    =    C_\Psi \cdot |\det C_j|^{\frac{1}{2} - \frac{1}{p}} \cdot B_{j,i} \, .
\end{align*}
Using that $\varphi_i = \varphi_i^\ast \varphi_i$,
and $\widehat{g} \in C^\infty (\RHat^d)$, we next see
\begin{align*}
  \big|
    \langle
      \widehat{f}
      \,\mid\,
      \varphi_i \, \psi_\ell \, \widehat{g} \,
    \rangle_{\CalD', \CalD}
  \big|
  & = \big|
        \langle
          \varphi_i^\ast \, \widehat{f}
          \,\mid\,
          \varphi_i \, \psi_\ell \, \widehat{g} \,
        \rangle_{\Schwartz', \Schwartz}
      \big|
  = \big|
      \langle
        \Fourier^{-1} (\varphi_i^\ast \, \widehat{f} \, )
        \,\mid\,
        \Fourier^{-1} (\varphi_i \, \psi_\ell \, \widehat{g})
      \rangle_{L^p, L^{p'}}
    \big| \\
  & \leq \| \Fourier^{-1} (\varphi_i^\ast \, \widehat{f} \,) \|_{L^p}
         \cdot \| \Fourier^{-1} (\varphi_i \, \psi_\ell \, \widehat{g} ) \|_{L^{p'}}
    \leq C_\Psi \cdot c_i \cdot |\det C_j|^{\frac{1}{2} - \frac{1}{p}} \cdot B_{j, i} \, ,
\end{align*}
where the right-hand side is independent of $\ell$.
Given this estimate, it follows immediately that
\[
  \sum_{i \in I}
    \sum_{\ell \in i^\ast}
      |
       \langle
         \widehat{f}
         \,\mid\,
         \varphi_i \, \psi_\ell \, \widehat{g} \,
       \rangle_{\CalD',\CalD}
      |
  \leq C_\Psi N_{\CalQ} \cdot |\det C_j|^{\frac{1}{2} - \frac{1}{p}} \cdot (\mathbf{B} \, c)_j
  < \infty \, .
\]
Therefore, we can interchange the sums in the following calculation:
\begin{align*}
  \langle f \,\mid\, g \rangle_{\Phi}
  & = \sum_{i \in I}
        \langle \, \widehat{f} \,\mid\, \varphi_i \, \widehat{g} \, \rangle_{\CalD',\CalD}
    = \sum_{i \in I} \,
        \sum_{\ell \in i^\ast}
          \langle\,
            \widehat{f}
            \,\mid\,
            \varphi_i \, \psi_\ell \, \widehat{g}\,
          \rangle_{\CalD', \CalD} \\
  & = \sum_{\ell \in I} \,
        \sum_{i \in \ell^\ast}
          \langle\,
            \widehat{f}
            \,\mid\,
            \varphi_i \, \psi_\ell \, \widehat{g}\,
          \rangle_{\CalD', \CalD}
    = \sum_{\ell \in I}
        \langle\,
          \widehat{f}
          \,\mid\,
          \psi_\ell \, \widehat{g}\,
        \rangle_{\CalD', \CalD}
    = \langle f \,\mid\, g \rangle_{\Psi} .
\end{align*}
This calculation implies in particular that both $\langle f \,\mid\, g \rangle_{\Phi}$
and $\langle f \,\mid\, g \rangle_{\Psi}$ are well-defined.
\end{proof}

\subsection{Sequence spaces and operators}
The frame operator can be factored into the \emph{coefficient} and the \emph{reconstruction} operator.
In this subsection, we investigate the boundedness of these operators on suitable sequence spaces.

\begin{definition}\label{def:CoefficientSpace}
  Let $\StandardGSI$ be a generalized shift-invariant system and let
  $p, q \in [1,\infty]$.
  For a  weight $v = (v_j)_{j \in J}$ and a sequence
  $c = (c_k^{(j)})_{j \in J, k \in \ZZ^d} \in \CC^{J \times \ZZ^d}$, define
  \[
    \| c \|_{Y_v^{p,q}}
    := \left\|
         \left( v_j \cdot
           | \det C_j |^{\frac{1}{p}-\frac{1}{2}}
           \cdot \| (c_k^{(j)})_{k \in \ZZ^d} \|_{\ell^p}
         \right)_{j \in J}
       \right\|_{\ell^q} \in [0,\infty].
  \]
  Finally, define the associated \emph{coefficient space} $Y^{p,q}_v$ as
  \[
    Y_v^{p,q}
    := \left\{
          c \in \CC^{J \times \ZZ^d}
          \with
          \| c \|_{Y_v^{p,q}} < \infty
       \right\} .
  \]
\end{definition}

 Let $\StandardDecompSp$ be a decomposition space.
 Given a GSI system $\StandardGSI$ and an associated coefficient space $Y_v^{p,q}$,
 the \emph{reconstruction} or \emph{synthesis} operator is formally defined as the mapping
 \begin{equation}
  \synthesis : Y_{v}^{p,q} \to \DecompSp(\CalQ,L^p,\ell_w^q),
               \quad
               (c_{k}^{(j)})_{j \in J,k \in \mathbb{Z}^d}
               \mapsto  \sum_{j \in J} \,
                          \sum_{k \in \mathbb{Z}^d} \,
                            c_{k}^{(j)} \,
                            \Translation{C_j k} \, g_j \, ,
  \label{eq:SynthesisOperator}
 \end{equation}
 while the \emph{coefficient} or \emph{analysis} operator is formally defined by
 \begin{align*}
   \analysis : \DecompSp(\CalQ, L^p, \ell_w^q) \to Y_v^{p,q},
               \quad
               f \mapsto \Big(
                           \langle
                             f \mid \Translation{C_j k} g_j
                           \rangle_{\Phi}
                         \Big)_{j \in J, k \in \ZZ^d} \, \, ,
 \end{align*}
 where $\langle \cdot, \cdot \rangle_{\Phi}$ denotes the extended pairing
 defined in Section~\ref{sec:pairing}.

\subsection{Boundedness of analysis and synthesis operators}
\label{sec:CoefficientReconstructionOperatorBounded}

For proving the boundedness of the operators $\synthesis$ and $\analysis$,
we will invoke the following lemma.

\begin{lemma}\label{lem:GeneralLatticeSynthesis}
  Let $g \in W(C_0, \ell^1)(\RR^d)$ and $M \in \GL(\RR^d)$.
  Then the map
  \[
    D_{M,g}  :
    c = (c_k)_{k \in \ZZ^d}
    \mapsto \sum_{k \in \ZZ^d}
              c_k \, \Translation{M k} g
  \]
  is bounded from $\ell^{\infty} (\ZZ^d)$ into $L^{\infty} (\RR^d)$,
  with the series converging pointwise absolutely.
  Furthermore, for any $p \in [1,\infty]$, the mapping
  $D_{M,g} : \ell^p (\ZZ^d) \to L^p (\RR^d)$
  is well-defined and bounded, with
  \(
    \| D_{M,g} \|_{\ell^p \to L^p}
    \leq | \det M \, |^{1/p} \cdot \| g\circ M  \|_{W(L^{\infty}, \ell^1)}.
  \)
\end{lemma}

\begin{proof}
  For the case $M = \identity_{\R^d}$, this follows from \cite[Lemma~2.9]{GroechenigNonuniformSampling}%
  ---see also \cite{Bui2008Affine}.
  For the general case, simply note that
  $D_{M,g} \, c (x) = \big( D_{\identity_{\R^d}, g \circ M} (c) \big) (M^{-1} x)$.
\end{proof}

The following technical lemma allows us to use density arguments
for the full range $p,q \in [1,\infty]$.

\begin{lemma}\label{lem:DominatedFunctionsHaveDominatedCoefficients}
  Let $p,q \in [1,\infty]$. Suppose the system $\StandardGSI$ is $(w,v,\Phi)$-adapted with
  matrix $M$ as in \eqref{eq:AnalysisOperatorSchurMatrix}.
  Then, for any $F \in \ell_w^q (I; L^p)$, there is a sequence
  $\theta = (\theta_{j,k})_{j \in J, k \in \ZZ^d} \in Y_v^{p,q}$ such that
  \[
    \|\theta\|_{Y_v^{p,q}}
    \leq \|M\|_{\schur}
         \cdot \|\Gamma_{\CalQ}\|_{\ell_w^q \to \ell_w^q}
         \cdot \|F\|_{\ell_w^q (I; L^p)}
  \]
  and
  \(
    |
     \langle
       f
       \,\mid\,
       \Translation{C_j k} g_j
     \rangle_{\Phi}
    |
    \leq \theta_{j,k}
  \)
   for all $j \in J, k \in \ZZ^d$
   and every $(F,\Phi)\text{-dominated } f \in Z'(\CalO)$.

  Moreover, if $(f_n)_{n \in \NN}$ is a sequence of $(F,\Phi)$-dominated
  Fourier distributions $f_n \in Z'(\CalO)$ satisfying $f_n \to f_0 \in Z'(\CalO)$
  with convergence in $Z'(\CalO)$, then
  \(
    \langle f_n \,\mid\, \Translation{C_j k} g_j \rangle_{\Phi}
    \to \langle f_0 \,\mid\, \Translation{C_j k} g_j \rangle_{\Phi}
  \)
  for all $j \in J, k \in \ZZ^d$.
\end{lemma}

\begin{proof}
  Let $f \in Z'(\CalO)$ be $(F,\Phi)$-dominated.
  Using $\overline{\varphi_i^\ast} \varphi_i = \varphi_i$
  and the estimate \eqref{eq:DominatedDefinition}, we see that
  \begin{align*}
    \big|
      \big\langle
        \widehat{f}
        \,\,\big|\,\,
        \varphi_i \cdot \Fourier [\Translation{C_j k} g_j]
      \big\rangle_{\CalD',\CalD}
    \big|
    & = \big|
          \big\langle
            \varphi_i^\ast \, \widehat{f}
            \,\,\big|\,\,
            \Fourier \big[ \Translation{C_j k} (\widecheck{\varphi_i} \ast g_j) \big]
          \big\rangle_{\Schwartz', \Schwartz}
        \big| \\
  &   \leq \sum_{\ell \in i^\ast}
             \Big|
               \big\langle
                 \Fourier^{-1} (\varphi_\ell \, \widehat{f} \,)
                 \,\,\big|\,\,
                 \Translation{C_j k} (\widecheck{\varphi_i} \ast g_j)
               \big\rangle_{\Schwartz', \Schwartz}
             \Big| \\
    & \leq \sum_{\ell \in i^\ast}
             \int_{\RR^d}
               F_\ell (x) \cdot \big( \Translation{C_j k} |\widecheck{\varphi_i} \ast g_j| \big) (x)
             \, d x
      =:   \sum_{\ell \in i^\ast}
             \zeta_{i,j,k,\ell} \, ,
  \numberthis  \label{eq:DominatedSequenceMainIdentity}
  \end{align*}
  and thus
  \begin{equation}
    \big| \langle f \mid \Translation{C_j k} g_j \rangle_{\Phi} \big|
    = \Big|
        \sum_{i \in I}
          \big\langle
            \widehat{f}
            \,\,\big|\,\,
            \varphi_i \cdot \Fourier[\Translation{C_j k} g_j]
          \big\rangle_{\CalD', \CalD}
      \Big|
    \leq \sum_{i \in I}
           \sum_{\ell \in i^\ast}
             \zeta_{i,j,k,\ell}
    =:   \theta_{j,k} \,
    \label{eq:DominatingSequenceThetaDefinition}
  \end{equation}
  with $\zeta_{i,j,k,\ell}$ and $\theta_{j,k}$ being independent of $f$.

  Next, define a measure $\mu_{i,j,k}$ on $\RR^d$ by
  $d \mu_{i,j,k} (x) := \big( \Translation{C_j k} |\widecheck{\varphi_i} \ast g_j| \big) (x) \, d x$.
  Then
  \begin{align*}
      \zeta_{i,j,k,\ell}
        = \int_{\RR^d}
            F_\ell (x) \cdot 1
          \, d \mu_{i,j,k}(x)
      & \leq \| F_\ell \|_{L^p (\mu_{i,j,k})} \cdot \| 1 \|_{L^{p'} (\mu_{i,j,k})} \\
      & =    \| F_\ell \|_{L^p (\mu_{i,j,k})}
             \cdot \| \Translation{C_j k} (\widecheck{\varphi_i} \ast g_j) \|_{L^1}^{1/p'} \\
     & \leq |\det C_j|^{1/p'}
             \cdot \| F_\ell \|_{L^p (\mu_{i,j,k})}
             \cdot \| ( \widecheck{\varphi_i} \ast g_j) \circ C_j \|_{W(L^\infty, \ell^1)}^{1/p'} \, .
    \numberthis \label{eq:DominatedSequenceThetaMainEstimate}
  \end{align*}
  There are now two cases.
  If $p = \infty$, then the estimate \eqref{eq:DominatedSequenceThetaMainEstimate} and
  $\| \cdot \|_{L^{\infty} (\mu_{i,j,k})} \leq \| \cdot \|_{L^{\infty}}$
  yield that
  \[
    \|( \zeta_{i,j,k,\ell} )_{k \in \ZZ^d}\|_{\ell^\infty}
    \leq |\det C_j|^{1/p'}
         \cdot \|F_\ell\|_{L^p(\RR^d)}
         \cdot \| (\widecheck{\varphi_i} \ast g_j) \circ C_j \|_{W(L^\infty,\ell^1)}.
  \]
  If $p < \infty$, then \eqref{eq:DominatedSequenceThetaMainEstimate}
  and Lemma~\ref{lem:GeneralLatticeSynthesis} together show that
  \begin{align*}
    \sum_{k \in \ZZ^d}
      \zeta_{i,j,k,\ell}^p
    & \leq |\det C_j|^{p/p'}
           \cdot \| (\widecheck{\varphi_i} \ast g_j) \circ C_j \|_{W(L^\infty, \ell^1)}^{p/p'}
           \cdot \sum_{k \in \ZZ^d}
                   \int_{\RR^d}
                     (F_\ell(x))^p
                     \cdot \big( \Translation{C_j k} |\widecheck{\varphi_i} \ast g_j| \big) (x)
                   \, d x \\
    & =    |\det C_j|^{p/p'}
           \cdot \| ( \widecheck{\varphi_i} \ast g_j ) \circ C_j \|_{W(L^\infty, \ell^1)}^{p/p'}
           \cdot \int_{\RR^d}
                   (F_\ell(x))^p
                   \cdot \big[ D_{C_j, |\widecheck{\varphi_i} \ast g_j|} (1)_{k \in \ZZ^d} \big] (x)
                 \, d x \\
    & \leq |\det C_j|^{p/p'}
           \cdot \|
                   ( \widecheck{\varphi_i} \ast g_j ) \circ C_j
                 \|_{W(L^\infty, \ell^1)}^{1 + (p/p')}
           \cdot \|F_\ell\|_{L^p (\RR^d)}^p \, .
  \end{align*}
  Hence,
  \(
    \| (\zeta_{i,j,k,\ell})_{k \in \ZZ^d} \|_{\ell^p}
    \leq |\det C_j|^{1/p'}
         \cdot \| (\widecheck{\varphi_i} \ast g_j) \circ C_j\|_{W(L^\infty, \ell^1)}
         \cdot \|F_\ell\|_{L^p} \,
  \)
  for any $p \in [1,\infty]$.

  \smallskip{}

  Define $c \in \ell_w^q(I)$ by $c_\ell := \|F_\ell\|_{L^p}$.
  Then, for all $j \in J$,
  \begin{align*}
      &v_j \, |\det C_j|^{\frac{1}{p}-\frac{1}{2}} \, \| (\theta_{j,k})_{k \in \ZZ^d} \|_{\ell^p} \\
      & \quad \quad \quad \leq \sum_{i \in I}
               \sum_{\ell \in i^\ast}
                 v_j \, |\det C_j|^{\frac{1}{p}-\frac{1}{2}}
                 \| (\zeta_{i,j,k,\ell})_{k \in \ZZ^d} \|_{\ell^p} \\
      & \quad \quad \quad \leq \sum_{i \in I}
               \Big[
                 v_j \, |\det C_j|^{\frac{1}{p} - \frac{1}{2}}\, |\det C_j|^{1 - \frac{1}{p}} \,
                 \| (\widecheck{\varphi_i} \ast g_j) \circ C_j\|_{W(L^\infty, \ell^1)}
                 \sum_{\ell \in i^\ast}
                   c_\ell
               \Big] \\
      &\quad \quad \quad  \leq \sum_{i \in I}
               M_{i,j} \cdot w_i \cdot (\Gamma_{\CalQ} \, c)_i, \numberthis
    \label{eq:DominatingSequenceExistenceAlmostDone}
  \end{align*}
  where $M_{i,j}$ is defined as in Equation~\eqref{eq:AnalysisOperatorSchurMatrix}.
  Next, since $\StandardGSI$ is $(w,v,\Phi)$-adapted, Schur's test
  shows that
  \(
    \mathbf{M} :
    \ell^q (I) \to \ell^q(J),
    (d_i)_{i \in I} \mapsto \big( \sum_{i \in I} M_{i,j} d_i \big)_{j \in J}
  \)
  is well-defined and bounded, with norm
  $\|\mathbf{M}\|_{\ell^q \to \ell^q} \leq \|M\|_{\mathrm{Schur}}$.
  Consequently, we obtain
  \begin{align*}
    \| (\theta_{j,k})_{j \in J, k \in \ZZ^d} \|_{Y_v^{p,q}}
          \leq \big\| \mathbf{M} \big(w \cdot \Gamma_{\CalQ} (c) \big) \big\|_{\ell^q (J)}
            \leq \|M\|_{\mathrm{Schur}}
           \cdot \| \Gamma_{\CalQ} \|_{\ell_w^q \to \ell_w^q}
           \cdot \|c\|_{\ell_w^q} \, .
  \end{align*}
  But $\|c\|_{\ell_w^q} = \|F\|_{\ell_w^q(I; L^p)}$,
  and thus the first part of the proof is complete.

  For the proof of the second part, first note
  \[
    \langle
      \widehat{f_n}
      \mid
      \varphi_i \cdot \Fourier [ \Translation{C_j k} g_j ]
    \rangle_{\CalD', \CalD}
    \xrightarrow[n \to \infty]{}
    \langle
      \widehat{f_0}
      \mid
      \varphi_i \cdot \Fourier [ \Translation{C_j k} g_j ]
    \rangle_{\CalD', \CalD}
  \]
  since $\varphi_i \cdot \Fourier [\Translation{C_j k} g_j] \in C_c^\infty (\CalO)$
  and since $f_n \to f_0$ in $Z'(\CalO)$ which implies $\widehat{f_n} \to \widehat{f_0}$
  in $\DistributionSpace(\CalO)$.
  Next, since the $f_n$ are $(F,\Phi)$-dominated,
  Equation~\eqref{eq:DominatedSequenceMainIdentity} shows that
  \[
    |
     \langle
       \widehat{f_n}
       \mid
       \varphi_i \cdot \Fourier [ \Translation{C_j k} g_j ]
     \rangle_{\CalD', \CalD}
    |
    \leq \sum_{\ell \in i^\ast}
           \zeta_{i,j,k,\ell}
   \leq u_j^{-1}
        \sum_{\ell \in i^\ast}
          u_j \, \| (\zeta_{i,j,k,\ell})_{k \in \ZZ^d}\|_{\ell^p}
   =: \gamma_{i,j},
  \]
  while Equation~\eqref{eq:DominatingSequenceExistenceAlmostDone}
  shows that $\sum_{i \in I} \gamma_{i,j} < \infty$.
  Thus,
  \[
    \langle f_n \mid \Translation{C_j k} g_j \rangle_{\Phi}
    \xrightarrow[n \to \infty]{}
        \langle f_0 \mid \Translation{C_j k} g_j \rangle_{\Phi}
  \]
  by definition of $\langle \cdot \mid \cdot \rangle_{\Phi}$ and by
  the dominated convergence theorem.
\end{proof}

We now prove the boundedness of the coefficient and reconstruction operators.

\begin{proposition} \label{prop:AnalysisSynthesisOperatorGeneral}
  Let $\mathcal{D}(\CalQ, L^p, \ell_w^q)$ be a decomposition space and
  let $Y_v^{p,q}$ be the sequence space associated to the GSI system $\StandardGSI$
  as per Definition~\ref{def:CoefficientSpace}.
  Suppose that $\StandardGSI$ is $(w,v,\Phi)$-adapted (where $\Phi$ is a BAPU for $\CalQ$)
  with matrix $M$ as in \eqref{eq:AnalysisOperatorSchurMatrix}.
  Then
  \begin{enumerate}
    \item[(i)] For all $p, q \in [1,\infty]$, the \emph{reconstruction map}
    \[
      \synthesis : Y_v^{p,q} \to \DecompSp(\CalQ,L^p,\ell_w^q), \quad
                   (c_k^{(j)})_{j \in J, k \in \ZZ^d}
                   \mapsto \sum_{j \in J}
                             \sum_{k \in \ZZ^d}
                               c_k^{(j)} \cdot \Translation{C_j k} g_j
    \]
    is well-defined and bounded with
    $\|\synthesis\|_{Y_v^{p,q} \to \DecompSp(\CalQ,L^p,\ell_w^q)} \leq \|M\|_{\schur}$.
    Furthermore, the defining double series converges unconditionally in $Z'(\CalO)$.

    \item[(ii)]  For all $p,q \in [1,\infty]$, the \emph{coefficient operator}
      \[
        \analysis : \DecompSp(\CalQ, L^p, \ell_w^q) \to Y_v^{p,q},
                    \quad
                    f \mapsto \Big(
                                \langle
                                  f \mid \Translation{C_j k} g_j
                                \rangle_{\Phi}
                              \Big)_{j \in J, k \in \ZZ^d}
      \]
      is well-defined and bounded with
      \(
        \|\analysis\|_{\DecompSp(\CalQ, L^p, \ell_w^q) \to Y_v^{p,q}}
        \leq \|M\|_{\mathrm{Schur}} \cdot \|\Gamma_{\CalQ}\|_{\ell_w^q \to \ell_w^q}
      \).

  \item[(iii)]
      If $\Psi$ is another BAPU for $\CalQ$,
      and if $f \in \DecompSp(\CalQ, L^p, \ell_w^q)$,
      then $\langle f \mid \Translation{C_j k} g_j \rangle_{\Psi}$ is
      well-defined and satisfies
      \(
        \langle f \mid \Translation{C_j k} g_j \rangle_{\Psi}
         = \langle f \mid \Translation{C_j k} g_j \rangle_{\Phi}
      \)
      for all $j \in J$ and $k \in \ZZ^d$.
  \end{enumerate}
\end{proposition}

\begin{proof}
To prove (i),  let $c = (c_k^{(j)})_{j \in J, k \in \ZZ^d} \in Y_v^{p,q}$ be arbitrary, and
 set $c^{(j)} := (c_k^{(j)})_{k \in \ZZ^d}$ for $j \in J$.
 Then $c^{(j)} \in \ell^p(\ZZ^d)$.
 Moreover, if  $d = (d_j)_{j \in J}$ is defined as
 $d_j := |\det C_j|^{\frac{1}{p} - \frac{1}{2}} \cdot \| c^{(j)} \|_{\ell^p}$,
 then $d \in \ell_v^q (J)$ and $\| d \|_{\ell_v^q} = \| c \|_{Y_v^{p,q}}$.
 Finally, let $|c^{(j)}| = (|c_k^{(j)}|)_{k \in \ZZ^d}$ for $j \in J$.

 We first prove the unconditional convergence of the double series defining $\synthesis c$.
 Since the Fourier transform $\Fourier : Z'(\CalO) \to \DistributionSpace(\CalO)$
 is a linear homeomorphism, it suffices to show that the double series
 $\sum_{j \in J} \sum_{k \in \ZZ^d} c_k^{(j)} \Fourier[ \Translation{C_j k} g_j ]$
 converges unconditionally in $\DistributionSpace(\CalO)$.
 To prove this, let $K \subset \CalO$ be compact.
 Since $\sum_{i \in I} \varphi_i \equiv 1$ on $\CalO$, the family
 $\big( \varphi_i^{-1} (\CC \setminus \{0\}) \big)_{i \in I}$ forms an open cover
 of $\CalO \supset K$.
 By compactness of $K$, there is a finite set $I_K \subset I$ for which
 \(
   K
   \subset \bigcup_{i \in I_K} \varphi_i^{-1} (\CC \setminus \{0\})
   \subset \bigcup_{i \in I_K} Q_i
 \).
 Note that $I_K^\ast := \bigcup_{\ell \in I_K} \ell^\ast \subset I$ is finite.
 Furthermore, for $j \in I \setminus I_K^\ast$, note that
 $Q_j \cap K \subset \bigcup_{i \in I_K} Q_j \cap Q_i = \emptyset$, whence
 $\varphi_j \equiv 0$ on $K$.
 Thus, any $g \in C_c^\infty (\CalO) \subset \Schwartz(\RHat^d)$
 with $\supp g \subset K$ can be written as
 $g = \sum_{i \in I} \varphi_i \, g = \sum_{i \in I_K^\ast} \varphi_i \, g$.
 A direct calculation using Lemma~\ref{lem:GeneralLatticeSynthesis} therefore shows
 \begin{align*}
    & \sum_{j \in J}
       \sum_{k \in \ZZ^d}
         |c_k^{(j)}|
         \cdot \big|
                 \langle \Fourier[\Translation{C_j k} g_j] , g \rangle_{\CalD', \CalD}
               \big|\\
     &\quad \quad \quad \leq \sum_{i \in I_K^\ast}
              \sum_{j \in J}
                \sum_{k \in \ZZ^d}
                  |c_k^{(j)}|
                  \cdot \big|
                          \langle
                            \varphi_i \, \Fourier[\Translation{C_j k} g_j] , g
                          \rangle_{\Schwartz', \Schwartz}
                        \big| \\
          & \quad \quad \quad \leq \sum_{i \in I_K^\ast}
              \sum_{j \in J}
                \int_{\RR^d}
                  |\widehat{g}(x)|
                  \sum_{k \in \ZZ^d}
                    |c_k^{(j)}| \, \big( \Translation{C_j k} |\widecheck{\varphi_i} \ast g_j| \big) (x)
                \, d x \\
     & \quad \quad \quad \leq \sum_{i \in I_K^\ast}
              \sum_{j \in J}
                \| \widehat{g} \|_{L^{p'}}
                \cdot \big\| D_{C_j, |\widecheck{\varphi_i} \ast g_j|} |c^{(j)}| \big\|_{L^p} \\
     & \quad \quad \quad \leq \|\widehat{g}\|_{L^{p'}}
            \sum_{i \in I_K^\ast}
              \sum_{j \in J}
                |\det C_j|^{\frac{1}{p}} \,
                \big\|
                  (\widecheck{\varphi_i} \!\ast\! g_j) \!\circ\! C_j
                \big\|_{W(L^\infty,\ell^1)}
                \, \|\, |c^{(j)}| \,\|_{\ell^p} \\
     & \quad \quad \quad \leq \| \widehat{g} \|_{L^{p'}}
            \cdot \sum_{i \in I_K^\ast}
                    \Big[
                      w_i^{-1}
                      \sum_{j \in J}
                        v_j \, d_j \, M_{i,j}
                    \Big] \\
     & \quad \quad \quad \leq \| \widehat{g} \|_{L^{p'}}
            \cdot \| d \|_{\ell_v^q}
            \cdot \| M \|_{\schur}
            \cdot \sum_{i \in I_K^\ast}
                    w_i^{-1}
       <    \infty \, . \numberthis
   \label{eq:SynthesisOperatorUnconditionalConvergenceMainEstimate}
 \end{align*}
 Since $g \mapsto \|\widehat{g}\|_{L^{p'}}$ is a continuous norm on $C_c^\infty (\CalO)$
 and since $g \in C_c^\infty (\CalO)$ with $\supp g \subset K$ was arbitrary,
 the desired unconditional convergence follows.

 Next, we show that $\synthesis : Y_v^{p,q} \to \StandardDecompSp$ is well-defined and bounded.
 For $i \in I$ and $j \in J$, define
 \(
   B_{i,j} := |\det C_j|^{\frac{1}{2}}
              \cdot \| (\widecheck{\varphi_i} \ast g_j) \circ C_j \|_{W (L^{\infty}, \ell^1)}
 \).
 The assumption that $\StandardGSI$ is $(w,v,\Phi)$-adapted yields by Schur's test that the map
 \(
   \mathbf{B} :
   \ell^q_v (J) \to \ell^q_w (I),
   \;
   (d_j)_{j \in J} \mapsto \big(
                             \sum_{j \in J}
                               B_{i,j} \cdot d_j
                           \big)_{i \in I} \,
 \)
 is bounded with $\|\mathbf{B}\|_{\op} \leq \| M \|_{\schur}$.
 The series defining $\synthesis c$ being unconditionally convergent yields
 \[
   \Fourier^{-1} (\varphi_i \cdot \widehat{\synthesis \, c})
   = \sum_{j \in J}
       \sum_{k \in \ZZ^d}
         c_k^{(j)} \, \Fourier^{-1} (\varphi_i \, \Fourier [\Translation{C_j k} g_j])
   = \sum_{j \in J}
       D_{C_j, \widecheck{\varphi_i} \ast g_j} \, c^{(j)} \, .
 \]
 Therefore, an application of Lemma~\ref{lem:GeneralLatticeSynthesis} shows
 \begin{align*}
   \big\| \Fourier^{-1} (\varphi_i \cdot \widehat{\synthesis \, c}) \big\|_{L^p}
  & \leq \sum_{j \in J}
          |\det C_j|^{\frac{1}{p}}
          \cdot \| (\widecheck{\varphi_i} \ast g_j) \circ C_j \|_{W(L^\infty, \ell^1)}
          \cdot \| c^{(j)} \|_{\ell^p} \\
   &=    \sum_{j \in J}
           B_{i,j} \, d_j
   =    (\mathbf{B} \, d)_i
   <    \infty \, ,
 \end{align*}
 whence
 \(
   \|  \synthesis \, c \|_{\StandardDecompSp}
   \leq \| \mathbf{B} \, d \|_{\ell_w^q}
   \leq \| M \|_{\schur} \cdot \| d \|_{\ell_v^q}
   =    \| M \|_{\schur} \cdot \| c \|_{Y_v^{p,q}}
 \).

 \smallskip{}

 To prove (ii), let $f \in \DecompSp(\CalQ,L^p,\ell_w^q)$ be arbitrary.
 Define $F_i := |\Fourier^{-1} (\varphi_i \widehat{f} \, )|$ for $i \in I$.
 Then $F = (F_i)_{i \in I} \in \ell_w^q (I; L^p)$
 and  $\|F\|_{\ell_w^q(I; L^p)} = \|f\|_{\DecompSp(\CalQ, L^p, \ell_w^q)}$.
 Clearly, $f$ is $(F,\Phi)$-dominated.
 Therefore, Lemma~\ref{lem:DominatedFunctionsHaveDominatedCoefficients} yields
 $\theta = (\theta_{j,k})_{j \in J, k \in \ZZ^d} \in Y_v^{p,q}$ satisfying the estimate
 $| \langle f \mid \Translation{C_j k} g_j \rangle_{\Phi}| \leq \theta_{j,k}$
 for all $j \in J$ and $k \in \ZZ^d$, and furthermore
 \(
   \| \theta \|_{Y_v^{p,q}}
   \leq \|M\|_{\mathrm{Schur}}
        \cdot \|\Gamma_{\CalQ}\|_{\ell_w^q \to \ell_w^q}
        \cdot \|F\|_{\ell_w^q (I; L^p)}
 \).
 Hence, $\analysis : \DecompSp(\CalQ, L^p, \ell_w^q) \to Y_v^{p,q}$
 is well-defined and bounded, with the claimed estimate for the operator norm.

 \smallskip{}

 Assertion (iii) is a direct consequence of Lemma~\ref{lem:AdaptednessIndependent}.
\end{proof}

Proposition~\ref{prop:AnalysisSynthesisOperatorGeneral} shows in particular that the reconstruction
operator $\synthesis : Y_v^{p,q} \!\to \DecompSp(\CalQ, L^p, \ell_w^q)$ is continuous.
However, in case $\max \{ p, q \} = \infty$, the convergence in $Y_v^{p,q}$ is a quite
restrictive condition.
To accommodate for this, we will often employ the following lemma.

\begin{lemma}\label{lem:ReconstructionOperatorSpecialContinuity}
  Under the assumptions of Proposition~\ref{prop:AnalysisSynthesisOperatorGeneral},
  the following holds:

  \smallskip{}

  For each $n \in \NN$, let
  $c^{(n)} = (c^{(n)}_{j,k})_{j \in J, k \in \ZZ^d} \in Y_v^{p,q}$
  be such that $c^{(n)}_{j,k} \xrightarrow[n \to \infty]{} c_{j,k} \in \CC$ for all $j \in J$
  and $k \in \ZZ^d$. Suppose there exists a sequence
  $\theta = (\theta_{j,k})_{j \in J, k \in \ZZ^d} \in Y_v^{p,q}$ satisfying
  $|c^{(n)}_{j,k}| \leq \theta_{j,k}$ for all $j \in J$, $k \in \ZZ^d$, and $n \in \NN$.
  Then the reconstruction operator $\synthesis$ satisfies
  $\synthesis \, c^{(n)} \xrightarrow[n \to \infty]{Z'(\CalO)} \synthesis \, c$.
\end{lemma}

\begin{proof}
  Let $f \in Z(\CalO)$.
  Then $K := \supp \Fourier^{-1} f \subset \CalO$ is compact.
  Since $(\varphi_i^{-1} (\CC \setminus \{0\}))_{i \in I}$ is an open cover of $ K$,
  there is a finite set $I_0 \subset I$ satisfying
  \(
    K
    \subset \bigcup_{i \in I_0}
              \varphi_i^{-1} (\CC \setminus \{0\})
    \subset \bigcup_{i \in I_0}
              Q_i
  \).
  This easily implies $Q_i \cap K = \emptyset$ for $i \in I \setminus I_f$,
  where $I_f := I_0^\ast := \bigcup_{\ell \in I_0} \ell^\ast \subset I$ is finite.
  Thus, $\varphi_i \cdot \Fourier^{-1}f \equiv 0$ for $i \in I \setminus I_f$,
  and hence $\Fourier^{-1} f = \sum_{i \in I_f} \varphi_i \, \Fourier^{-1} f$.
  Therefore,
  \begin{align*}
    \langle \Translation{C_j k} g_j, f \rangle_{\Schwartz', \Schwartz}
&    = \langle
        \Fourier [ \Translation{C_j k} g_j ],
        \Fourier^{-1} f
      \rangle_{\Schwartz', \Schwartz}
    = \sum_{i \in I_f}
        \langle
          \varphi_i \, \Fourier [\Translation{C_j k} g_j],
          \Fourier^{-1} f
        \rangle_{\Schwartz', \Schwartz} \\
&    = \sum_{i \in I_f}
        \langle
          \Translation{C_j k} (\widecheck{\varphi_i} \ast g_j),
          f
        \rangle_{\Schwartz', \Schwartz} \, .
  \end{align*}
  For $\nu = (\nu_{j,k})_{j \in J, k \in \ZZ^d} \in Y_v^{p,q}$,
  it follows therefore by the convergence in $Z'(\CalO)$ of the series defining $\synthesis \nu$ that
  \begin{equation}
    \langle \synthesis \nu, f \rangle_{Z', Z}
      = \sum_{j \in J}
          \sum_{k \in \ZZ^d}
            \nu_{j,k} \, \langle \Translation{C_j k} g_j, f \rangle_{\Schwartz', \Schwartz}
      = \sum_{i \in I_f}
          \sum_{j \in J}
            \sum_{k \in \ZZ^d}
              \nu_{j,k} \,
              \langle
                \Translation{C_j k} (\widecheck{\varphi_i} \ast g_j),
                f
              \rangle_{\Schwartz', \Schwartz} \, .
    \label{eq:ReconstructionOperatorSpecialContinuityMainIdentity}
  \end{equation}
  Next, Lemma~\ref{lem:GeneralLatticeSynthesis} shows that
  \begin{align*}
           \sum_{k \in \ZZ^d}
             \theta_{j,k} \,
             |
              \langle
                \Translation{C_j k} (\widecheck{\varphi_i} \ast g_j),
                f
              \rangle_{\Schwartz', \Schwartz}
             |
    & \leq \int_{\RR^d}
             |f(x)|
             \cdot \sum_{k \in \ZZ^d}
                     \Big[
                       \theta_{j,k}
                       \cdot \big( \Translation{C_j k} |\widecheck{\varphi_i} \ast g_j| \big) (x)
                     \Big]
           \, d x \\
    & \leq \|f\|_{L^{p'}}
           \cdot \big\|
                   D_{C_j, |\widecheck{\varphi_i} \ast g_j|}
                   \big( (\theta_{j,k})_{k \in \ZZ^d} \big)
                 \big\|_{L^p} \\[0.15cm]
    & \leq \|f\|_{L^{p'}}
           \cdot |\det C_j|^{1/p}
           \cdot \| (\widecheck{\varphi_i} \ast g_j) \circ C_j \|_{W(L^\infty, \ell^1)}
           \cdot \gamma_j \, ,
  \end{align*}
  where we defined $\gamma_j := \| (\theta_{j,k})_{k \in \ZZ^d} \|_{\ell^p}$ in the last step.

  For brevity, let $u_j := v_j \cdot |\det C_j|^{\frac{1}{p} - \frac{1}{2}}$.
  Note that since $\theta \in Y_v^{p,q}$, we have
  $\gamma = (\gamma_j)_{j \in J} \in \ell_u^q \hookrightarrow \ell_u^\infty$,
  which yields a constant $C_1 > 0$ such that $u_j \, \gamma_j \leq C_1$ for all $j \in J$.
  Using this, we see
  \begin{align*}
    & \quad
      \sum_{i \in I_f}
        \sum_{j \in J}
          \sum_{k \in \ZZ^d}
            \theta_{j,k} \,
            |
             \langle
               \Translation{C_j k} (\widecheck{\varphi_i} \ast g_j),
               f
             \rangle_{\Schwartz', \Schwartz}
            | \\
    & \leq \|f\|_{L^{p'}}
           \sum_{i \in I_f}
             \Big[
               w_i^{-1}
               \sum_{j \in J}
                 \frac{w_i}{v_j}
                 \cdot |\det C_j|^{\frac{1}{2}}
                 \cdot \| (\widecheck{\varphi_i} \ast g_j) \circ C_j \|_{W(L^\infty, \ell^1)}
                 \cdot u_j \, \gamma_j
             \Big] \\
    & \leq C_1 \cdot \|f\|_{L^{p'}}
           \sum_{i \in I_f}
               \sum_{j \in J}
                 \frac{M_{i,j}}{w_i}
      \leq C_1
           \cdot \|f\|_{L^{p'}}
           \cdot \Big( \sum_{i \in I_f} w_i^{-1} \Big)
           \cdot \|M\|_{\mathrm{Schur}}
      < \infty \, .
  \end{align*}
  Finally, since $|c^{(n)}_{j,k}| \leq \theta_{j,k}$ for all $j \in J$, $k \in \ZZ^d$,
  and $n \in \NN$, and since $c^{(n)}_{j,k} \xrightarrow[n \to \infty]{} c_{j,k}$,
  applying the dominated convergence theorem in
  Equation~\eqref{eq:ReconstructionOperatorSpecialContinuityMainIdentity} shows that
  \[
    \langle \synthesis c^{(n)} , f \rangle_{Z',Z}
     \xrightarrow[n \to \infty]{}
      \langle  \synthesis c, f \rangle_{Z',Z} ,
  \]
  as desired.
\end{proof}

\begin{corollary}\label{cor:FrameOperatorSpecialContinuity}
  Under the assumptions of Proposition~\ref{prop:AnalysisSynthesisOperatorGeneral},
  the following holds:
  The frame operator
  $S := \synthesis \circ \analysis : \DecompSp(\CalQ,L^p,\ell_w^q) \to \DecompSp(\CalQ,L^p,\ell_w^q)$
  is well-defined and bounded.

  Furthermore, if $(f_n)_{n \in \NN} \subset \DecompSp(\CalQ,L^p,\ell_w^q)$ is a sequence
  satisfying $f_n \to f \in Z'(\CalO)$, with convergence in $Z'(\CalO)$, and
  for which there exists $F \in \ell_w^q (I; L^p)$ such that all $f_n$ are $(F, \Phi)$-dominated,
  then $f \in \DecompSp(\CalQ,L^p,\ell_w^q)$ and $S f_n \to S f$ with convergence in $Z'(\CalO)$.
\end{corollary}

\begin{proof}
  $S$ is well-defined, bounded by Proposition~\ref{prop:AnalysisSynthesisOperatorGeneral}.
  Since
  $ \|f_n\|_{\DecompSp(\CalQ,L^p,\ell_w^q)} \leq \|F\|_{\ell_w^q (I; L^p)}  $
  for all $n \in \NN$, Lemma~\ref{lem:DecompositionSpaceFatouProperty} yields
  $f \in \DecompSp(\CalQ,L^p,\ell_w^q)$, where $c := \analysis \, f \in Y_v^{p,q}$.
  Next, Lemma~\ref{lem:DominatedFunctionsHaveDominatedCoefficients} shows that
  there is a sequence $\theta = (\theta_{j,k})_{j \in J, k \in \ZZ^d} \in Y_v^{p,q}$
  such that if we set $c^{(n)} := \analysis f_n$, then $| c^{(n)}_{j,k} | \leq \theta_{j,k}$
  for all $(n,j,k) \in \NN \times J \times \ZZ^d$.
  The same lemma also shows that $c^{(n)}_{j,k} \to c_{j,k}$ for all $j \in J$ and $k \in \ZZ^d$.
  Therefore, Lemma~\ref{lem:ReconstructionOperatorSpecialContinuity} shows that
  $S f_n = \synthesis \, c^{(n)} \to \synthesis \, c = S f$ with convergence in $Z'(\CalO)$.
\end{proof}

\section{Invertibility of the frame operator}
\label{sec:FrameOperatorInvertible}

\subsection{Representation of the frame operator}
The frame properties of generalized shift-invariant systems are usually
studied under a compatibility condition that controls the interaction
between the generating functions and the translation lattices of the system.
Specifically, we will use the so-called local integrability conditions
\cite{hernandez2002unified,JakobsenReproducing2014,velthoven2019on}.

\begin{definition}
For an open set $\CalO \subset \RHat^d$ of full measure, let
\[
  \mathcal{B}_{\CalO} (\RR^d)
  := \bigg\{
       f \in L^2 (\RR^d)
       \; : \;
       \widehat{f} \in L^{\infty} (\RHat^d)
       \text{ and }
       \supp \widehat{f} \subset \CalO \text{ compact}
     \bigg\}.
\]
A generalized shift-invariant system $\StandardGSI$ is said to satisfy the
\emph{$\alpha$-local integrability condition ($\alpha$-LIC)}, relative to $\mathcal{O}^c$,
if, for all $f \in \mathcal{B}_{\CalO} (\RR^d)$,
\begin{align}\label{eq:LIC}
  \sum_{j \in J}
    \frac{1}{|\det C_j|}
    \sum_{\alpha \in C_j^{-t} \mathbb{Z}^d } \,\,\,
      \int_{\RHat^d} \,
       |\widehat{f} (\xi) \widehat{f} (\xi + \alpha)
       \widehat{g_j} (\xi) \widehat{g}_j (\xi + \alpha)|
      \, d \xi
  < \infty .
\end{align}
\end{definition}

Given $\StandardGSI$, we set
$\Lambda := \bigcup_{j \in J} C_j^{-t} \ZZ^d$
and $\kappa (\alpha) := \{ j \in J \; : \; \alpha \in C_j^{-t} \ZZ^d\}$
for $\alpha \in \Lambda$.
For $\alpha \in \Lambda$, we define the functions
\begin{equation}
 t_{\alpha} : \RHat^d \to \CC,
 \quad \xi \mapsto
   \sum_{j \in \kappa (\alpha)}
                \frac{1}{|\det C_j|} \,\,
                \overline{\widehat{g}_j (\xi)}
                \, \widehat{g_j} (\xi + \alpha).
  \label{eq:TAlphaDefinition}
\end{equation}
Note that $t_{\alpha} \in L^{\infty} (\RHat^d)$ for all $\alpha \in \Lambda$
by \eqref{eq:GSI_assumption}.
Furthermore, $t_\alpha (\xi - \alpha) = \overline{t_{-\alpha} (\xi)}$.

Under the $\alpha$-local integrability condition, the following (weak-sense) representation
of the frame operator can be obtained;
this follows by polarization from the proofs of
\cite[Proposition~2.4]{hernandez2002unified} and \cite[Theorem 3.4]{JakobsenReproducing2014}.

\begin{proposition}\label{prop:GSI_identity}
Suppose $\StandardGSI$ satisfies the $\alpha$-local integrability condition
\eqref{eq:LIC}, relative to $\CalO^c$.
Then, for all $f_1, f_2 \in \mathcal{B}_{\CalO} (\RR^d)$,
\begin{align*}
  \sum_{(j,k) \in J \times \ZZ^d} \!\!\!\!\!
      \langle f_1 \mid T_{C_j k} g_j \rangle
      \langle T_{C_j k} g_j \mid f_2 \rangle
 & \!=\! \sum_{\alpha \in \Lambda}
          \int_{\RHat^d}
            \widehat{f}_1 (\xi) \,
            \overline{\widehat{f}_2 (\xi + \alpha)} \,
            t_{\alpha} (\xi)
          \; d\xi \\
&  \!=\! \sum_{\alpha \in \Lambda}
          \langle
            \Fourier^{-1} \big[ \Translation{\alpha} (t_\alpha \, \widehat{f_1}) \big]
            \mid
            f_2
          \rangle_{L^2} \, ,
\numberthis  \label{eq:WalnutRepresentation}
\end{align*}
where the series converges absolutely; in fact,
\begin{equation}
  \sum_{\alpha \in \Lambda}
    \int_{\RHat^d} \,\,
       |\widehat{f_1} (\xi) \,
       \widehat{f_2}(\xi+\alpha)|
        \sum_{j \in \kappa (\alpha)}
          \frac{1}{|\det C_j|}
          | \widehat{g_j} (\xi)
          \widehat{g_j} (\xi+\alpha)|
    \, d \xi < \infty \, .
  \label{eq:WalnutRepresentationAbsoluteConvergence}
\end{equation}
\end{proposition}

Proposition~\ref{prop:GSI_identity} yields an analogous representation of
the frame operator on $\StandardDecompSp$, at least on the subspace $\DenseSpace(\RR^d)$.

\begin{corollary}\label{cor:FrameOperatorOnDenseSpace}
  Under the assumptions of Proposition~\ref{prop:GSI_identity}, the series
  $\sum_{\alpha \in \Lambda_0} \Fourier^{-1} [\Translation{\alpha} (t_\alpha \, \widehat{f} \, )]$
  converges unconditionally in $Z'(\CalO)$ for any subset $\Lambda_0 \subset \Lambda$,
  and any $f \in \DenseSpace(\RR^d)$.

  Furthermore, if $\CalQ$ is a decomposition cover of $\CalO$, with subordinate BAPU $\Phi$,
  if $w$ is $\CalQ$-moderate, and if $v = (v_j)_{j \in J}$ is a weight such that
  $\StandardGSI$ is $(w,v,\Phi)$-adapted, then the frame operator
  $S : \StandardDecompSp \to \StandardDecompSp$ fulfills for each $f \in \DenseSpace(\R^d)$
  the identity
  \begin{equation}
    S f
    = \sum_{j \in J}
        \sum_{k \in \ZZ^d}
          \langle f \mid \Translation{C_j k} g_j \rangle_{\Phi} \,
          \Translation{C_j k} g_j
    = \sum_{j \in J}
        \sum_{k \in \ZZ^d}
          \langle f \mid \Translation{C_j k} g_j \rangle_{L^2} \,
          \Translation{C_j k} g_j
    = \sum_{\alpha \in \Lambda}
        \Fourier^{-1} \big[ \Translation{\alpha} (t_\alpha \, \widehat{f}) \big] .
    \label{eq:FrameRepresentationOnDenseSpace}
  \end{equation}
\end{corollary}

\begin{proof}
  Since $t_\alpha \in L^\infty (\RHat^d)$ and $\widehat{f} \in \Schwartz(\RHat^d)$, we have
  \(
    \Translation{\alpha} (t_\alpha \, \widehat{f})
    \in L^1 (\RHat^d)
    \hookrightarrow \Schwartz'(\RHat^d)
    \hookrightarrow \DistributionSpace (\CalO)
  \),
  and hence $\Fourier^{-1} [\Translation{\alpha} (t_\alpha \, \widehat{f} \, )] \in Z'(\CalO)$.
  The Fourier transform $\Fourier : Z'(\CalO) \to \CalD'(\CalO)$ is a linear homeomorphism;
  hence, it suffices to prove that the series
  $\sum_{\alpha \in \Lambda_0} \Translation{\alpha} (t_\alpha \, \widehat{f} \, )$
  converges unconditionally in $\CalD'(\CalO)$.
  To see this, let $K \subset \CalO$ be compact.
  Define $f_1 := f \in \DenseSpace(\RR^d) \subset \CalB_{\CalO} (\RR^d)$,
  and set $f_2 := \Fourier^{-1} \Indicator_K \in \CalB_{\CalO} (\RR^d)$.
  By Equation~\eqref{eq:WalnutRepresentationAbsoluteConvergence},
  the constant
  \(
    C_K := \sum_{\alpha \in \Lambda}
             \int_{\RHat^d}
               |\widehat{f} (\xi)| \,
               \Indicator_K (\xi + \alpha) \,
               |t_\alpha (\xi)|
             \, d \xi
  \)
  is finite.
  Now, let $\psi \in C_c^\infty (\CalO)$ be arbitrary with $\supp \psi \subset K$.
  Then
  \begin{align*}
      \sum_{\alpha \in \Lambda_0}
        \big|
          \langle
            \Translation{\alpha} (t_\alpha \, \widehat{f} \, ) ,
            \psi
          \rangle_{\CalD', \CalD}
        \big|
      & \leq \sum_{\alpha \in \Lambda}
               \|\psi\|_{L^\infty} \!\!
               \int_{\RHat^d}
                 |t_\alpha (\eta - \alpha) \, \widehat{f}(\eta - \alpha)|
                 \cdot \Indicator_K (\eta)
               \, d \eta \\
        &      = C_K \, \|\psi\|_{L^\infty} < \infty.
    \numberthis
    \label{eq:WalnutRepresentationUnconditionalConvergenceMainEstimate}
  \end{align*}
  Since $\|\cdot\|_{L^\infty}$ is continuous with respect to the topology on $C_c^\infty (\CalO)$,
  and since $\psi \in C_c^\infty (\CalO)\vphantom{\sum_j}$ with $\supp \psi \subset K$ was arbitrary,
  the estimate \eqref{eq:WalnutRepresentationUnconditionalConvergenceMainEstimate} simultaneously
  yields that
  $\sum_{\alpha \in \Lambda_0} \Translation{\alpha} (t_\alpha \, \widehat{f} \, ) \in \CalD'(\CalO)$,
  cf.~\cite[Theorem 6.6]{RudinFunctionalAnalysis},
  as well as the unconditional convergence of the series in $\CalD'(\CalO)$.

  For the remaining part, note if
  $f \in \DenseSpace(\RR^d)$, then
  \(
    \langle f \mid \Translation{C_j k} g_j \rangle_{\Phi}
    = \langle f \mid \Translation{C_j k} g_j \rangle_{L^2}
  \)
  by Remark~\ref{rem:SpecialDualPairing}.
  This proves everything but the last equality in Equation~\eqref{eq:FrameRepresentationOnDenseSpace}.
  To prove this, let $g \in Z(\CalO)$.
  Then $\widehat{\overline{g}} = \overline{\Fourier^{-1} g} \in C_c^\infty (\CalO)$,
  and hence $\overline{g} \in \CalB_{\CalO} (\RR^d)$.
  This, together with Equation~\eqref{eq:WalnutRepresentation}, shows
  \[
    \langle Sf , g \rangle_{Z', Z}
    = \langle S f \mid \overline{g} \rangle_{L^2}
    = \sum_{\alpha \in \Lambda}
        \big\langle
          \Fourier^{-1} \big[ \Translation{\alpha} (t_\alpha \, \widehat{f} \, ) \big]
          \mid
          \overline{g}
        \big\rangle_{L^2}
    = \Big\langle
        \sum_{\alpha \in \Lambda}
          \Fourier^{-1} \big[ \Translation{\alpha} (t_\alpha \, \widehat{f} \, ) \big] , g
      \Big\rangle_{Z', Z} ,
  \]
  and hence \eqref{eq:FrameRepresentationOnDenseSpace} follows.
\end{proof}

\subsection{Towards invertibility}
According to Corollary~\ref{cor:FrameOperatorOnDenseSpace},
on the set $\DenseSpace(\RR^d)$, the frame operator can be represented as
\begin{equation}
  S f = T_0 f + R f \, ,
  \label{eq:FrameOperatorSplit}
\end{equation}
with
\begin{align} \label{eq:MainTermDefinition}
   T_0 f & = \Fourier^{-1} ( \, t_0 \cdot \widehat{f} \, )
\end{align}
and
\begin{align}
   R   f & = \Fourier^{-1}
             \Big(
               \sum_{\alpha \in \Lambda \setminus \{0\}}
                 \Translation{\alpha} (\, t_\alpha \cdot \widehat{f} \, )
             \Big) \, ,
  \label{eq:RemainderTermDefinition}
\end{align}
for $f \in \DenseSpace(\RR^d)$.
In the following, we estimate the norms of $T_0^{-1}$ and $R$ as operators
on the decomposition space $\DecompSp (\CalQ, L^p, \ell_w^q)$.
This will be used, together with the following elementary result,
to provide conditions ensuring that the frame operator is invertible.

\begin{lemma}\label{lem:FrameOperatorInvertibility}
  Let $X$ be a Banach space, and let $S : X \to X$ be a linear operator that
  can be written as $S = T_0 + R$, where $T_0 , R$ are bounded
  linear operators on $X$. Finally, assume that $T_0$ is boundedly invertible and that
  \[
    \| T_0^{-1} \|_{X \to X} \cdot \| R \|_{X \to X} < 1 \, .
  \]
  Then $S : X \to X$ is also boundedly invertible.
\end{lemma}

\begin{proof}
  We have $S = T_0 + R = T_0 \big( \identity_X - (- T_0^{-1} R) \big)$.
  But $\|- T_0^{-1} R\|_{X \to X} \leq \|T_0^{-1}\|_{X \to X} \cdot \| R \|_{X \to X} < 1$,
  so that $\identity_X - (- T_0^{-1} R)$ is boundedly invertible by a Neumann series argument.
  This implies that $S$ is boundedly invertible as a composition of boundedly invertible operators.
\end{proof}

\subsection{Estimates for Fourier multipliers}
\label{sub:MainTermEstimates}

The operator $T_0$ is a Fourier multiplier, and we aim to estimate its inverse.
As a first step, we prove a general result concerning the boundedness of Fourier multipliers
on Besov-type spaces; see Proposition~\ref{prop:FourierMultiplierBound} below.
More qualitative versions of that proposition
can be found in \cite[Section 2.4.3]{TriebelFourierAnalysisAndFunctionSpaces},
\cite[Section 2.3]{tr78} and
\cite[Theorem~2.11]{DecompositionSpaces1}.
Corresponding results for Triebel-Lizorkin spaces hold under more stringent assumptions
on the decomposition cover; see
\cite[Sections 2.4.2 and 2.5.4]{TriebelFourierAnalysisAndFunctionSpaces} and \cite{sttr79}.

In contrast to \cite[Section 2.4.3]{TriebelFourierAnalysisAndFunctionSpaces},
we consider Fourier symbols with limited regularity. This entails certain
technical difficulties because of our choice of the reservoir
$Z'(\CalO)$, where $Z(\CalO) = \Fourier(C_c^\infty (\CalO))$.
More precisely, if $f \in \DecompSp(\CalQ,L^p,\ell_w^q) \subset Z'(\CalO)$,
then $\widehat{f} \in \CalD'(\CalO)$ is a distribution, and can be multiplied by a
function $h \in C^\infty (\CalO)$.
We need, however, to make sense of the product with more general functions $h$,
by fully exploiting the fact that $f \in \DecompSp(\CalQ,L^p,\ell_w^q)$.
To this end, we introduce the following notion:

\begin{definition}\label{def:SpecialMultiplication}
  Let $p \in [1,\infty]$.
  For $f \in \Fourier L^1(\R^d)$ and $g \in \Fourier L^p (\R^d)$, we define
  the \emph{generalized product} of $f$ and $g$ as
  \[
    f \odot g := \Fourier[(\Fourier^{-1} f) \ast (\Fourier^{-1} g)]
              \in \Fourier L^p (\R^d) \subset \Schwartz'(\R^d).
  \]
\end{definition}

\begin{remark}\label{rem:SpecialMultiplication}
  The definition makes sense because of Young's inequality:
  $(\Fourier^{-1} f) \ast (\Fourier^{-1} g) \in L^p(\R^d)$.
  Furthermore, our definition indeed generalizes the usual product:
  if $f \in \Schwartz (\R^d)$ and ${g \in \Schwartz'(\R^d)}$, then
  $f \cdot g = \Fourier[(\Fourier^{-1} f) \ast (\Fourier^{-1} g)]$%
  ---see, for instance \cite[Theorem 7.19]{RudinFunctionalAnalysis}.
\end{remark}

We can now derive an estimate for Fourier multipliers on decomposition spaces.
The proof is deferred to Appendix~\ref{sec:FourierMultiplierProof}.

\begin{proposition}\label{prop:FourierMultiplierBound}
  Let $\CalQ = (Q_i)_{i \in I}$ be a decomposition cover of an open
  set $\emptyset \neq \CalO \subset \RHat^d$, and let $(\varphi_i)_{i \in I}$ be a BAPU
  subordinate to $\CalQ$.
  A continuous function $h \in C(\CalO)$ is called \emph{tame} if
  \begin{equation}
    C_h := \sup_{i \in I} \|\Fourier^{-1} (\varphi_i \cdot h)\|_{L^1} < \infty .
    \label{eq:TameCondition}
  \end{equation}

  If $h$ is tame and if $f \in \DecompSp(\CalQ,L^p,\ell_w^q)$ for certain $p,q \in [1,\infty]$
  and a $\CalQ$-moderate weight $w$, then the series
  \begin{equation}
    \Phi_h \, f
    := \sum_{i \in I}
         \Fourier^{-1} [(\varphi_i^\ast h) \odot (\varphi_i \widehat{f} \, )]
    \label{eq:SpecialFourierMultiplierDefinition}
  \end{equation}
  converges unconditionally in $Z'(\CalO)$.
  Furthermore, the operator $\Phi_h$ satisfies the following properties:
  \begin{enumerate}[leftmargin=0.9cm]
    \item[(i)] $\Phi_h \!:\! \DecompSp(\CalQ,L^p,\ell_w^q) \!\to\! \DecompSp(\CalQ,L^p,\ell_w^q)$
          is bounded, with
          \(
            \|\Phi_h\|_{\DecompSp(\CalQ,L^p,\ell_w^q) \to \DecompSp(\CalQ,L^p,\ell_w^q)}
            \leq N_\CalQ^2 C_{\Phi} C_h
          \)
          for arbitrary $p,q \in [1,\infty]$ and any $\CalQ$-moderate weight $w$.
          \vspace{0.1cm}

    \item[(ii)] If $(f_n)_{n \in \N} \subset Z'(\CalO)$ is $(F,\Phi)$-dominated for some
          $F \in \ell_w^q(I; L^p)$ and if $f_n \to f$ with convergence in
          $Z'(\CalO)$, then also $\Phi_h f_n \to \Phi_h f$
          with convergence in $Z'(\CalO)$.
          In addition, there is $G \in \ell_w^q(I; L^p)$ such that $\Phi_h f_n$
          is $(G,\Phi)$-dominated for all $n \in \N$ and such that
          $\|G\|_{\ell_w^q(I;L^p)} \leq N_\CalQ^2 C_{\Phi} C_h \cdot \|F_\ell\|_{\ell_w^q(I;L^p)}$.

    \item[(iii)] If $f \in \DecompSp(\CalQ,L^p,\ell_w^q)$ and $\widehat{f} \in C_c (\CalO)$,
          then $\Phi_h f = \Fourier^{-1} (h \cdot \widehat{f} \,)$.
          \vspace{0.1cm}

    \item[(iv)] If $g,h \in C(\CalO)$ are tame, then so is $g \cdot h$,
          and we have $\Phi_h \Phi_g = \Phi_{g h}$.
  \end{enumerate}
\end{proposition}

\begin{rem*}
  One can show that if $C_h$ is finite for one BAPU $(\varphi_i)_{i \in I}$,
  then the same holds for any other BAPU.
  Still, the precise value of the constant $C_h$ depends on the choice of the BAPU.
\end{rem*}

\subsection{Estimates for the remainder term \texorpdfstring{$R$}{R}}
\label{sub:RemainderEstimate}
The following proposition provides a general condition under which
$R$ defines a bounded operator on $\DecompSp(\CalQ,L^p,\ell_w^q)$.
Simplified versions of these are derived in Section~\ref{sec:SimplifiedEstimates}.

\begin{proposition}\label{prop:RemainderAbstractTAlphaEstimate}
  Let $\CalQ = (Q_i)_{i \in I}$ be a decomposition cover
  of an open set $\CalO \subset \RHat^d$ of full measure,
  with associated BAPU $\Phi = (\varphi_i)_{i \in I}$.
  Let $w = (w_i)_{i \in I}$ be $\CalQ$-moderate.
  Suppose the system $\StandardGSI$ satisfies the
  $\alpha$-local integrability condition \eqref{eq:LIC}, with respect to $\CalO^c$.
  Moreover, suppose that, for all $i, \ell \in I$,
  \begin{equation}\label{eq:matrix_entries_N}
    N_{i,\ell}
    := \frac{w_i}{w_{\ell}}
        \sum_{\alpha \in \Lambda \setminus \{0\}}
          \bigg\|
            \Fourier^{-1}
            \bigg(
              \varphi_i (\cdot + \alpha)
              \cdot t_{\alpha}
              \cdot \varphi_{\ell}
            \bigg)
          \bigg\|_{L^1} < \infty
  \end{equation}
  and that the matrix $N = (N_{i,\ell})_{i,\ell \in I} \in \CC^{I \times I}$
  is of Schur-type.
  Then, for all $p, q \in [1,\infty]$, the ``remainder operator $R$'' defined in
  \eqref{eq:RemainderTermDefinition} satisfies
  \[
    \| R f \|_{\DecompSp (\CalQ, L^p, \ell^q_w)}
    \leq \| N \|_{\schur} \,
         \| \Gamma_{\CalQ} \|_{\ell^q_w (I) \to \ell^q_w (I)} \,
         \| f \|_{\StandardDecompSp}
    \qquad \forall \, f \in \DenseSpace(\RR^d) \, .
  \]
\end{proposition}

\begin{proof}
 The assumptions yield, by Schur's test, that the operator
 \[
   \mathbf{N} : \ell^q_w (I) \to \ell_w^q (I),
   \quad (c_\ell)_{\ell \in I} \mapsto \bigg(
                                         \sum_{\ell \in I}
                                           \bigg[
                                             \sum_{\alpha \in \Lambda \setminus \{0\}}
                                               \|
                                                 \varphi_i (\cdot + \alpha)
                                                 \cdot t_{\alpha}
                                                 \cdot \varphi_{\ell}
                                               \|_{\Fourier L^1}
                                           \bigg] \cdot c_{\ell}
                                       \bigg)_{i \in I},
 \]
 is bounded, with $\| \mathbf{N} \|_{\ell_w^q (I) \to \ell_w^q(I)} \leq \| N \|_{\schur}$.

 Let $f \in \DenseSpace (\RR^d)$ be arbitrary.
 For any $\ell \in I$, define $c_\ell := \|\varphi_\ell^\ast \cdot \widehat{f}\|_{\Fourier L^p}$
 and $\theta_\ell := \|\varphi_\ell \cdot \widehat{f} \|_{\Fourier L^p}$,
 where $\varphi_\ell^\ast := \sum_{i \in \ell^\ast} \varphi_i$.
 Let $c = (c_i)_{i \in I}$ and $\theta = (\theta_i)_{i \in I}$.
 Then $0 \leq c_\ell \leq \sum_{i \in \ell^\ast} \theta_i = (\Gamma_\CalQ \, \theta )_\ell$,
 and hence
 \(
   \|c\|_{\ell_w^q}
   \leq \|\Gamma_\CalQ\|_{\ell_w^q \to \ell_w^q} \cdot \|\theta\|_{\ell_w^q}
   =    \|\Gamma_\CalQ\|_{\ell_w^q \to \ell_w^q} \cdot \|f\|_{\DecompSp(\CalQ,L^p,\ell_w^q)}
   <    \infty
 \).

 Since $f \in \DenseSpace (\RR^d)$, we have $\widehat{f} \in C_c^\infty (\CalO)$, and hence
 \(
   \widehat{f}
   = \sum_{\ell \in I}
       \varphi_\ell \cdot \widehat{f}
   = \sum_{\ell \in I}
       \varphi_\ell \, \varphi_\ell^\ast \, \widehat{f}
 \),
 where only finitely many terms of the series do not vanish.
 Therefore, by the unconditional convergence of the series defining $R f$
 (see Corollary~\ref{cor:FrameOperatorOnDenseSpace}), we see
 \[
   \varphi_i \cdot \widehat{Rf}
   = \varphi_i
     \cdot \sum_{\alpha \in \Lambda \setminus \{0\}}
             \Translation{\alpha} (t_\alpha \cdot \widehat{f} \,)
   = \sum_{\ell \in I}
       \sum_{\alpha \in \Lambda \setminus \{0\}}
         \varphi_i
         \cdot \Translation{\alpha}
               (
                t_\alpha
                \cdot \varphi_\ell
                \cdot \varphi_\ell^\ast
                \cdot \widehat{f}\,
               ) \, .
 \]
 Hence, for all $i \in I$,
 \begin{align*}
   \|
     \varphi_i \cdot \widehat{R f}
   \|_{\Fourier L^p}
   & \leq \sum_{\ell \in I}
            \sum_{\alpha \in \Lambda \setminus \{0\}}
              \|
                \varphi_i
                \cdot \Translation{\alpha}
                      (
                       t_\alpha
                       \cdot \varphi_\ell
                       \cdot \varphi_\ell^\ast
                       \cdot \widehat{f}\,
                      )
              \|_{\Fourier L^p}   \\
   &\leq \sum_{\ell \in I}
            \sum_{\alpha \in \Lambda \setminus \{0\}}
              \|
                (\Translation{-\alpha} \varphi_i)
                \cdot t_\alpha
                \cdot \varphi_\ell
              \|_{\Fourier L^1}
              \,
              \|
                \varphi_\ell^\ast \cdot \widehat{f}\,
              \|_{\Fourier L^p}
     =    (\mathbf{N} \, c)_i \, ,
 \end{align*}
 and thus
 \begin{align*}
   \| R f \|_{\DecompSp(\CalQ,L^p,\ell_w^q)}
   = \big\|
       \big(
         \| \varphi_i \cdot \widehat{R f} \|_{\Fourier L^p}
       \big)_{i \in I}
     \big\|_{\ell_w^q}
   \leq \| \mathbf{N} \, c \|_{\ell_w^q}
   \leq \| N \|_{\schur}
         \|\Gamma_\CalQ\|_{\ell_w^q \to \ell_w^q}
         \|f\|_{\StandardDecompSp},
 \end{align*}
 as claimed.
\end{proof}

\begin{corollary} \label{cor:RemainderTermR0}
  Assume that the hypotheses of Proposition~\ref{prop:RemainderAbstractTAlphaEstimate} are satisfied.
  Furthermore, assume that the function $t_0$ defined in \eqref{eq:TAlphaDefinition}
  is continuous on $\CalO$ and tame (see Proposition~\ref{prop:FourierMultiplierBound}),
  so that the operator $\Phi_{t_0} : \StandardDecompSp \to \StandardDecompSp$
  is well-defined and bounded.
  Finally, assume that $\StandardGSI$ is $(w,v,\Phi)$-adapted for some weight $v = (v_j)_{j \in J}$.

  Define $T_0 := \Phi_{t_0}$.
  Then the frame operator
  $S : \DecompSp(\CalQ,L^p,\ell_w^q) \to \DecompSp(\CalQ,L^p,\ell_w^q)$
  is well-defined and bounded and satisfies
  $S = T_0 + R_0$ with a bounded linear operator
  $R_0 : \DecompSp(\CalQ,L^p,\ell_w^q) \to \DecompSp(\CalQ,L^p,\ell_w^q)$ satisfying
  \[
    \|R_0\|_{\DecompSp(\CalQ,L^p,\ell_w^q) \to \DecompSp(\CalQ,L^p,\ell_w^q)}
    \leq C_{p,q} \| N \|_{\schur} \,
         \| \Gamma_{\CalQ} \|_{\ell^q_w (I) \to \ell^q_w (I)},
  \]
  where $N \in \CC^{I \times I}$ is as in \eqref{eq:matrix_entries_N},
  and $C_{p,q} := 1$ if $\max \{p,q\} < \infty$
  and $C_{p,q} := C_\Phi \, \|\Gamma_{\CalQ}\|_{\ell_w^q \to \ell_w^q}^2$ otherwise.
\end{corollary}

\begin{proof}
 Corollary~\ref{cor:FrameOperatorSpecialContinuity} shows that the frame operator
 $S : \StandardDecompSp \to \StandardDecompSp$ is well-defined and bounded,
 and hence so is $R_0 := S - T_0$.
 Note for $f \in \DenseSpace(\RR^d)$ that $T_0 f = \Fourier^{-1}(t_0 \cdot \widehat{f})$
 by Proposition~\ref{prop:FourierMultiplierBound}(iii).
 Therefore, Corollary~\ref{cor:FrameOperatorOnDenseSpace} shows
 for $f \in \DenseSpace(\RR^d)$ that $R_0 f = R f$
 with $R f$ as in Equation~\eqref{eq:RemainderTermDefinition}.
 Thus, if $\max \{p,q\} < \infty$, the density of $\DenseSpace(\RR^d)$
 in $\DecompSp(\CalQ,L^p,\ell_w^q)$ (Proposition~\ref{prop:Density}),
 combined with Proposition~\ref{prop:RemainderAbstractTAlphaEstimate}, shows the claim.

 Now, suppose that $\max \{p,q\} = \infty$, and let $f \in \StandardDecompSp$ be arbitrary.
 Then, Proposition~\ref{prop:Density} yields a sequence
 $(g_n)_{n \in \NN} \subset \DenseSpace(\RR^d)$ and some $F \in \ell_w^q (I; L^p)$
 such that $g_n \to f$ with convergence in $Z'(\CalO)$,
 and such that each $g_n$ is $(F,\Phi)$-dominated, where
 $\|F\|_{\ell_w^q(I; L^p)} \leq C_{p,q} \cdot \|f\|_{\StandardDecompSp}$
 with $C_{p,q}$ as in the statement of the current corollary.
 By Proposition~\ref{prop:FourierMultiplierBound}(ii), we get
 ${T_0 g_n  \to T_0 f}$ with convergence in $Z'(\CalO)$.
 In addition, Corollary~\ref{cor:FrameOperatorSpecialContinuity} shows that $S \, g_n \to S \, f$
 in $Z'(\CalO)$.
 Therefore, $R \, g_n = R_0 \, g_n = (S - T_0) g_n \to (S - T_0) f = R_0  f$,
 while Proposition~\ref{prop:RemainderAbstractTAlphaEstimate} shows
 \begin{align*}
   \|R \, g_n\|_{\StandardDecompSp}
 &  \leq \|N\|_{\schur} \|\Gamma_{\CalQ}\|_{\ell_w^q \to \ell_w^q} \, \|g_n\|_{\StandardDecompSp}\\
&   \leq C_{p,q} \|N\|_{\mathrm{Schur}}
         \|\Gamma_{\CalQ}\|_{\ell_w^q \to \ell_w^q}
           \|f\|_{\StandardDecompSp} \, .
 \end{align*}
  Lemma~\ref{lem:DecompositionSpaceFatouProperty} yields
 \(
   \|R_0 f\|_{\StandardDecompSp}
   \leq C_{p,q} \|N\|_{\mathrm{Schur}}
         \|\Gamma_{\CalQ}\|_{\ell_w^q \to \ell_w^q}
           \|f\|_{\StandardDecompSp}
 \).
\end{proof}

In many cases, instead of verifying that the matrix $N$ defined in
Equation~\eqref{eq:matrix_entries_N} is of Schur-type, it is easier to consider the
matrix $\widetilde{N}$ defined next.

\begin{corollary}\label{cor:RemainderAbstractEstimate}
  Let $\CalQ = (Q_i)_{i \in I}$ be a decomposition cover of an open set $\CalO \subset \RHat^d$
  of full measure with BAPU $\Phi = (\varphi_i)_{i \in I}$, and let $w = (w_i)_{i \in I}$
  be $\CalQ$-moderate.
  Let $\StandardGSI$ be a generalized shift-invariant system.
  Suppose that the matrix $\widetilde{N} = (\widetilde{N}_{i,\ell})_{i,\ell \in I}$ given by
  \begin{align}\label{eq:matrix_entries_NTilde}
    \widetilde{N}_{i,\ell}
    :=  \max \bigg\{ 1, \frac{w_i}{w_{\ell}} \bigg\}
       \sum_{j \in J}
         \frac{1}{|\det C_j|}
         \sum_{\alpha \in C_j^{-t} \ZZ^d \setminus \{0\}}
           \bigg\|
             \Fourier^{-1}
             \bigg(
               \varphi_i (\cdot - \alpha)
               \cdot \overline{\widehat{g_j}}
               \cdot \widehat{g_j}(\cdot - \alpha)
               \cdot \varphi_{\ell}
             \bigg)
           \bigg\|_{L^1}
  \end{align}
  is of Schur-type.
  Then $\StandardGSI$ satisfies the $\alpha$-local integrability condition relative to $\CalO^c$,
  and $\|N\|_{\schur} \leq \|\widetilde{N} \|_{\schur}$,
  where $N$ is as defined in Equation~\eqref{eq:matrix_entries_N}.
\end{corollary}

\begin{proof}
 By assumption, $\| \widetilde{N} \|_{\schur} < \infty$.
 We first show that
 \begin{align} \label{eq:daubechies_bessel_criterion}
   C
   := \esssup_{\xi \in \CalO}
        \sum_{j \in J}
          \frac{1}{|\det C_j|}
          \sum_{\alpha \in C_j^{-t} \ZZ^d}
            |\widehat{g_j} (\xi) \widehat{g_j} (\xi + \alpha) |
   < \infty.
 \end{align}
 To show this, first note that since $\CalO \subset \RHat^d$ is of full measure, so is
 \[
   \CalO_0
   := \big\{
        \xi \in \RHat^d
        \quad\colon\quad
            \xi + \alpha \in \CalO, \quad \forall \, j \in J \, , \;
          \forall \, \alpha \in C_j^{-t} \ZZ^d
      \big\},
 \]
 since $\CalO_0^c = \bigcup_{j \in J} \bigcup_{\alpha \in C_j^{-t} \ZZ^d} (\CalO^c - \alpha)$
 is a countable union of null-sets.
 If $\xi \in \CalO_0$ and $j \in J$, $\alpha \in C_j^{-t} \ZZ^d$ are arbitrary,
 then $\xi + \alpha \in \CalO$ and hence $\sum_{i \in I} \varphi_i(\xi+\alpha) = 1$,
 whence $1 \leq \sum_{i \in I} |\varphi_i(\xi+\alpha)|$.
 Now, let $\xi \in \CalO_0 \subset \CalO$ be arbitrary and choose
 $i_0 \in I$ such that $\xi \in Q_{i_0}$.
 Then $\sum_{\ell \in i_0^*} \varphi_{\ell} (\xi) = 1$.
 Thus, using the estimate $\|f\|_{\sup} \leq \|\Fourier^{-1} f\|_{L^1}$, we see that
 \begin{align*}
  & \sum_{j \in J}
     \frac{1}{|\det C_j|}
     \sum_{\alpha \in C_j^{-t} \ZZ^d \setminus \{0\}} \!\!\!\!
       |\widehat{g_j} (\xi) \widehat{g_j} (\xi + \alpha) | \\
   &\quad \quad \quad \leq \sum_{i \in I, \ell \in i_0^\ast}
            \sum_{j \in J}
              \frac{1}{|\det C_j|}
              \sum_{\alpha \in C_j^{-t} \ZZ^d \setminus \{0\}} \!\!\!\!
                |
                 \widehat{g_j} (\xi)
                 \varphi_i (\xi + \alpha)
                 \widehat{g_j} (\xi + \alpha)
                 \varphi_{\ell} (\xi)
                |\\
   &\quad \quad \quad \leq \sum_{\ell \in i_0^\ast, i \in I}
           \widetilde{N}_{i,\ell}
    \leq N_{\CalQ} \cdot \| \widetilde{N} \|_{\schur}
    <    \infty.
 \end{align*}
 In combination with our standing assumption \eqref{eq:GSI_assumption},
 this proves \eqref{eq:daubechies_bessel_criterion}.

 \smallskip{}

 Now, the monotone convergence theorem and \eqref{eq:daubechies_bessel_criterion} show
 for arbitrary $f \in \mathcal{B}_{\CalO} (\RR^d)$ that
 \[
   \sum_{j \in J}
     \frac{1}{|\det C_j|}
     \sum_{\alpha \in C_j^{-t} \ZZ^d}
       \int_{\RHat^d}
         |
          \widehat{f}(\xi) \,
          \widehat{f}(\xi+\alpha) \,
          \widehat{g_j}(\xi) \,
          \widehat{g_j} (\xi + \alpha)
         |
       \, d \xi
     \leq C \, \|\widehat{f}\|_{L^\infty}
          \cdot \int_{\RHat^d}
                  |\widehat{f}(\xi)|
                \, d \xi
     < \infty,
 \]
 since $\widehat{f} \in L^\infty(\RHat^d)$ and $\supp \widehat{f} \subset \CalO$ is compact.
 This shows that $\StandardGSI$ satisfies the $\alpha$-LIC.

 Finally, recall that
 \(
   t_\alpha (\xi)
    = \sum_{j \in \kappa(\alpha)}
        |\det C_j|^{-1} \, \overline{\widehat{g_j}(\xi)} \, \widehat{g_j}(\xi + \alpha),
 \)
 where $\kappa(\alpha) = \{ j \in J \colon \alpha \in C_j^{-t} \ZZ^d \}$.
 Therefore,  the matrix entries $N_{i,\ell}$ defined in
 \eqref{eq:matrix_entries_N} satisfy
 \[
   N_{i,\ell}
   \leq  \max \bigg\{ 1, \frac{w_i}{w_{\ell}} \bigg\}
         \sum_{\alpha \in \Lambda \setminus \{0\}}
                \sum_{j \in \kappa(\alpha)}
                  |\det C_j|^{-1} \,
                  \big\|
                    \Fourier^{-1}
                    \big(
                      \varphi_i (\cdot + \alpha)
                      \cdot \overline{\widehat{g_j}}
                      \cdot \widehat{g_j} (\cdot + \alpha)
                      \cdot \varphi_\ell
                    \big)
                  \big\|_{L^1}
   = \widetilde{N}_{i,\ell}.
 \]
 Thus $\|N\|_{\mathrm{Schur}} \leq \|\widetilde{N} \|_{\mathrm{Schur}}$, as claimed.
\end{proof}

\subsection{Invertibility in the case \texorpdfstring{$(p,q) = (2,2)$}{(p,q) = (2,2)}}
\label{sub:WeightedL2CaseSimplified}

In this subsection, we focus on the special case $(p,q) = (2,2)$,
where the following identification holds; see \cite[Lemma~6.10]{DecompositionEmbedding}.

\begin{lemma}\label{lem:D-L^2-equivalence}
Let $\CalQ = (Q_i)_{i \in I}$ be a decomposition cover of an open set
$\emptyset \neq \CalO \subset \RHat^d$,
and let ${w = (w_i)_{i \in I}}$ be a $\CalQ$-moderate weight.
Then there is a measurable weight $v : \CalO \to (0,\infty)$ with
$v (\xi) \asymp w_i$ for all $\xi \in Q_i$ and $i \in I$.
Furthermore,
\(
  \DecompSp(\CalQ, L^2, \ell_w^2) = \Fourier^{-1}( L^2_v (\CalO))
\)
with equivalent norms, where the norm
\(
 \|f\|_{\Fourier^{-1} (L_v^2 (\CalO))} := \|\widehat{f}\|_{L_v^2 (\CalO)}
\)
is used on
\(
  \Fourier^{-1} (L_v^2 (\CalO))
  = \big\{ f \in Z'(\CalO) \,:\, \widehat{f} \in L_v^2 (\CalO) \big\}
\).
\end{lemma}

We will also make use of the following two lemmata.

\begin{lemma}\label{lem:estimateT0-L2}
Let $\emptyset \neq \CalO \subset \RHat^d$ be an open set,
let $v : \CalO \to (0,\infty)$ be a weight function,
and let $t_0$ be as in Equation~\eqref{eq:TAlphaDefinition}.
Then the Fourier multipliers
\(
  T_0 : \Fourier^{-1} (L^2_v (\CalO)) \to \Fourier^{-1} (L_v^2 (\CalO)),
        f \mapsto \Fourier^{-1} (t_0 \, \widehat{f} \,)
\)
and
\[
  T_0^{-1} : \Fourier^{-1} (L^2_v (\CalO)) \to \Fourier^{-1} (L^2_v (\CalO)),
  \quad f \mapsto \Fourier^{-1} (t_0^{-1} \cdot \widehat{f} \, )
\]
are well-defined and bounded, with $\|T_0^{-1} \|_{\op} \leq A^{-1}$
and $\|T_0\|_{\op} \leq B$, where $A,B > 0$ are as in \eqref{eq:GSI_assumption}.
\end{lemma}

\begin{proof}
  If $f \in \Fourier^{-1}(L_v^2 (\CalO))$, then
  \[
    \|T_0^{-1} f\|_{\Fourier^{-1} (L_v^2 (\CalO))}
     = \| t_0^{-1} \cdot \widehat{f} \, \|_{L_v^2 (\CalO)}
     \leq \|t_0^{-1}\|_{L^\infty (\CalO)}
        \cdot \|f\|_{\Fourier^{-1} (L_v^2 (\CalO))}
    .
  \]
  The argument for $T_0$ is similar.
\end{proof}

\begin{lemma}\label{lem:estimateR-L2}
  Let $\CalO \subset \RHat^d$ be an open set of full measure and let
  $v : \RHat^d \to (0,\infty)$ be $v_0$-moderate for some symmetric weight
  $v_0 : \RHat^d \to (0,\infty)$; that is,
  $v(\xi+\eta) \leq C_v \cdot v (\xi) \cdot v_0(\eta)$
  for all $\xi,\eta \in \RHat^d$ and some $C_v > 0$.
  Then the operator $R$ defined in Equation~\eqref{eq:RemainderTermDefinition} satisfies
  \begin{equation}
    \| R \|_{\Fourier^{-1}(L^2_v (\CalO)) \to \Fourier^{-1} (L_v^2 (\CalO))}
    \leq C_v \cdot \esssup_{\xi \in \CalO}
                     \sum_{\alpha \in \Lambda \setminus \{0\}}
                       |t_{\alpha} (\xi) | \cdot v_0 (\alpha) .
    \label{eq:RemainderEstimateL2}
  \end{equation}
\end{lemma}

\begin{proof}
Since $\CalO$ is of full measure, we have
$\Fourier^{-1} (L_v^2 (\CalO)) = \Fourier^{-1} (L_v^2 (\RHat^d))$,
up to canonical identifications.
Let $g \in L^2 (\RHat^d)$ and $f \in \Fourier^{-1}(L_v^2 (\CalO))$ be such that
$\|g\|_{L^2} \leq 1$ and $\|f\|_{\Fourier^{-1}(L_v^2 (\CalO))} \leq 1$.
Using the estimates $v(\xi) \leq C_v \cdot v(\xi - \alpha) \cdot v_0 (\alpha)$
and  $|a b| \leq \frac{1}{2} \big( |a|^2 + |b|^2 \big)$
and the identity $t_\alpha (\xi - \alpha) = \overline{t_{-\alpha} (\xi)}$,
it follows that
\begin{align*}
  & \int_{\RHat^d}
      |g(\xi)|
      \cdot v(\xi)
      \cdot \sum_{\alpha \in \Lambda \setminus \{0\}}
              \big|
                t_\alpha (\xi - \alpha) \,
                \widehat{f} (\xi - \alpha)
              \big|
    \, d\xi \\
  & \leq C_v \cdot \sum_{\alpha \in \Lambda \setminus \{0\}}
                     v_0 (\alpha)
                     \int_{\RHat^d}
                       \Big( |t_{-\alpha} (\xi)|^{1/2} \cdot |g(\xi)| \Big)
                       \cdot \Big(
                               |t_\alpha (\xi - \alpha)|^{1/2}
                               \cdot | (v \widehat{f} \, ) (\xi - \alpha)|
                             \Big)
                     \, d \xi \\
  & \leq \frac{C_v}{2}
         \cdot \sum_{\alpha \in \Lambda \setminus \{0\}}
                 v_0 (\alpha)
                 \int_{\RHat^d}
                   |t_{-\alpha} (\xi)| \cdot |g(\xi)|^2
                   + |t_\alpha (\xi - \alpha)|
                     \cdot |(v \widehat{f} \, ) (\xi - \alpha)|^2
                 \, d \xi \\
  & =
      \frac{C_v}{2} \cdot \bigg(
                            \int_{\RHat^d}
                              \Big(
                                \sum_{\beta \in \Lambda \setminus \{0\}}
                                  v_0 (-\beta) \, |t_{\beta} (\xi)|
                              \Big)
                              \cdot |g(\xi)|^2
                            \, d \xi \\
                            & \quad \quad \quad \quad \quad \quad \quad \quad \quad
                            + \int_{\RHat^d}
                                \Big(
                                  \sum_{\alpha \in \Lambda \setminus \{0\}}
                                    v_0 (\alpha) \, |t_\alpha (\eta)|
                                \Big)
                                \cdot |(v \widehat{f} \, ) (\eta)|^2
                              \, d \eta
                          \bigg) \\
  & \leq C_v \cdot \esssup_{\xi \in \CalO}
                   \sum_{\alpha \in \Lambda \setminus \{0\}}
                     v_0 (\alpha) \, |t_\alpha (\xi)| \, .
\end{align*}
Since this holds for all $g \in L^2(\RHat^d)$ with $\|g\|_{L^2} \leq 1$, the series
\[
  \sum_{\alpha \in \Lambda \setminus \{0\}}
    t_\alpha (\xi - \alpha) \, \widehat{f} (\xi - \alpha)
  = \sum_{\alpha \in \Lambda \setminus \{0\}}
      \big[ \Translation{\alpha} (t_\alpha \cdot \widehat{f} \, ) \big](\xi)
  = \big[ \widehat{R f} \big] (\xi)
\]
is almost everywhere absolutely convergent, and
\[
  \|R f\|_{\Fourier^{-1} (L_v^2 (\CalO))}
  \leq \Big\|
         v
         \cdot \sum_{\alpha \in \Lambda \setminus \{0\}}
                 \big|
                   \Translation{\alpha} (t_\alpha \, \widehat{f} \, )
                 \big|
       \Big\|_{L^2}
  \leq C_v \cdot \esssup_{\xi \in \CalO}
                   \sum_{\alpha \in \Lambda \setminus \{0\}}
                     v_0 (\alpha) \, |t_\alpha (\xi)| \, ,
\]
for all $f \in \Fourier^{-1} (L_v^2 (\CalO))$ with
$\|f\|_{\Fourier^{-1}(L_v^2 (\CalO))} \leq 1$.
This proves the claim.
\end{proof}

Using the previous lemmata, the following result follows easily.
See \cite[Theorem 3.3]{Lemvig2018Criteria} for a similar result in $L^2$.

\begin{proposition} \label{prop:sufficient_l2}
Let $\CalQ = (Q_i)_{i \in I}$ be a decomposition cover of an open set
$\CalO \subset \RHat^d$ of full measure, and let $w = (w_i)_{i \in I}$ be $\CalQ$-moderate.
Suppose $\StandardGSI$ satisfies the $\alpha$-local integrability condition \eqref{eq:LIC}
relative to $\CalO^c$.
Finally, assume that
\begin{align} \label{eq:S-invertible-L2}
 C_v \cdot
 \esssup_{\xi \in \CalO}
   \sum_{\alpha \in \Lambda \setminus \{0\}}
     |t_{\alpha} (\xi) | \cdot v_0 (\alpha)
 < A \, ,
\end{align}
where $A > 0$ is as in \eqref{eq:GSI_assumption},
where $v : \RHat^d \to (0,\infty)$ is a measurable weight that satisfies
$v(\xi) \asymp w_i$ for all $\xi \in Q_i$ and $i \in I$, and where
$v_0 : \RHat^d \to (0,\infty)$ is assumed to be a symmetric weight
satisfying $v(\xi + \eta) \leq C_v \cdot v(\xi) \cdot v_0 (\eta)$
for all $\xi, \eta \in \RHat^d$.

Then the frame operator $S : \DenseSpace(\RR^d) \to L^2(\R^d)$ associated to $\StandardGSI$
uniquely extends to a bounded linear operator
$S_0 : \CalD(\CalQ,L^2,\ell_w^2) \to \CalD(\CalQ,L^2,\ell_w^2)$.
This extended operator is boundedly invertible.
\end{proposition}

\begin{proof}
  Lemmas~\ref{lem:estimateT0-L2} and \ref{lem:estimateR-L2} show, respectively, that the operators
  $T_0$ and $R$ defined in these lemmas yield bounded operators on $\Fourier^{-1}(L_v^2(\CalO))$,
  so that $S_0 := T_0 + R : \Fourier^{-1} (L_v^2(\CalO)) \to \Fourier^{-1} (L_v^2(\CalO))$
  is well-defined and bounded.
  As seen in Proposition~\ref{prop:GSI_identity}, we have
  $S_0 f = S f$ for all ${f \in \DenseSpace(\RR^d) \subset \CalB_{\CalO}(\RR^d)}$.
  Furthermore, $\DenseSpace(\RR^d) \subset \CalD(\CalQ,L^2,\ell_w^2) = \Fourier^{-1}(L_v^2(\CalO))$
  is dense (see Proposition~\ref{prop:Density} and Lemma~\ref{lem:D-L^2-equivalence});
  therefore, $S_0$ is the unique bounded extension of $S$.

  Finally, conditions \eqref{eq:GSI_assumption} and \eqref{eq:S-invertible-L2} together with
  Lemma~\ref{lem:estimateT0-L2} and Lemma~\ref{lem:estimateR-L2} yield that
  \[
    \|
      T_0^{-1}
    \|_{\Fourier^{-1} (L^2_v (\CalO)) \to \Fourier^{-1} (L^2_v (\CalO))}
    \cdot \|
            R
          \|_{\Fourier^{-1} (L^2_v (\CalO)) \to \Fourier^{-1} (L^2_v (\CalO))}
    < 1.
  \]
  Hence, $S_0 = T_0 + R$ is boundedly invertible on $\Fourier^{-1}(L_v^2(\CalO))$
  by Lemma~\ref{lem:FrameOperatorInvertibility}.
  Using the norm equivalence
  \(
    \| \cdot \|_{\Fourier^{-1} (L^2_v (\CalO))}
    \asymp \| \cdot \|_{\mathcal{D} (\CalQ, L^2, \ell^2_w)}
  \)
  provided by Lemma~\ref{lem:D-L^2-equivalence}, it follows therefore that also
  $S_0 : \mathcal{D}(\CalQ, L^2, \ell^2_w) \to \mathcal{D}(\CalQ, L^2, \ell^2_w)$
  is boundedly invertible.
\end{proof}

\begin{remark}
  The formulation of Proposition~\ref{prop:sufficient_l2} is rather technical, because,
  under those assumptions, the formula defining the frame operator might not make sense
  for $f \in \CalD(\CalQ,L^2,\ell_w^2)$.
  Indeed, the hypothesis are satisfied for every tight frame,
  even if $g_j \notin \CalD(\CalQ,L^2,\ell_w^2)$.
  If, in addition, $\StandardGSI$ is assumed to be $(w,v,\Phi)$-adapted for some weight $v$,
  then Proposition~\ref{prop:AnalysisSynthesisOperatorGeneral} applies and
  we can conclude unambiguously that $S : \CalD(\CalQ,L^2,\ell_w^2) \to \CalD(\CalQ,L^2,\ell_w^2)$
  is well-defined, bounded and boundedly invertible on $\CalD(\CalQ,L^2,\ell_w^2)$.
\end{remark}

\begin{remark}
  If $\StandardGSI$ is a tight frame for $L^2 (\RR^d)$ with lower frame bound $A > 0$,
  which furthermore satisfies the $\alpha$-local integrability condition, then
  the multipliers $t_{\alpha} \in L^{\infty} (\RHat^d)$ satisfy
  $t_{\alpha} (\xi) = A \, \delta_{\alpha, 0}$ for a.e.~$\xi \in \RHat^d$
  and all $\alpha \in \Lambda$, cf.~\cite[Theorem 3.4]{JakobsenReproducing2014}.
  The condition \eqref{eq:S-invertible-L2} is then obviously satisfied.
  The placement of the absolute value sign outside of the series defining
  the multipliers $t_{\alpha}$ allows for cancellations,
  which can be very important \cite{Lemvig2018Criteria}.
\end{remark}

\section{Concrete estimates for affinely generated covers}
\label{sec:SimplifiedEstimates}

In this section, we simplify the results of Section~\ref{sec:FrameOperatorInvertible} for the case
that the decomposition cover $\CalQ$ is affinely generated.
The results obtained here will be further simplified
in Section~\ref{sec:Structured_Compatible}.

In the sequel, we will repeatedly use \emph{$\CalQ$-localized} versions of
the generating functions $g_j$ of the system $\StandardGSI$.
Precisely, given a family $(g_j)_{j \in J}$ of generating functions
${g_j \in L^1 (\RR^d) \cap L^2 (\RR^d)}$ and a family $(S_i)_{i \in I}$
of invertible affine-linear maps $S_i = A_i (\cdot) + b_i$, we let
\begin{equation}
  g_{i,j}^{\localized}
  := |\det A_i|^{-1}
     \cdot \big( \Modulation{-b_i} g_j \big) \circ A_i^{-t}
   = \Fourier^{-1} (\widehat{g_j} \circ S_i)
  \quad \text{for} \quad (i,j) \in I \times J \, ,
  \label{eq:NormalizedVersionDoubleIndex}
\end{equation}
so that $\Fourier g^{\localized}_{i,j} = \widehat{g_j} \circ S_i$.

\subsection{Boundedness of the frame operator}

As a first step, we provide a sufficient condition for a system to be adapted
(see Definition~\ref{def:Adaptedness}).
The proof makes use of the following self-improving property of amalgam spaces,
which is taken from \cite[Theorem~2.17]{StructuredBanachFrames}.

\begin{lemma}\label{lem:BandlimitedWienerAmalgamEstimate}
  Let $f \in \Schwartz'(\RR^d)$ with
  $\supp \widehat{f} \subset A [-R, R]^d + \xi_0$
  for some $A \in \GL(d, \RR)$, $\xi_0 \in \RHat^d$, and $R > 0$.
  Then there exists a constant $C = C(d) > 0$ which only depends on $d \in \N$ such that
  \[
    \|f\|_{W_{A^{-t} [-1,1]^d} (L^{\infty}, L^1)}
    \leq C \cdot (1 + R)^d \cdot \|f\|_{L^1} \, .
  \]
\end{lemma}

\begin{proposition}\label{prop:AmalgamMatrixEntriesEstimate}
  Let $\CalQ = \big( A_i (Q_i') + b_i \big)_{i \in I}$ be an affinely generated cover of
  $\CalO \subset \RHat^d$, and let $\Phi = (\varphi_i)_{i \in I}$ be a regular partition of unity
  subordinate to $\CalQ$.
  Let $w = (w_i)_{i \in I}$ be $\CalQ$-moderate, and let $v = (v_j)_{j \in J}$ be a weight.
  Suppose that the system $\StandardGSI$ satisfies, for $(i,j) \in I \times J$,
  \begin{align*}
    G_{i,j}
    :=  \max \bigg\{ \frac{w_i}{v_j}, \; \frac{v_j}{w_i} \bigg\} \,
        \frac{ (1 + \|C_j^t A_i\|)^d}{|\det C_j|^{1/2}}
        \int_{Q_i'}
          \max_{|\theta| \leq d+1}
            \big| \partial^{\theta} [\Fourier g_{i,j}^{\localized}] (\xi) \big|
        \; d\xi
    < \infty
  \end{align*}
  and that $G = (G_{i,j})_{i \in I, j \in J} \in \CC^{I \times J}$
  is of Schur-type.
  Then $\StandardGSI$ is $(w,v,\Phi)$-adapted.
  Consequently, the frame operator $S : \StandardDecompSp \to \StandardDecompSp$
  is well-defined and bounded.
\end{proposition}

\begin{proof}
  We will estimate $\| ( \widecheck{\varphi_i} \ast g_j ) \circ C_j \|_{W(C_0, \ell^1)}$
  for $(i,j) \in I \times J$.
  Choose $r > 1$ such that $\overline{Q_i '} \subset [-r, r]^d$ for all $i \in I$.
  The norm equivalence
  $\|\mybullet\|_{W(C_0,\ell^1)} \asymp \|\mybullet\|_{W_{[-1,1]^d} (C_0,L^1)}$
  yields an absolute constant $K_1 = K_1 (d) > 0$ satisfying
  \begin{align*}
    \| ( \widecheck{\varphi_i} \ast g_j ) \circ C_j \|_{W(C_0, \ell^1)}
    &\leq K_1 \cdot \|
                     ( \widecheck{\varphi_i} \ast g_j ) \circ C_j
                   \|_{W_{[-1,1]^d}(C_0, L^1)} \\
    &= K_1 \cdot |\det C_j|^{-1}
          \cdot \|
                  \widecheck{\varphi_i} \ast g_j
                \|_{W_{C_j ([-1,1]^d)}(C_0, L^1)} \,
  \end{align*}
  for $i \in I$ and $j \in J$.
  Here, we used Equation~\eqref{eq:WienerAmalgamLinearTransformation} in the last step.
  Define $P_{i,j} := r \cdot \|C_j^t  A_i \|_{\ell^\infty \to \ell^\infty}$.
  Since $\supp \varphi_i \subset A_i(\overline{Q_i '}) + b_i$, it follows that
  \[
    \supp \Fourier (\widecheck{\varphi_i} \ast g_j)
    \subset A_i \, [-r, r]^d + b_i
    =       C_j^{-t} \big( C_j^t \cdot A_i [-r, r]^d \big)  + b_i
    \subset C_j^{-t} [-P_{i,j},P_{i,j}]^d + b_i \, .
  \]
  Therefore, Lemma~\ref{lem:BandlimitedWienerAmalgamEstimate} yields
  a constant $K_2 = K_2 (d) > 0$ such that
  \begin{align} \label{eq:amalgam_L1_estimate}
    \| ( \widecheck{\varphi_i} \ast g_j ) \circ C_j \|_{W(C_0, \ell^1)}
    & \leq K_1 K_2 \cdot (1 + P_{i,j})^d \cdot |\det C_j|^{-1}
                   \cdot \| \Fourier^{-1} (\varphi_i \cdot \widehat{g_j}) \|_{L^1} \, .
  \end{align}
  Next, recalling the notion of the normalized version
  $\varphi_i^{\normalizedBAPU} = \varphi_i \circ S_i$ of $\varphi_i$
  (Definition~\ref{def:RegularBAPUNormalizedVersion}), we see
  \begin{align*}
    \big\|
      \Fourier^{-1} \big( \varphi_i \cdot \widehat{g_j} \big)
    \big\|_{L^1}
    = \big\|
        \Fourier^{-1}
        \big(
         ( \varphi_i \circ S_i ) \cdot
          (\widehat{g_j}  \circ S_i )
        \big)
      \big\|_{L^1}
    = \big\|
        \Fourier^{-1}
        \big(
          \varphi_i^{\normalizedBAPU} \cdot
           \Fourier g^{\localized}_{i,j}
        \big)
      \big\|_{L^1},
  \end{align*}
  whence Lemma~\ref{lem:FourierDecayViaSmoothness} shows that
  \begin{align*}
   \big\|
     \Fourier^{-1} \big( \varphi_i \cdot \widehat{g_j} \big)
   \big\|_{L^1}
   \leq \frac{d+1}{\pi^d}
        \max_{|\theta| \leq d+1}
          \big\|
            \partial^{\theta}
            \big(
              \varphi_i^{\normalizedBAPU}
              \cdot \Fourier g_{i,j}^{\localized}
            \big)
          \big\|_{L^1} \, .
  \end{align*}
  Now, since $\varphi^{\normalizedBAPU}_i$ vanishes outside of $Q_i '$,
  it follows that
  \(
    | (\partial^{\alpha} \varphi_i^{\normalizedBAPU} )(\xi) |
    \leq K_3 \cdot \mathds{1}_{Q_i '} (\xi)
  \)
  for all $\xi \in \RHat^d$ and any $\alpha \in \NN_0^d$ with $|\alpha| \leq d + 1$, where
  \(
    K_3
    := \max_{|\alpha| \leq d+1}
         \sup_{i \in I}
           \| \partial^\alpha \varphi_i^{\normalizedBAPU} \|_{L^\infty}
  \).
  An application of the Leibniz rule therefore yields
  \begin{align*}
   \big|
     \partial^{\theta}
     \big(
       \varphi_i^{\normalizedBAPU}
       \cdot \Fourier g_{i,j}^{\localized}
     \big)(\xi)
   \big|
   &\leq \sum_{\beta \leq \theta}
           \binom{\theta}{\beta}
           \big| \big( \partial^{\theta - \beta} \varphi_i^{\normalizedBAPU} \big) (\xi) \big|
           \cdot \big|
                   \partial^{\beta} [\Fourier g_{i,j}^{\localized}] (\xi)
                 \big| \\
   & \leq 2^{d+1} K_3 \cdot
          \mathds{1}_{Q_i '} (\xi) \,
          \max_{|\nu| \leq d+1}
                 \big| \big( \partial^{\nu} [\Fourier g_{i,j}^{\localized}] \big) ( \xi ) \big|
  \end{align*}
  for any $\theta \in \NN_0^d$ with $|\theta| \leq d+1$.
  Integrating this last inequality and combining it with \eqref{eq:amalgam_L1_estimate}
  yields
  \[
    \| (\widecheck{\varphi_i} \ast g_j) \circ C_j \|_{W(C_0, \ell^1)}
    \leq K \frac{(1 + \|C_j^t A_i \|)^d}{|\det C_j|}
         \int_{Q_i '}
           \max_{|\theta| \leq d+1}
              \big| \big( \partial^{\theta} [\Fourier g_{i,j}^{\localized}] \big) (\xi) \big|
         \; d\xi
  \]
  for a constant $K = K(\CalQ, d, \Phi) > 0$.
  Therefore, the matrix entries $M_{i,j}$ defined in Equation~\eqref{eq:AnalysisOperatorSchurMatrix}
  satisfy
  \begin{align*}
    0 \leq M_{i,j}
    & = \max \Big\{ \frac{w_i}{v_j}, \frac{v_j}{w_i} \Big\}
        \cdot |\det C_j|^{1/2}
        \cdot \| (\widecheck{\varphi_i} \ast g_j) \circ C_j \|_{W(C_0, \ell^1)} \\
    & \leq K
           \cdot \max \Big\{ \frac{w_i}{v_j}, \frac{v_j}{w_i} \Big\}
           \cdot \frac{(1 + \|C_j^t A_i \|)^d}{|\det C_j|^{1/2}}
           \int_{Q_i '}
             \max_{|\theta| \leq d+1}
                \big| \big( \partial^{\theta} [ \Fourier g_{i,j}^{\localized} ] \big) (\xi) \big|
           \; d\xi
    =    K \cdot G_{i,j} \, .
  \end{align*}
  This implies $\|M\|_{\schur} \leq K \cdot \|G\|_{\schur} < \infty$,
  so that $\StandardGSI$ is $(w,v,\Phi)$-adapted.
\end{proof}

\subsection{The main term}

In this section, we provide a simplified bound for the operator norm of
$T_0^{-1} : \StandardDecompSp \to \StandardDecompSp$.

\begin{proposition}\label{prop:MainTermSimplified}
  Let $\CalQ = (S_i (Q_i '))_{i \in I}$ be an affinely generated cover of an open set
  $\CalO \subset \RHat^d$ of full measure.
  Let $\Phi = (\varphi_i)_{i \in I}$ be a regular partition of unity subordinate to $\CalQ$.
  Suppose the system $\StandardGSI$ satisfies
  \begin{equation}
    M :=
    \sup_{i \in I} \,
        \sum_{j \in J}
          \bigg(\,
            |\det C_j|^{-1}
            \cdot \Big\|\,
                    \max_{|\nu| \leq d+1} \,
                      \big|\,
                        \partial^\nu | \Fourier g_{i,j}^{\localized} |^2
                      \,\big|
                  \,\Big\|_{L^{d+1} (Q_i ')}
          \,\bigg)
    < \infty \, .
    \label{eq:T0InverseMainAssumption}
  \end{equation}
  Then the function $t_0$ defined in Equation~\eqref{eq:TAlphaDefinition}
  is continuous on $\CalO$ and tame, and Equation~\eqref{eq:GSI_assumption}
  holds for \emph{all} $\xi \in \CalO$.
  Furthermore, for all $p,q \in [1,\infty]$ and
  any $\CalQ$-moderate weight $w = (w_i)_{i \in I}$, the operator
  \[
    T_0 := \Phi_{t_0} : \DecompSp (\CalQ, L^p, \ell_w^q) \to \DecompSp (\CalQ, L^p, \ell_w^q)
  \]
  with $\Phi_{t_0}$ as in Proposition~\ref{prop:FourierMultiplierBound}
  is well-defined, bounded, and boundedly invertible, with
  \begin{equation}
    \|T_0^{-1}\|_{\DecompSp(\CalQ,L^p,\ell_w^q)
                  \to \DecompSp(\CalQ,L^p,\ell_w^q)}
    \leq C_d
         \cdot N_\CalQ^2 C_\Phi
         \cdot \Big[\max_{|\alpha| \leq d+1} C_{\CalQ,\Phi,\alpha}\Big]
         \cdot A^{-1}
         \cdot \bigg(\frac{M}{A}\bigg)^{d+1},
    \label{eq:MainTermInverseOperatorNorm}
  \end{equation}
  where $A > 0$ is as in \eqref{eq:GSI_assumption} and
  \begin{equation}
    C_d := \frac{3 \cdot (d+1)^{3/2} \cdot 2^{d+1} }{\pi^d}
           \left(\frac{\frac{0.8}{e} \cdot (d+1)^2}{\ln(2+d)}\right)^{d+1}.
    \label{eq:MainTermSimplifiedSpecialConstant}
  \end{equation}
\end{proposition}

\begin{proof}
  We divide the proof into four steps.
\\\\
  \textbf{Step 1.}
  We show that the series defining $t_0$ converges locally uniformly on $\CalO$,
  that Equation~\eqref{eq:GSI_assumption} holds pointwise on $\CalO$, and that $t_0$ is tame.

  To see this, set $\gamma_j := |\widehat{g_j}|^2 / |\det C_j|$,
  and note $t_0 = \sum_{j \in J} \gamma_j$ and that $\gamma_j \in C^\infty (\RHat^d)$
  thanks to our standing assumptions regarding the $g_j$.
  Now, for arbitrary $i \in I$, recall that $\varphi_i^{\normalizedBAPU} = \varphi_i \circ S_i$
  vanishes outside $Q_i '$, so that the Leibniz rule shows
  \begin{align*}
    \big|
      \partial^\alpha
      \big(
        \varphi_i^{\normalizedBAPU} \cdot (\gamma_j \circ S_i)
      \big) (\xi)
    \big|
    & \leq \sum_{\beta \leq \alpha}
             \binom{\alpha}{\beta} \,
             |\partial^{\alpha - \beta} \varphi_i^{\normalizedBAPU} (\xi)| \,
             |\partial^\beta (\gamma_j \circ S_i) (\xi)| \\
    & \leq c_0
           \cdot |\det C_j|^{-1}
           \cdot \Indicator_{Q_i '} (\xi)
           \cdot \max_{|\nu| \leq d+1}
                   \big|
                     \partial^\nu | \Fourier g_{i,j}^{\localized} |^2 (\xi)
                   \big|
  \end{align*}
  for $c_0 := 2^{d+1} \, \max_{|\nu| \leq d+1} C_{\CalQ,\Phi,\nu}$ and
  arbitrary $\alpha \in \N_0^d$ with $|\alpha| \leq d+1$.

  Therefore, using the notation
  $\mathbb{I} := \{0\} \cup \{ (d+1) \, e_\ell \colon \ell \in \FirstN{d} \}$
  (where $(e_1,\dots,e_d)$ denotes the standard basis of $\R^d$),
  Lemma~\ref{lem:FourierDecayViaSmoothness} shows because of
  \(
    \|\varphi_i \cdot \gamma_j\|_{\Fourier L^1}
    = \|\varphi_i^{\normalizedBAPU} \cdot (\gamma_j \circ S_i)\|_{\Fourier L^1}
  \)
  and $\tfrac{d+1}{\pi^d} \leq 1$ that
  \begin{equation}
    \begin{split}
      \| \varphi_i \cdot \gamma_j \|_{\Fourier L^1}
      & \leq \max_{\alpha \in \mathbb{I}}
               \big\|
                 \partial^\alpha
                 \big(
                   \varphi_i^{\normalizedBAPU} \cdot (\gamma_j \circ S_i)
                 \big)
               \big\|_{L^1}
        \leq c_0 \cdot |\det C_j|^{-1}
             \cdot \Big\|
                     \max_{|\nu| \leq d+1}
                       \big|
                         \partial^\nu |\Fourier g_{i,j}^{\localized}|^2
                       \big|
                   \Big\|_{L^1 (Q_i ')} \\
      & \leq c_0 \, c_1 \cdot |\det C_j|^{-1}
             \cdot \Big\|
                     \max_{|\nu| \leq d+1}
                       \big|
                         \partial^\nu |\Fourier g_{i,j}^{\localized}|^2
                       \big|
                   \Big\|_{L^{d+1} (Q_i ')} ,
    \end{split}
    \label{eq:T0InverseLocalizedGammaEstimate}
  \end{equation}
  where $c_1 = c_1(\CalQ, d) > 0$ is a constant satisfying
  $\|\cdot\|_{L^1 (Q_i ')} \leq c_1 \cdot \|\cdot\|_{L^{d+1} (Q_i ')}$ for all $i \in I$,
  which exists since the $(Q'_i)_{i \in I}$ are uniformly bounded.
  Estimate \eqref{eq:T0InverseLocalizedGammaEstimate} implies that
  \[
    \sup_{i \in I}
      \sum_{j \in J}
        \|\varphi_i \cdot \gamma_j\|_{\sup}
    \leq \sup_{i \in I}
           \sum_{j \in J}
             \|\varphi_i \cdot \gamma_j\|_{\Fourier L^1}
    \leq c_0 c_1 \cdot M < \infty ,
  \]
  where $M$ is as in \eqref{eq:T0InverseMainAssumption}.
  This guarantees the locally uniform convergence on $\CalO$ of the series
  ${t_0 = \sum_{j \in J} \gamma_j}$.
  Indeed, if $\xi \in \CalO$ is arbitrary, then $\xi \in Q_i$ for some $i \in I$ where $Q_i$ is open;
  furthermore, $\sum_{\ell \in i^\ast} \varphi_\ell \equiv 1$ on $Q_i$ and hence
  \(
    \sum_{j \in J}
      \|\gamma_j\|_{L^\infty (Q_i)}
    \leq \sum_{j \in J}
           \sum_{\ell \in i^\ast}
             \|\varphi_\ell \cdot \gamma_j\|_{\sup}
    < \infty
  \),
  which shows that the series $t_0 = \sum_{j \in J} \gamma_j$ converges uniformly on $Q_i$.
  By locally uniform convergence, we see that $t_0$ is continuous on $\CalO$.
  Equation~\eqref{eq:GSI_assumption} shows that $A \leq t_0 \leq B$ almost everywhere
  on $\CalO$; since $\CalO$ is open and $t_0$ continuous,
  this estimate necessarily holds pointwise on $\CalO$.

  Finally, since $\supp \varphi_i \subset \CalO$ is compact, we see
  $\varphi_i \, t_0 = \sum_{j \in J} \varphi_i \gamma_j$ with uniform convergence of the series,
  and hence with convergence in $L^1(\RHat^d)$, since all summands have support in the fixed
  compact set $\overline{Q_i} \subset \CalO$.
  Thus, $\Fourier^{-1} (\varphi_i \, t_0) = \sum_{j \in J} \Fourier^{-1} (\varphi_i \, \gamma_j)$,
  which leads to the estimate
  \(
    \sup_{i \in I}
      \|\Fourier^{-1} (\varphi_i \, t_0)\|_{L^1}
    \leq \sup_{i \in I}
           \sum_{j \in J}
             \|\varphi_i \cdot \gamma_j\|_{\Fourier L^1}
    \leq c_0 \, c_1 \cdot M
    < \infty
  \).
  Thus, $t_0$ is tame, so that Proposition~\ref{prop:FourierMultiplierBound} shows that
  $T_0 = \Phi_{t_0} : \StandardDecompSp \to \StandardDecompSp$ is well-defined and bounded.
 \\\\
 \textbf{Step 2.}
  In this step, we prepare for applying Lemma~\ref{lem:DerivativesOfReciprocal};
  we cannot apply it directly, since $t_0$ might not be $C^{d+1}$.
  Thus, we will construct a sequence $(g_N)_{N \in \N}$ of smooth functions approximating $t_0$.
  We will then apply Lemma~\ref{lem:DerivativesOfReciprocal} to the $g_N$ in Step~3.

  For the construction of the $(g_N)_{N \in \N}$, first note that $J$ is infinite;
  indeed, we have $\widehat{g_j} \in C_0 (\RHat^d)$ for all $j \in J$ since $g_j \in L^1(\R^d)$;
  thus, \eqref{eq:GSI_assumption} can only hold if $J$ is infinite.
  Since $J$ is countable, we thus have $J = \{ j_n \colon n \in \N \}$
  for certain pairwise distinct $j_n \in J$.
  With this, define $g_N := \sum_{n=1}^N \gamma_{j_n} \in C^\infty (\RHat^d)$.
  As seen in Step~1, $g_N \to t_0$ locally uniformly on $\CalO$.
  Since $0 < A \leq t_0 \leq B$ on $\CalO$, this easily implies $G_{N} \to \tfrac{1}{t_0}$
  locally uniformly on $\CalO$, where we defined
  \[
    G_N : \CalO \to \R,
          \xi \mapsto \begin{cases}
                        (g_N (\xi) )^{-1} , & \text{if } g_N (\xi) \neq 0 , \\
                        0,                  & \text{otherwise}.
                      \end{cases}
  \]
  Thus, $\varphi_i \cdot G_N \to \varphi_i \cdot t_0^{-1}$ in $L^1(\RHat^d)$,
  and hence $\Fourier^{-1} (\varphi_i \, G_N) \to \Fourier^{-1} (\varphi_i \cdot t_0^{-1})$
  uniformly as $N \to \infty$.
  Therefore, Fatou's lemma shows that
  \begin{equation}
    \| \Fourier^{-1} (\varphi_i \cdot t_0^{-1}) \|_{L^1}
    \leq \liminf_{N \to \infty}
           \| \Fourier^{-1} (\varphi_i \, G_N) \|_{L^1}
    =    \liminf_{N \to \infty}
           \|
             \varphi_i^{\normalizedBAPU}
             \cdot (G_N \circ S_i)
           \|_{\Fourier L^1}.
    \label{eq:T0InverseApproximationFatou}
  \end{equation}
\\\\
  \textbf{Step 3.} We next estimate
  $\liminf_{N \to \infty} \| \varphi_i^{\normalizedBAPU} \cdot (G_N \circ S_i) \|_{\Fourier L^1}$.
  Define
  \[
    K_i^{(N)}
    : S_i^{-1}(\CalO) \to [0,\infty),
    \xi \mapsto \sum_{n=1}^N
                  \max_{|\alpha| \leq d+1}
                    \big|
                      \partial^\alpha (\gamma_{j_n} \circ S_i) (\xi)
                    \big| .
  \]
  Let $V_i \subset \CalO$ be open and bounded with
  $\overline{Q_i} \subset V_i \subset \overline{V_i} \subset \CalO$ and let $\eps \in (0,1)$.
  Since $g_N \to t_0$ uniformly on $V_i$ and $t_0 \geq A > 0$ on $\CalO \supset V_i$,
  there is $N_0 = N_0 (i,\eps) \in \N$ such that $g_N \geq (1-\eps) \, A =: A_\eps$ on $V_i$
  for all $N \geq N_0$.
  Note that
  $K_i^{(N)} (\xi) \geq \sum_{n=1}^N \gamma_{j_n} (S_i \xi) = g_N (S_i \xi) \geq A_\eps$
  for $\xi \in S_i^{-1}(V_i)$ and $N \geq N_0$.

  Define $U_i := S_i^{-1} (V_i)$, fix $\xi^{(0)} \in U_i$ and $\ell \in \FirstN{d}$, set
  \[
    U := \{
           \xi \in \RHat
           \colon
           (
            \xi_1^{(0)},
            \dots,
            \xi_{\ell-1}^{(0)},
            \xi,
            \xi_{\ell+1}^{(0)},
            \dots,
            \xi_d^{(0)}
           )
           \in U_i
         \}
  \]
  and, for $N \geq N_0$, let
  \(
    f_N : U \to [A_\eps, \infty),
          \xi \mapsto (g_N \circ S_i)
                      (
                       \xi_1^{(0)},
                       \dots,
                       \xi_{\ell-1}^{(0)},
                       \xi,
                       \xi_{\ell+1}^{(0)},
                       \dots,
                       \xi_d^{(0)}
                      ),
  \)
  noting that $\big| f_N^{(m)} (\xi_\ell^{(0)}) \big| \leq K_i^{(N)} (\xi^{(0)})$
  for all $m \in \FirstN{d+1}$.
  Hence, Lemma~\ref{lem:DerivativesOfReciprocal} shows for all $m \in \FirstN{d+1}$ that
  \begin{equation}
    \begin{split}
      \bigg|
        \frac{\partial^m}{\partial \xi_\ell^m} \Big|_{\xi = \xi^{(0)}}
          (G_N \circ S_i) (\xi)
      \bigg|
      & = \bigg|
            \frac{d^m}{d \xi^m} \Big|_{\xi = \xi_{\ell}^{(0)}}
            \frac{1}{f_N (\xi)}
          \bigg| \\
      & \leq C_{d+1} \cdot A_\eps^{-1}
             \cdot \max \big\{
                          A_\eps^{-1} \cdot K_i^{(N)} (\xi^{(0)}) ,
                          \big( A_\eps^{-1} \cdot K_i^{(N)} (\xi^{(0)}) \big)^m
                        \big\} \\
      & \leq C_{d+1} \cdot A_\eps^{-(d+2)} \cdot \big( K_i^{(N)} (\xi^{(0)}) \big)^{d+1},
    \end{split}
    \label{eq:T0InverseDerivativeEstimate}
  \end{equation}
  where $C_{d+1}$ is as in Lemma~\ref{lem:DerivativesOfReciprocal}.

  Since $\xi^{(0)} \in U_i$ was arbitrary, we have thus shown, for all $\xi \in U_i$ and $N \geq N_0$,
  \[
    \max_{\ell \in \FirstN{d}}
      \max_{0 \leq m \leq d+1}
        \big|
          \partial_\ell^m (G_N \circ S_i) (\xi)
        \big|
    \leq C_{d+1} \cdot A_\eps^{(-d+2)} \cdot \big( K_i^{(N)} (\xi) \big)^{d+1} .
  \]
  Finally, since $\varphi_i^{\normalizedBAPU} = \varphi_i \circ S_i$ vanishes outside of
  $Q_i' = S_i^{-1} (Q_i) \subset S_i^{-1} (V_i) = U_i$, the Leibniz rule shows
  \begin{align*}
    \big|
      \partial_\ell^m
      \big(
        \varphi_i^{\normalizedBAPU}
        \cdot (G_N \circ S_i)
      \big)(\xi)
    \big|
  &  \leq \sum_{s=0}^m
           \binom{m}{s} \,
           |\partial_\ell^{m-s} \varphi_i^{\normalizedBAPU} (\xi)|
           \, |\partial_\ell^s (G_N \circ S_i) (\xi)| \\
&    \leq c_0 C_{d+1} \cdot A_\eps^{-(d+2)}
         \cdot \big( K_i^{(N)} (\xi) \big)^{d+1}
         \cdot \Indicator_{Q_i '} (\xi)
  \end{align*}
  for all $\xi \in \RHat^d$, $\ell \in \FirstN{d}$, $0 \leq m \leq d+1$, and $N \geq N_0$.
  Thus, Lemma~\ref{lem:FourierDecayViaSmoothness} shows
  \begin{align*}
  &  \| \varphi_i^{\normalizedBAPU} \cdot (G_N \circ S_i) \|_{\Fourier L^1} \\
    & \quad \quad \quad \leq \frac{d+1}{\pi^d} \!\!
           \max_{\substack{\ell \in \FirstN{d} \\ 0 \leq m \leq d+1}} \!\!
             \big\|
               \partial_\ell^m \big( \varphi_i^{\normalizedBAPU} \cdot (G_N \circ S_i) \big)
             \big\|_{L^1} \! \\
&  \quad \quad \quad    \leq \frac{d+1}{\pi^d} \cdot c_0 \, C_{d+1}
           \cdot A_\eps^{-(d+2)}
           \, \big\| K_i^{(N)} \big\|_{L^{d+1} (Q_i ')}^{d+1} \\
    & \quad \quad \quad \leq \frac{d+1}{\pi^d} \cdot c_0 \, C_{d+1}
           \cdot A_\eps^{-(d+2)}
           \cdot \bigg(
                   \sum_{j \in J}
                     |\det C_j|^{-1} \,
                     \Big\|
                       \max_{|\alpha| \leq d+1}
                         \big| \partial^\alpha |\Fourier g_{i,j}^{\localized}|^2 \big| \,
                     \Big\|_{L^{d+1} (Q_i ')}
                 \bigg)^{d+1} \\
    & \quad \quad \quad \leq \frac{d+1}{\pi^d} \cdot c_0 C_{d+1}
           \cdot A_\eps^{-(d+2)} \cdot M^{d+1}.
  \end{align*}
  Since this holds for all $N \geq N_0 = N_0(i,\eps)$, and since $A_\eps = (1-\eps) A$
  where $\eps \in (0,1)$ is arbitrary, we thus see by virtue of
  Equation~\eqref{eq:T0InverseApproximationFatou} that
  \[
    \|
      \Fourier^{-1} (\varphi_i \cdot t_0^{-1})
    \|_{L^1}
    \leq \frac{d+1}{\pi^d} \cdot c_0 C_{d+1} \cdot A^{-(d+2)} \cdot M^{d+1}
    <    \infty
  \]
  for all $i \in I$.
  Hence, $t_0^{-1}$ is tame, and Proposition~\ref{prop:FourierMultiplierBound} shows that
  $\Phi_{t_0^{-1}} : \StandardDecompSp \to \StandardDecompSp$ is well-defined and
  bounded, with operator norm bounded
  by the right-hand side of Equation~\eqref{eq:MainTermInverseOperatorNorm}.
\\\\
  \textbf{Step 4.}
  Proposition~\ref{prop:FourierMultiplierBound}(iv) shows
  $\Phi_{t_0^{-1}} \Phi_{t_0} = \Phi_{\mathbf{1}} = \Phi_{t_0} \Phi_{t_0^{-1}}$,
  where $\mathbf{1} : \CalO \to \R, \xi \mapsto 1$.
  Directly from the definition of $\Phi_{\mathbf{1}}$ in
  Proposition~\ref{prop:FourierMultiplierBound}, we see $\Phi_{\mathbf{1}} f = f$ for all
  $f \in \StandardDecompSp$.
  Hence, $T_0 : \StandardDecompSp \to \StandardDecompSp$ is boundedly invertible
  with $T_0^{-1} = \Phi_{t_0^{-1}}$.
\end{proof}

\subsection{The remainder term}
\label{sub:RemainderTerm}

The next (technical) result provides an estimate of the operator norm
of the remainder term $R_0 : \StandardDecompSp \to \StandardDecompSp$
considered in Corollary~\ref{cor:RemainderTermR0}.
Here, we make use of a \emph{normalized} version $g_j^{\normalizedFactorisation}$
of the generators $(g_j)_{j \in J}$ of $\StandardGSI$, namely
\[
  g_j^{\normalizedFactorisation}
  := |\det B_j|^{-1/2} \cdot (\Modulation{-c_j} g_j)\circ B_j^{-t}
\]
for invertible affine-linear maps  $U_j = B_j (\mybullet) + c_j$; note that
$\widehat{\, g_j^{\normalizedFactorisation} \,} = |\det B_j|^{1/2} \cdot \widehat{g_j} \circ U_j$.

\begin{lemma}\label{lem:RemainderSimplified}
  Let $\CalQ = (S_i (Q_i '))_{i \in I} = (A_i \, (Q_i ') + b_i)_{i \in I}$ be
  an affinely generated cover of an open set ${\CalO \subset \RHat^d}$ of full measure.
  Let $\Phi = (\varphi_i)_{i \in I}$ be a regular partition of unity subordinate to $\CalQ$,
  and let $w = (w_i)_{i \in I}$ be a $\CalQ$-moderate weight.
  Let $\StandardGSI$ be a generalized shift-invariant system.
  Furthermore, assume that $\StandardGSI$ is $(w,v,\Phi)$-adapted for some weight
  $v = (v_j)_{j \in J}$, and assume that the function $t_0$ introduced in
  Equation~\eqref{eq:TAlphaDefinition} is tame.

  Suppose that there is a family $(U_j)_{j \in J}$ of invertible affine-linear maps
  $U_j = B_j (\mybullet) + c_j \vphantom{\sum_j}$ and a weight $v = (v_j)_{j \in J}$
  such that the Fourier transform of
  $g_j^{\normalizedFactorisation} = |\det B_j|^{-1/2} \cdot (\Modulation{-c_j} g_j)\circ B_j^{-t}$
  can be factorized as $\Fourier g_j^{\normalizedFactorisation} = h_{j,1} \cdot h_{j,2}$
  with $h_{j,1}, h_{j,2} \in C^{d+1}(\RHat^d)$ satisfying
  \[
    \max_{|\alpha| \leq d+1}
      |\partial^\alpha h_{j,2} (\xi)|
    \leq C' \cdot (1 + |\xi|)^{-(d+1)}
    \quad \text{for} \quad \xi \in \RHat^d .
  \]
  Moreover, suppose that $Y = (a_{i,j} X_{i,j})_{i \in I, j \in J}  \vphantom{\sum_j}$
  and $Z = (b_{i,j} X_{i,j})_{i \in I, j \in J}$ are of Schur-type, where
  \begin{align*}
    a_{i,j}  = \max \big\{ 1, \tfrac{w_i}{v_j} \big\} \,
       |\det B_j^t C_j|^{-1}
       \max \big\{ 1 ,\, |A_i^{-1} (b_i - c_j)|^{d+1} \big\}
       \max \big\{ 1 ,\, \|A_i^{-1} B_j\|^{d+1} \big\}
       \|C_j^t A_i\|^{d+1}
  \end{align*}
  and
  \[
    b_{i,j}
    = \max \{ 1, \tfrac{v_j}{w_i} \}
       \, \max \{ 1, |A_i^{-1} (b_i - c_j)| \}
       \, \max \{ 1, \|A_i^{-1} B_j\|^{d+1} \}
       \, \max \{ \|C_j^t A_i\|, \|C_j^t A_i\|^{d+1} \},
  \]
  and
  \[
    X_{i,j}  :=
     \max \big\{ 1 ,\, \|B_j^{-1} A_i\|^{d+1} \big\}
     \int_{Q_i '}
                  \max_{|\alpha| \leq d+1}
                    \big|
                      (\partial^\alpha h_{j,1}) \big( U_j^{-1} S_i (\xi) \big)
                    \big|
                \, d \xi \, .
  \]
  Then, for all $p,q \in [1,\infty]$, the operator
  $R_0 : \StandardDecompSp \to \StandardDecompSp$
  of Corollary~\ref{cor:RemainderTermR0} is bounded, with
  \(
    \|R_0\|_{\op}
    \leq C_0 C_{p,q} \| \Gamma_{\CalQ} \|_{\ell^q_w \to \ell^q_w}
         \cdot (C')^2
         \cdot \|Y\|_{\schur} \|Z\|_{\schur}
  \),
  where
  \begin{equation}
    C_0 := 24 \, \pi^2
            \left(\frac{8d}{\pi}\right)^{2d+2}
            12^d \, (d+1)^3
            \max \big\{ 1 ,\, R_\CalQ^{d+2} \big\}
            \max_{|\alpha| \leq d+1} C_{\CalQ,\Phi,\alpha}^2
    \label{eq:RemainderSimplifiedSpecialConstant}
  \end{equation}
  with $R_{\CalQ} := \max_{i \in I} \sup_{\xi \in Q_i '} |\xi|$
  and  $C_{p,q} := 1$ if $\max \{p,q\} < \infty$
  and $C_{p,q} := C_{\Phi} \cdot \|\Gamma_{\CalQ}\|_{\ell_w^q \to \ell_w^q}^2$ otherwise.
\end{lemma}

\begin{proof}
  For brevity, set $\nu(x) := \max \{1, x\}$ for $x \in [0,\infty)$,
  and note $\nu(xy) \leq \nu(x) \, \nu(y)$.
  This implies $\nu(w_i / w_\ell) \leq \nu(w_i/v_j) \cdot \nu(v_j/w_\ell)$,
  an estimate that we will employ frequently.

  According to Proposition~\ref{prop:RemainderAbstractTAlphaEstimate}
  and Corollary~\ref{cor:RemainderAbstractEstimate}, it suffices to estimate
  \begin{equation}
    L_1 = \sup_{i \in I} \,
            \sum_{\ell \in I} \,
              \sum_{j \in J}
                \sum_{k \in \ZZ^d \setminus \{0\}} \!\!
                  \nu \left(\frac{w_i}{w_\ell}\right)
                  \, K_{i,\ell,j,k}
    \quad \text{and} \quad
    L_2 = \sup_{\ell \in I} \,
            \sum_{i \in I} \,
              \sum_{j \in J}
                \sum_{k \in \ZZ^d \setminus \{0\}} \!\!
                  \nu \left( \frac{w_i}{w_\ell}\right)
                  \, K_{i,\ell,j,k}
    \label{eq:RemainderSimplifiedTargetQuantitiy}
  \end{equation}
  where
  \(
    K_{i,\ell,j,k}
    := |\det C_j|^{-1}
       \cdot \big\|
               \overline{\widehat{g_j}}
               \cdot \widehat{g_j}(\mybullet - C_j^{-t}k)
               \cdot \varphi_i (\mybullet - C_j^{-t}k)
               \cdot \varphi_\ell
             \big\|_{\Fourier L^1}
  \).
  In order to do so, note that
  \(
    \widehat{g_j}
    = |\det B_j|^{-1/2}
      \cdot (\Fourier g_j^{\normalizedFactorisation}) \circ U_j^{-1}
  \).
  Hence, since $\Fourier g_j^{\normalizedFactorisation} = h_{j,1} \cdot h_{j,2}$ by assumption,
  the term $K_{i,\ell,j,k}$ can be estimated as follows:
  \begin{align*}
    & K_{i,\ell,j,k}
      = \big|\det B_j^t C_j \,\big|^{-1}
        \Big\|\,
          \big(\,
            \overline{\Fourier g_j^{\normalizedFactorisation}} \circ U_j^{-1}
          \,\big)
          \cdot \Translation{C_j^{-t} k}
                \big( (\Fourier g_j^{\normalizedFactorisation}) \circ U_j^{-1} \big)
          \cdot \big( \Translation{C_j^{-t} k} \varphi_i \big)
          \cdot \varphi_\ell
        \,\Big\|_{\Fourier L^1} \\
    & \leq \big|\det B_j^t C_j \,\big|^{-1}
            \Big\|
                   \Translation{C_j^{-t} k}
                   \big(
                     \varphi_i
                     \cdot (h_{j,1} \circ \! U_j^{-1})
                   \big)
                   \cdot \big(
                           \overline{\strut h_{j,2}} \circ \! U_j^{-1}
                         \big)
            \Big\|_{\Fourier L^1} \\
            & \quad \quad \quad \quad \quad \quad \quad \quad \quad \cdot
            \Big\|
                   \varphi_\ell
                   \cdot \big(
                           \overline{\strut h_{j,1}} \circ \! U_j^{-1}
                         \big)
                   \cdot \Translation{C_j^{-t} k}
                         \big( h_{j,2} \circ \! U_j^{-1} \big)
            \Big\|_{\Fourier L^1} \\
    & =: \big|\det B_j^t C_j \,\big|^{-1}
         \cdot K_{i,j,k}^{(1)}
         \cdot K_{\ell,j,k}^{(2)} \, .
  \end{align*}
  Using the preceding estimate, one can bound $L_1$ from
  Equation~\eqref{eq:RemainderSimplifiedTargetQuantitiy} as follows:
  \begin{align*}
      L_1
      &= \sup_{i \in I}
           \sum_{j\in J, \ell \in I}
             \sum_{k \in \ZZ^d \setminus \{0\}}
               \nu \left(\frac{w_i}{w_\ell}\right)
               K_{i,\ell,j,k} \\
      & \leq \sup_{i \in I}
               \sum_{j \in J}
                 \left[
                   \bigg(
                     \sum_{\ell \in I}
                       \sum_{k \in \ZZ^d \setminus \{0\}}
                         \nu \left( \frac{v_j}{w_\ell}\right)
                         \, K_{\ell,j,k}^{(2)}
                   \bigg)
                   \cdot
                   |\det B_j^t C_j|^{-1} \,\,
                   \nu \left(\frac{w_i}{v_j}\right)
                   \sup_{k \in \ZZ^d \setminus \{0\}}
                     K_{i,j,k}^{(1)}
                 \right] \\
      & \leq \bigg(
               \sup_{j \in J}
                 \sum_{\ell \in I}
                   \nu \left(\frac{v_j}{w_\ell}\right)
                   \sum_{k \in \ZZ^d \setminus \{0\}}
                     K_{\ell,j,k}^{(2)}
             \bigg)
             \cdot
             \sup_{i \in I}
               \sum_{j \in J}
                 \bigg(
                   \nu \left(\frac{w_i}{v_j}\right)
                   \, |\det B_j^t C_j|^{-1}
                   \sup_{k \in \ZZ^d \setminus \{0\}}
                     K_{i,j,k}^{(1)}
                 \bigg).
                  \numberthis \label{eq:RemainderSimplifiedL1Estimate}
  \end{align*}
  A similar calculation gives
  \begin{equation}
    \begin{split}
      L_2
       \leq \bigg(
               \sup_{\ell \in I}
                 \sum_{j \in J}
                   \nu \left(\frac{v_j}{w_\ell}\right)
                   \sum_{k \in \ZZ^d \setminus \{0\}}
                     K_{\ell,j,k}^{(2)}
             \bigg)
             \cdot
             \sup_{j \in J}
               \sum_{i \in I}
                 \bigg(
                   \nu \left(\frac{w_i}{v_j}\right)
                   |\det B_j^t C_j|^{-1}
                   \sup_{k \in \ZZ^d \setminus \{0\}}
                     K_{i,j,k}^{(1)}
                 \bigg) \, .
    \end{split}
    \label{eq:RemainderSimplifiedL2Estimate}
  \end{equation}
  The remainder of the proof is divided into four steps:
\\\\
  \textbf{Step 1.} \emph{Estimates for $K_{i,j,k}^{(1)}$ and $K_{\ell,j,k}^{(2)}$.}
  For $j \in J$ and $k \in \ZZ^d$, set
  \(
    H_{j,k}
    := \overline{h_{j,1}} \cdot \Translation{B_j^{-1} C_j^{-t} k} h_{j,2}.
  \)
  Since
  $\Translation{\xi} (g \circ U_j^{-1})
   = (\Translation{B_j^{-1} \xi} g) \circ U_j^{-1}$ for any $\xi \in \RHat^d$
  and $g : \RHat^d \to \CC$, it follows that
  \[
    \big( h_{j,1} \circ U_j^{-1} \big)
    \cdot \Translation{-C_j^{-t} k}
          \big( \overline{\strut h_{j,2}} \circ U_j^{-1} \big)
    = \big(
        h_{j,1}
        \cdot \Translation{-B_j^{-1} C_j^{-t} k} \overline{\strut h_{j,2}}
      \big)
      \circ U_j^{-1}
    = \overline{\strut H_{j,-k}} \circ U_j^{-1}.
  \]
  Using the normalization
  $\varphi_i^{\normalizedBAPU} = \varphi_i \circ S_i$ of $\varphi_i$, a direct calculation shows
  \begin{align}
      K_{i,j,k}^{(1)}
       = \Big\|
            \varphi_i
            \cdot \big( h_{j,1} \circ U_j^{-1} \big)
            \cdot \Translation{-C_j^{-t} k}
                  \big( \overline{h_{j,2}} \circ U_j^{-1} \big)
          \Big\|_{\Fourier L^1}
      = \Big\|
          \varphi_i^{\normalizedBAPU}
          \cdot \big(
                  \overline{\strut H_{j,-k}} \circ U_j^{-1} \circ S_i
                \big)
        \Big\|_{\Fourier L^{1}} \, .
    \label{eq:RemainderSimplifiedGamma1FourierEstimate}
  \end{align}
  Now, define
  \(
    \zeta_j : \RHat^d \to [0,\infty), \;
               \xi \mapsto \max_{|\alpha| \leq d+1}
                             |\partial^\alpha h_{j,1} (\xi)| \, .
  \)
  By applying Leibniz' rule, combined with the assumption
  $ \max_{|\alpha| \leq d+1}
      |\partial^\alpha h_{j,2} (\xi)|
    \leq C' \cdot (1 + |\xi|)^{-(d+1)}$
  and the identity $\sum_{\beta \leq \alpha} \binom{\alpha}{\beta} = 2^{|\alpha|}$, we  see
  \begin{align} \label{eq:partialHjk}
       \big| \partial^\alpha H_{j,k} (\xi) \big|
       \leq 2^{|\alpha|} \cdot C' \cdot
              (1 + |\xi - B_j^{-1} C_j^{-t} k|)^{-(d+1)}
             \cdot \zeta_j (\xi)
  \end{align}
  for all $\alpha \in \NN_0^d$ with $|\alpha| \leq d+1$ and all $\xi \in \RHat^d$.
  This, together with Lemma~\ref{lem:LinearChainRule}, yields that, for all
  $n \in \underline{d}$ and $m \in \{0,\dots,d+1\}$,
  \begin{align*}
      & \Big|
          \big[
            \partial_n^m
            \big(
              \overline{\strut H_{j,k}} \circ U_j^{-1} \circ S_i
            \big)
          \big](\xi)
        \Big| \\
&         \leq \|B_j^{-1} A_i\|^m
             \cdot d^m
             \cdot \max_{\beta \in \NN_0^d \text{ with } |\beta| = m}
                     \big|
                       (\partial^\beta H_{j,k})
                       \big( U_j^{-1} (S_i (\xi)) \big)
                     \big|
             \\
      & \leq (2d)^{d+1} C'
              \max \big\{ 1,\, \|B_j^{-1} A_i\|^{d+1} \big\}
              \zeta_j \big( U_j^{-1} (S_i (\xi)) \big)
              \big(
                     1 + |U_j^{-1} (S_i (\xi)) - B_j^{-1} C_j^{-t} k|
                   \big)^{-(d+1)} \, .
  \end{align*}
  Since $\Phi$ is a regular partition of unity, we have
  \(
    |\partial^\alpha \varphi_i^{\normalizedBAPU} (\xi)|
    \leq C_{\CalQ,\Phi,\alpha} \cdot \Indicator_{Q_i '} (\xi) \vphantom{\sum_j}
  \)
  for all $\xi \in \RHat^d$ and $\alpha \in \NN_0^d$.
  Thus, setting $C_1 := (4d)^{d+1} C' \cdot \max_{|\alpha| \leq d+1} C_{\CalQ,\Phi,\alpha}$
  and invoking Leibniz's rule once more, we see that
  \begin{align*}
    & \big|
        \big[
          \partial_n^{d+1}
          \big(
            \varphi_i^{\normalizedBAPU}
            \cdot \big(
                    \overline{\strut H_{j,-k}} \circ U_j^{-1} \circ S_i
                  \big)
          \big)
        \big] (\xi)
      \big| \\
&       \leq \sum_{m=0}^{d+1}
             \binom{d+1}{m}
             \big|
               \partial_n^{d+1-m} \varphi_i^{\normalizedBAPU} (\xi)
             \big|
             \cdot \big|
                     \partial_n^m
                     \big(
                       \overline{\strut H_{j,-k}}
                       \circ U_j^{-1} \circ S_i
                     \big)
                     (\xi)
                   \big| \\
    & \leq
           C_1
            \max \big\{ 1,\, \|B_{j}^{-1} A_i\|^{d+1} \big\}
            \Indicator_{Q_i '} (\xi)
            \zeta_j \big( U_j^{-1} (S_i (\xi)) \big)
            \big(
                   1 + |U_j^{-1} (S_i (\xi)) + B_j^{-1} C_j^{-t} k|
                 \big)^{-(d+1)} \, .
  \end{align*}
  Clearly, the same overall estimate also holds for
  \(
    |[\varphi_i^{\normalizedBAPU}
    \cdot (\overline{\strut H_{j,-k}} \circ U_j^{-1} \circ S_i)] (\xi)|
  \)
  itself instead of its derivative
  \(
    \big|
      \partial_n^{d+1}
      \big(
        \varphi_i^{\normalizedBAPU}
        \cdot (\overline{\strut H_{j,-k}} \circ U_j^{-1} \circ S_i)
      \big)(\xi)
    \big|
  \).
  Thus, setting
  \[
    C_2 := (4d/\pi)^{d+1}
           \cdot (d+1)\pi
           \cdot C'
           \cdot \max_{|\alpha| \leq d+1}
                   C_{\CalQ,\Phi,\alpha}
  ,
  \]
  we can apply Lemma~\ref{lem:FourierDecayViaSmoothness}
  and Equation~\eqref{eq:RemainderSimplifiedGamma1FourierEstimate} to conclude
  \begin{align*}
    K_{i,j,k}^{(1)}
    &=
      \Big\|
        \varphi_i^{\normalizedBAPU}
        \cdot \big(
                \overline{\strut H_{j,-k}} \circ U_j^{-1} \circ S_i
              \big)
      \Big\|_{\Fourier L^{1}}
      \leq \frac{d+1}{\pi^d}
           \cdot \max_{\alpha \in \mathbb{I}}
                   \Big\|
                     \partial^\alpha
                     \Big(
                       \varphi_i^{\normalizedBAPU}
                       \cdot \big(
                               \overline{\strut H_{j,-k}}
                               \circ U_j^{-1}
                               \circ S_i
                             \big)
                     \Big)
                   \Big\|_{L^1} \\
    & \leq C_2
           \cdot \max \big\{1,\, \|B_j^{-1} A_i\|^{d+1} \big\}
           \cdot \int_{Q_i '}
                   \zeta_j \big( U_j^{-1} (S_i (\xi)) \big) \\
     & \quad \quad \quad \quad \quad \quad \quad \quad \quad
                   \cdot \big(
                           1 + |
                                U_j^{-1}(S_i (\xi))
                                + B_j^{-1} C_j^{-t} k
                               |
                         \big)^{-(d+1)}
                 \, d \xi \, ,
  \end{align*}
  where $\mathbb{I} := \{0\} \cup \{(d+1) \cdot e_n \,:\, n \in \underline{d} \}$.
  By similar arguments as for $K_{i,j,k}^{(1)}$, one obtains
  \begin{align*}
    K_{\ell,j,k}^{(2)}
    & \leq C_2
           \cdot \max \big\{ 1, \|B_j^{-1} A_\ell \|^{d+1} \big\}
           \\
           & \quad \quad \quad \quad \quad \quad \quad
                 \cdot \int_{Q_\ell '}   \zeta_j \big( U_j^{-1} (S_\ell (\xi)) \big)
                   \cdot \big(
                           1 + |
                                U_j^{-1}(S_\ell (\xi))
                                - B_j^{-1} C_j^{-t} k
                               |
                         \big)^{-(d+1)}
                 \, d \xi .
  \end{align*}
\\\\
  \textbf{Step 2.}
  \emph{Estimating the supremum over $k \in \ZZ^d \setminus \{0\}$.}
  Note that $|\xi|  \leq \|A^{-1}\| \cdot |A \xi|$,
  and thus $|A \xi| \geq \|A^{-1}\|^{-1} \cdot |\xi|$ for any $\xi \in \RHat^d$ and
  $A \in \GL(\RR^d)$.
  Hence,
  \begin{equation}
    \begin{split}
      \big|
        U_j^{-1} (S_i (\xi)) \pm B_j^{-1} C_j^{-t} k
      \big|
      & = \big|
            B_j^{-1} \big( S_i (\xi) - c_j \big)
            \pm B_j^{-1} C_j^{-t} k
          \big| \\
      & = \big|
            B_j^{-1} A_i
            \big(
              \xi
              + A_i^{-1} (b_i - c_j)
              \pm A_i^{-1} C_j^{-t} k
            \big)
          \big| \\
      & \geq \|A_i^{-1} B_j\|^{-1}
             \cdot \big|
                     \xi + A_i^{-1} (b_i - c_j) \pm A_i^{-1} C_j^{-t} k
                   \big| \, .
    \end{split}
    \label{eq:SumSupremumJointAuxiliaryEstimate}
  \end{equation}
  This implies for arbitrary $i \in I$, $\xi \in Q_i '$, $k \in \ZZ^d \setminus \{0\}$,
  and $j \in J$ that
  \begin{align*}
    & \quad
      \|C_j^t A_i\|^{-1}
      \leq 1 + |A_i^{-1} C_j^{-t} k|
      \leq 1
           + \big| \xi + A_i^{-1} (b_i - c_j) \pm A_i^{-1} C_j^{-t} k \big|
           + |\xi|
           + |A_i^{-1} (b_i - c_j)| \\
    &
      \leq 3 \, \max \{1, R_{\CalQ}\}
            \max \big\{ 1 ,\, |A_i^{-1} (b_i - c_j)| \big\}
            \big(
                   1 + |\xi + A_i^{-1} (b_i - c_j) \pm A_i^{-1} C_j^{-t} k|
                 \big) \\
    &
      \leq 3 \, \max \{1, R_{\CalQ}\}
            \max \big\{ 1 ,\, |A_i^{-1} (b_i - c_j)| \big\}
            \Big(
                   1
                   + \|A_i^{-1} B_j\|
                     \cdot \big|
                             U_j^{-1} (S_i (\xi)) \pm B_j^{-1} C_j^{-t} k
                           \big|
                 \Big) \\
    &
      \leq 3 \, \max \{1, R_{\CalQ}\}
            \max \big\{ 1 ,\, |A_i^{-1} (b_i - c_j)| \big\} \\
   &  \quad \quad \quad \quad \quad \quad
           \cdot  \max \big\{1 ,\, \|A_i^{-1} B_j\| \big\}
            \big(
                   1 + \big|
                         U_j^{-1} (S_i (\xi)) \pm B_j^{-1} C_j^{-t} k
                       \big|
                 \big) \, .
  \end{align*}
  Setting $C_3 := 3^{d+1} \cdot \max \big\{1, R_{\CalQ}^{d+1} \big\}$,
  the preceding estimate implies
  \begin{align*}
    & \sup_{k \in \ZZ^d \setminus \{0\}}
        \Big(
          1 + \big| U_j^{-1} (S_i (\xi)) \pm B_j^{-1} C_j^{-t} k \big|
        \Big)^{-(d+1)} \\[0.1cm]
    & \quad \quad \quad
      \leq C_3
            \max \big\{ 1 ,\, |A_i^{-1} (b_i - c_j)|^{d+1} \big\}
            \max \big\{1 ,\, \|A_i^{-1} B_j\|^{d+1} \big\}
            \|C_j^t A_i\|^{d+1}
  \end{align*}
  for all $i \in I$, $\xi \in Q_i '$, and $j \in J$.
  Using this, and the estimates for $K_{i,j,k}^{(n)}$ that we derived in Step~1, we see that
  \begin{equation}
    \begin{split}
      \sup_{k \in \ZZ^d \setminus \{0\}}
        K_{i,j,k}^{(n)}
      &\leq C_2 C_3
            \max \big\{ 1 ,\, |A_i^{-1} (b_i - c_j)|^{d+1} \big\}
             \max \big\{1 ,\, \|A_i^{-1} B_j\|^{d+1} \big\}
             \|C_j^t A_i\|^{d+1}
             X_{i,j} \\
      &=    C_2 C_3  |\det B_j^t C_j|  \big( \nu(w_i/v_j) \big)^{-1} \cdot Z_{i,j}
    \end{split}
    \label{eq:RemainderKSupremumEstimate}
  \end{equation}
  for $n \in \{1,2\}$, $i \in I$, and $j \in J$.
\\\\
  \textbf{Step 3.} \emph{Estimating the sum over $k \in \ZZ^d \setminus \{0\}$.}
  Estimate~\eqref{eq:SumSupremumJointAuxiliaryEstimate} implies
  \begin{align*}
    1 + \big| U_j^{-1} (S_i (\xi)) + B_j^{-1} C_j^{-t} k \big|
    & \geq 1 + \|A_i^{-1} B_j \|^{-1} \cdot |\xi + A_i^{-1} (b_i - c_j) + A_i^{-1} C_j^{-t} k| \\
    & \geq \big( \max \{ 1, \|A_i^{-1} B_j\| \} \big)^{-1} \\
&\quad \quad \quad\quad \quad \quad           \cdot \big( 1 + |\xi + A_i^{-1} (b_i - c_j) + A_i^{-1} C_j^{-t} k \big) .
  \end{align*}
  By combining this estimate with Corollary~\ref{cor:nonzero_series_estimate},
  we see for any $\xi \in Q_i '$ that
  \begin{align*}
    & \sum_{k \in \ZZ^d \setminus \{0\}} \!\!\!
        \big(
          1 + |U_j^{-1} (S_i (\xi)) + B_j^{-1} C_j^{-t} k |
        \big)^{-(d+1)} \\
    & \quad \leq \max\{ 1 , \|A_i^{-1} B_j\|^{d+1} \}
           \sum_{k \in \ZZ^d \setminus \{0\}} \!\!\!
             (1 + |\xi + A_i^{-1} (b_i - c_j) + A_i^{-1} C_j^{-t} k|)^{-(d+1)} \\
    & \quad \leq (d+1) \, 2^{3 + 4d}
           \cdot \max\{ 1 , \|A_i^{-1} B_j\|^{d+1} \} \\
     & \quad \quad \quad \quad \quad \quad     \quad \quad \quad
           \cdot (1 + |\xi + A_i^{-1} (b_i - c_j)|)
           \cdot \max \big\{ \| C_j^{t} A_i \|, \| C_j^t A_i \|^{d+1} \big\} \\
    & \quad \leq (d+1) \, 2^{3 + 4d} (2 + R_\CalQ)
           \cdot \max\{ 1 , \|A_i^{-1} B_j\|^{d+1} \} \\
    &       \quad \quad \quad \quad \quad \quad \quad \quad \quad
           \cdot \max\{ 1, |A_i^{-1} (b_i - c_j)| \}
           \cdot \max \big\{ \| C_j^{t} A_i \|, \| C_j^t A_i \|^{d+1} \big\} .
  \end{align*}
  Here, we used in the last step that $|\xi| \leq R_\CalQ$ since $\xi \in Q_i '$.

  \smallskip{}

  By combining this estimate with the estimate for $K_{i,j,k}^{(n)}$ from Step~1,
  we see for $n \in \{1,2\}$ and arbitrary $i \in I$ and $j \in J$ that
  \begin{align*}
      &\sum_{k \in \ZZ^d \setminus \{0\}}
          K_{i,j,k}^{(n)} \\
      &\leq C_2 \max \big\{1, \|B_j^{-1} A_i \|^{d+1} \big\} \\
      & \quad \quad \quad \quad \quad \quad
           \cdot \int_{Q'_i}
              \zeta_j \big( U_j^{-1} (S_i(\xi)) \big)
              \sum_{k \in \ZZ^d \setminus \{0\}}
                \big(
                  1 + |U_j^{-1} (S_i (\xi)) + B_j^{-1} C_j^{-t} k |
                \big)^{-(d+1)}
            \; d\xi \\
       &\leq C_4
             \max \{ 1, \|A_i^{-1} B_j\|^{d+1} \}
             \max \{ 1, |A_i^{-1} (b_i - c_j)| \}
             \cdot \max \big\{ \|C_j^t A_i\|, \|C_j^t A_i\|^{d+1} \big\}
             \, X_{i,j} \\
       & =   C_4 \cdot \big( \nu (v_j / w_i) \big)^{-1} \cdot Y_{i,j} \, ,
   \numberthis \label{eq:RemainderKSeriesEstimate}
  \end{align*}
  where we defined $C_4 := (d+1) \cdot 2^{3 + 4d} \cdot (2 + R_\CalQ) \cdot C_2$.
\\\\
  \textbf{Step 4.} \emph{Completing the proof.}
  Combining the two estimates \eqref{eq:RemainderSimplifiedL1Estimate}
  and \eqref{eq:RemainderSimplifiedL2Estimate}
  with the estimates obtained in Equations~\eqref{eq:RemainderKSeriesEstimate} and
  \eqref{eq:RemainderKSupremumEstimate}, we conclude that
  \begin{align*}
     L_1 &\leq  \bigg(
                   \sup_{j \in J}
                     \sum_{\ell \in I}
                       \nu \left( \frac{v_j}{w_\ell} \right)
                       \sum_{k \in \ZZ^d \setminus \{0\}} \!\!
                         K_{\ell,j,k}^{(2)}
                 \bigg)
             \cdot
             \sup_{i \in I}
               \sum_{j \in J}
                 \bigg(
                   \nu \left( \frac{w_i}{v_j} \right)
                   \, |\det B_j^t C_j|^{-1}
                   \sup_{k \in \ZZ^d \setminus \{0\}} \!\!
                     K_{i,j,k}^{(1)}
                 \bigg) \\
    & \leq C_2 C_3 C_4 \,
           \| Y \|_{\schur} \, \| Z \|_{\schur}
      \leq C_0 \cdot (C')^2 \cdot \| Y \|_{\schur} \, \| Z \|_{\schur} \, .
  \end{align*}
  The estimate $L_2 \leq C_0 \cdot (C')^2 \cdot \| Y \|_{\schur} \, \| Z \|_{\schur}$
  is obtained similarly.
  Hence, an application of Corollaries~\ref{cor:RemainderTermR0} and
  \ref{cor:RemainderAbstractEstimate} gives
  \(
    \| R_0 \|_{\op}
    \leq C_0 C_{p,q} \|\Gamma_{\CalQ} \|_{\ell^q_w \to \ell^q_w}
         \cdot (C')^2
         \cdot \| Y \|_{\schur} \|Z\|_{\schur}
  \),
  as desired.
\end{proof}

\section{Results for structured systems}
\label{sec:Structured_Compatible}

In this section, we provide further simplified conditions for the boundedness and
invertibility of the frame operator.
For this, we will assume throughout this section that the family $(g_j)_{j \in J}$
of functions $g_j \in L^1(\RR^d) \cap L^2 (\RR^d)$ defining the system $\StandardGSI$ possess
the form
\begin{equation}\label{eq:structured_system}
  g_j = |\det A_j|^{1/2} \cdot M_{b_j} [g \circ A^t_j]
\end{equation}
for certain $A_j \in \mathrm{GL}(d, \RR)$ and $b_j \in \RHat^d$ and a fixed
$g \in L^1(\R^d) \cap L^2(\R^d)$ satisfying $\widehat{g} \in C^\infty (\RHat^d)$.

Observe that \eqref{eq:structured_system} can be written as
$g_j = |\det A_j|^{-1/2} \cdot \Fourier^{-1} (\widehat{g} \circ S_j^{-1})$,
where $S_j = A_j (\cdot ) + b_j$.

\subsection{Simplified criteria for invertibility of the frame operator}
\label{sub:invertibility_simplified}

In this subsection, we give simplified versions of the estimates for the operator norms of
$T_0^{-1}$
and $R_0$,
under the assumption that the generators $(g_j)_{j \in J}$ of the system $\StandardGSI$
have the form \eqref{eq:structured_system} and that the lattices
$C_j \ZZ^d$ are given by $C_j = \delta A_{j}^{-t}$
for a suitable $\delta > 0$.
We begin with a simplified version of Proposition~\ref{prop:MainTermSimplified}.

\begin{proposition}\label{prop:MainTermForStructuredSystems}
  Let $\CalQ = \big(S_j (Q_j ')\big)_{j \in J} = \big(A_j (Q_j ') + b_j)_{j \in J}$ be an
  affinely generated cover of an open set $\CalO \subset \RHat^d$ of full measure.
  Let $\Phi = (\varphi_j)_{j \in J}$ be a regular partition of unity subordinate to $\CalQ$.
  Let $\StandardGSI$ be  such that $C_j := \delta \cdot A_j^{-t}$ for some $\delta > 0$
  and $g_j := |\det A_j|^{1/2} \cdot M_{b_j} [g \circ A_j^t]$
  for some $g \in L^1(\R^d) \cap L^2(\R^d)$ with $\widehat{g} \in C^\infty(\RHat^d)$.
  Suppose that there is some $A' > 0$ satisfying
  $A' \leq \sum_{j \in J} |\widehat{g} (S_j^{-1} \xi)|^2$ for almost all $\xi \in \CalO$,
  and that
  \begin{align*}
    M_0 &:= \sup_{i \in J}
             \sum_{j \in J}
               \bigg[
                 \max \big\{ 1 ,\, \|A_j^{-1} A_i\|^{d+1} \big\}
                 \cdot \bigg(
                         \int_{Q_i '}
                           \max_{|\alpha| \leq d+1}
                             \big|
                              (\partial^\alpha \widehat{g}) \big( S_j^{-1} (S_i \xi) \big)
                             \big|^{2 (d+1)}
                         \, d\xi
                       \bigg)^{1/(d+1)}\,
               \bigg] \\
&        < \infty.
 \end{align*}
 Then the function $t_0$ defined in Equation~\eqref{eq:TAlphaDefinition} is continuous
 on $\CalO$ and tame, and the estimate ${A' \leq \sum_{j \in J} |\widehat{g} (S_j^{-1} \xi)|^2}$
 holds for all $\xi \in \CalO$.
 Furthermore, for any $p,q \in [1,\infty]$ and any $\CalQ$-moderate weight $w = (w_j)_{j \in J}$,
 the operator
 \[
   T_0 := \Phi_{t_0} : \DecompSp(\CalQ,L^p,\ell_w^q) \to \DecompSp(\CalQ,L^p,\ell_w^q)
 \]
 with $\Phi_{t_0}$ as defined in Proposition~\ref{prop:FourierMultiplierBound}
 is well-defined, bounded, and boundedly invertible, with
 \[
   \|
     T_0^{-1}
   \|_{\DecompSp(\CalQ,L^p,\ell_w^q) \to \DecompSp(\CalQ,L^p,\ell_w^q)}
   \leq C'_d  \cdot N_{\CalQ}^2 C_\Phi
        \cdot \Big[ \max_{|\alpha| \leq d+1} C_{\CalQ,\Phi,\alpha} \Big]
        \cdot (A')^{-1}
        \cdot \bigg( \frac{M_0}{A'} \bigg)^{d+1}
        \cdot \delta^{d} \, ,
 \]
  where $C'_d = C_d \cdot (2d)^{(d+1)^2}$ with $C_d$ as in
  Equation~\eqref{eq:MainTermSimplifiedSpecialConstant}.
\end{proposition}

\begin{proof}
  We apply Proposition~\ref{prop:MainTermSimplified}.
  For this, note that since $C_j = \delta \cdot A_j^{-t}$ and
  $\widehat{g_j} = |\det A_j|^{-1/2} \cdot \widehat{g} \circ S_j^{-1}$,
  the $\CalQ$-localized version $g_{i,j}^{\localized}$ of $g_j$
  defined in \eqref{eq:NormalizedVersionDoubleIndex} satisfies
  \(
    \Fourier g_{i,j}^{\localized}
    = \widehat{g_j} \circ S_i
    = |\det A_j|^{-1/2} \cdot \widehat{g} \circ S_j^{-1} \circ S_i
  \)
 and, moreover,
 \(
   |\det C_j|^{-1} \cdot |\Fourier g_{i,j}^{\localized}|^2
   = \delta^{-d} \cdot |\widehat{g}|^2 \circ S_j^{-1} \circ S_i .
 \)
Leibniz rule entails the pointwise estimate
 \[
   \big|
     \partial^\alpha |\widehat{g}|^2 (\xi)
   \big|
   = \Big|
       \partial^\alpha
       \big( \widehat{g}(\xi) \cdot \overline{\strut \widehat{g}}(\xi) \big)
     \Big|
   \leq \sum_{\beta \leq \alpha}
          \binom{\alpha}{\beta} \,
          \big|
            \partial^\beta \, \widehat{g} (\xi)
          \big|
          \cdot \big|
                  \partial^{\alpha - \beta} \,
                  \overline{\strut \widehat{g}}(\xi)
                \big|
   \leq 2^{|\alpha|}
        \cdot \Big(
                \max_{|\alpha| \leq d+1} |\partial^{\alpha} \widehat{g} (\xi)|
              \Big)^2
 \]
 for any $\alpha \in \NN_0^d$ with $|\alpha| \leq d+1$.
 Since
 $S_j^{-1} S_i = A_j^{-1} A_i (\mybullet) + \, A_j^{-1} (b_i - b_j)$,
 it follows by the chain rule as in Lemma~\ref{lem:LinearChainRule}
 that, for any $\nu \in \NN_0^d$ with $|\nu| \leq d+1$,
 \begin{align*}
   |\det C_j|^{-1}
    \big|
           \partial^\nu |\Fourier g_{i,j}^{\localized}|^2 (\xi)
         \big|
   &\leq \delta^{-d}
           d^{|\nu|}
           \|A_j^{-1} A_i\|^{|\nu|}
           \max_{|\alpha| = |\nu|}
                  \Big|
                    (\partial^\alpha |\widehat{g}|^2)
                    \big( S_j^{-1} (S_i \xi) \big) \bigg| \\
   & \leq \delta^{-d}
           (2d)^{d+1}
           \max \big\{ 1 ,\, \|A_j^{-1} A_i\|^{d+1} \big\}
           \Big(
                  \max_{|\alpha| \leq d+1}
                    \big| (\partial^{\alpha} \widehat{g}) \big( S_j^{-1} (S_i \xi) \big) \big|
                \Big)^{2} \,
 \end{align*}
 for $\xi \in \RHat^d$.
 Using this, we can estimate the constant $M$ from
 Proposition~\ref{prop:MainTermSimplified} as follows:
 \begin{align*}
   M
   & = \sup_{i \in J} \,
         \sum_{j \in J}
           \bigg(\,
             |\det C_j|^{-1}
             \cdot \Big\|\,
                     \max_{|\nu| \leq d+1} \,
                       \big|\,
                         \partial^\nu | \Fourier g_{i,j}^{\localized} |^2
                       \,\big|
                   \,\Big\|_{L^{d+1} (Q_i ')}
           \,\bigg) \\
   & \leq \delta^{-d} \, (2d)^{d+1} \cdot
          \sup_{i \in J}
            \sum_{j \in J}
              \bigg[
                \max \big\{ 1 ,\, \|A_j^{-1} A_i\|^{d+1} \big\} \\
                & \quad \quad \quad \quad \quad \quad \quad \quad \quad \quad \quad \quad
                \cdot \big(
                        \int_{Q_i '}
                          \max_{|\alpha| \leq d+1}
                            \big|
                             (\partial^\alpha \widehat{g}) \big( S_j^{-1} (S_i \xi) \big)
                            \big|^{2 (d+1)}
                        \, d\xi
                      \big)^{1/(d+1)} \,
              \bigg] \\
  & = \delta^{-d} \, (2d)^{d+1} \cdot M_0 \, ,
 \end{align*}
 with $M_0$ as defined in the statement of the current proposition.

 By assumption, we have $A' \leq \sum_{j \in J} |\widehat{g} (S_j^{-1} \xi)|^2$, and thus
 \[
   t_0 (\xi)
   = \sum_{j \in J}
       | \det C_j |^{-1} \, |\widehat{g_j} (\xi)|^2
   = \delta^{-d} \cdot \sum_{j \in J} |\widehat{g}(S_j^{-1} \xi)|^2
   \geq A' \cdot \delta^{-d}
 \]
 for almost all $\xi \in \CalO$ and hence for almost all $\xi \in \RHat^d$.
 Therefore, Proposition~\ref{prop:MainTermSimplified} shows that $t_0$ is continuous on $\CalO$
 and tame, that the preceding estimate holds pointwise on $\CalO$, and that
 the operator $T_0 : \StandardDecompSp \to \StandardDecompSp$ is well-defined, bounded,
 and boundedly invertible with
 \[
   \|
     T_0^{-1}
   \|_{\StandardDecompSp \to \StandardDecompSp}
    \leq  (2d)^{(d+1)^2} C_d \cdot N_{\CalQ}^2 C_\Phi
          \cdot \Big[ \max_{|\alpha| \leq d+1} C_{\CalQ,\Phi,\alpha} \Big]
          \cdot (A')^{-1}
          \cdot \bigg( \frac{M_0}{A'} \bigg)^{d+1}
          \cdot \delta^{d} \, .
 \]
 This completes the proof.
\end{proof}

Our next aim is to present a simplified version of the technical Lemma~\ref{lem:RemainderSimplified}.
For this, we will use the following result
whose proof we postpone to Appendix~\ref{sub:ProofFactorizationLemma}.

\begin{lemma}\label{lem:FactorizationLemma}
  Let $g \in C^{d+1}(\RHat^d)$ be such that
  there exists a function  $\varrho : \RHat^d \to [0,\infty)$ satisfying
  \(
    |\partial^\alpha g (\xi)| \leq \varrho (\xi) \cdot (1 + |\xi|)^{-(d+1)}
  \)
  for all $\xi \in \RHat^d$ and $\alpha \in \NN_0^d$ with $|\alpha| \leq d+1$.
  Then, setting
  \[
    h_1 (\xi) := (1 + |\xi|^2)^{(d+1)/2} \cdot g (\xi), \quad
    \qquad
    h_2 (\xi)  := (1 + |\xi|^2)^{-(d+1)/2}
  \]
  we have $g = h_1 \cdot h_2$ on $\RHat^d$.
  Furthermore, $h_1,h_2 \in C^{d+1}(\RHat^d)$ satisfy the estimates
  \begin{align} \label{eq:h1h2_decay}
    \max_{|\alpha| \leq d+1}
      |\partial^\alpha h_2 (\xi)|
    \leq C' \cdot (1 + |\xi|)^{-(d+1)},
    \quad
    \max_{|\alpha| \leq d+1}
      |\partial^\alpha h_1 (\xi)|
    \leq C' \cdot \varrho (\xi)
  \end{align}
  for all $\xi \in \RHat^d$, where $C' := \big( 12 \cdot (d+1)^2 \big)^{d+1}$.
\end{lemma}

\begin{proposition}\label{prop:RemainderSimplifiedForStructuredSystems}
  Let $\CalQ = \big(S_j (Q_j ')\big)_{j \in J} = \big(A_j (Q_j ') + b_j)_{j \in J}$ be an
  affinely generated cover of an open set $\CalO \subset \RHat^d$ of full measure.
  Let $\Phi = (\varphi_j)_{j \in J}$ be a regular partition of unity subordinate to $\CalQ$,
  and let $w = (w_j)_{j \in J}$ be $\CalQ$-moderate.
  Let $\StandardGSI$ be such that $C_j := \delta \cdot A_j^{-t}$
  for some $\delta \in (0,1]$ and $g_j := |\det A_j|^{1/2} \cdot M_{b_j} [g \circ A_j^t]$
  for some $g \in L^1(\R^d) \cap L^2(\R^d)$ satisfying $\widehat{g} \in C^\infty (\RHat^d)$.
  Assume that the function $t_0$ defined in Equation~\eqref{eq:TAlphaDefinition} is tame.
  Assume that $\widetilde{Y} = (\widetilde{Y}_{i,j} )_{i, j \in J}$
  is of Schur-type, where
  \[
   \widetilde{Y}_{i,j} := K_{i,j} \cdot \int_{Q_i'}
                 (1+|S_j^{-1} (S_i \xi) |)^{d+1}
                  \max_{|\alpha| \leq d+1}
                                  |[\partial^\alpha \widehat{g}] (S_j^{-1} (S_i \xi))|
               \, d \xi ,
  \]
  with
  \[
    K_{i,j} := \max \Big\{ \frac{w_i}{w_j}, \frac{w_j}{w_i} \Big\}
               \Big(
                       \max \big\{ 1, |A_i^{-1} (b_i - b_j)| \big\}
                        \max \big\{ 1, \|A_i^{-1} A_j\| \big\}
                        \max \big\{ 1, \|A_j^{-1} A_i\|^2 \big\}
                     \Big)^{d+1}
  .\]
  Then the system $\StandardGSI$ is $(w,w,\Phi)$-adapted.
  Furthermore, for any $p,q \in [1,\infty]$,
  the operator $R_0 : \StandardDecompSp \to \StandardDecompSp$
  defined in Corollary~\ref{cor:RemainderTermR0} is well-defined and bounded, with
  \[
    \|R_0\|_{\StandardDecompSp \to \StandardDecompSp}
    \leq C_0 C_{p,q} (C')^4
         \| \Gamma_{\CalQ} \|_{\ell^q_w \to \ell^q_w}
         \cdot \delta^2 \cdot \| \widetilde{Y} \|_{\schur}^2,
  \]
  with $C_0$ as in \eqref{eq:RemainderSimplifiedSpecialConstant},
  $C'$ as in Lemma~\ref{lem:FactorizationLemma} and
  $C_{p,q} := 1$ if $\max\{p,q\} < \infty$ and
  $C_{p,q} := C_{\Phi} \| \Gamma_{\CalQ} \|^2_{\ell^q_w \to \ell^q_w}$, otherwise.
\end{proposition}

\begin{proof}
  To show that $\StandardGSI$ is $(w,w,\Phi)$-adapted,
  we use Proposition~\ref{prop:AmalgamMatrixEntriesEstimate}.
  Let us set ${v_j := w_j}$ for $j \in J$.
  Note that
  \(
    \Fourier g_{i,j}^{\localized}
    = \widehat{g_j} \circ S_i
    = |\det A_j|^{-1/2} \cdot \widehat{g} \circ S_j^{-1} \circ S_i
  \).
  An application of the chain rule as in Lemma~\ref{lem:LinearChainRule} shows,
  for any $\alpha \in \NN_0^d$ with $|\alpha| \leq d+1$, that
  \begin{align*}
    \big| \partial^\alpha [\Fourier g_{i,j}^{\localized}] (\xi) \big|
    & \leq |\det A_j|^{-1/2}
           \cdot d^{|\alpha|} \, \|A_j^{-1} A_i\|^{|\alpha|}
           \max_{|\beta| = |\alpha|}
             |(\partial^\beta \widehat{g}) \big( S_j^{-1} (S_i \xi) \big)| \\
    & \leq |\det A_j|^{-1/2}
           \cdot d^{d+1} \, \max \{ 1, \|A_j^{-1} A_i\|^{d+1} \}
           (1+|S_j^{-1} (S_i \xi)|)^{d+1} \\
           & \quad \quad \quad \quad \quad \quad \quad \quad \quad \quad \quad \quad \cdot
           \max_{|\alpha| \leq d+1}
             |(\partial^{\alpha} \widehat{g}) ( S_j^{-1} (S_i \xi)) | \, ,
  \end{align*}
  and hence
  \(
    \int_{Q_i'}
      \max_{|\alpha| \leq d+1}
        |\partial^\alpha [\Fourier g_{i,j}^{\localized}] (\xi)|
    \, d \xi
    \leq |\det A_j|^{-1/2}
         \cdot d^{d+1} \, \max \{ 1, \|A_j^{-1} A_i\|^{d+1} \}
         \cdot \widetilde{Y}_{i,j} K_{i,j}^{-1}
  \).
  Thus, the matrix entries $G_{i,j}$ introduced in Proposition~\ref{prop:AmalgamMatrixEntriesEstimate}
  satisfy
  \[
    G_{i,j}
    \leq \delta^{-d/2} d^{d+1}
         \max \bigg\{ \frac{w_i}{w_j}, \; \frac{w_j}{w_i} \bigg\} \,
         \max \{1, \|A_j^{-1} A_i\|^{d+1} \}
           (1 + \delta \|A_j^{-1} A_i\|)^d
          \frac{ \widetilde{Y}_{i,j}}{K_{i,j}}
    \leq C_{d,\delta} \cdot \widetilde{Y}_{i,j} \, ,
  \]
  for a suitable constant $C_{d,\delta} > 0$ which is independent of $i,j \in J$.
  Thus $\|G\|_{\schur} < \infty$.

  To finish the proof, we will show the claimed bound on
  $\|R_0\|_{\StandardDecompSp \to \StandardDecompSp}$.
  For this, we will apply Lemma~\ref{lem:RemainderSimplified} with the choices
  $I = J$, $B_j = A_j$, $c_j = b_j$ and $v_j = w_j$.
  In this setting, we have ${g_j^{\normalizedFactorisation} = g}$ for all $j \in J$.
  By defining
  \(
    \varrho :
    \RHat^d \to [0,\infty), \;
    \xi \mapsto (1+|\xi|)^{d+1} \max_{|\alpha| \leq d+1} | \partial^{\alpha} \widehat{g}(\xi)|
  \),
  we clearly have
  $|\partial^{\alpha} \widehat{g}(\xi) | \leq \varrho (\xi) \cdot (1 + |\xi| )^{-(d+1)}$
  for all $\xi \in \RHat^d$ and $\alpha \in \NN_0^d$ with $|\alpha| \leq d+1$.
  Hence, by Lemma~\ref{lem:FactorizationLemma}, we can factorize
  $\widehat{g} = h_1 \cdot h_2$ with $h_1, h_2 \in C^{d+1} (\RHat^d)$
  satisfying \eqref{eq:h1h2_decay}.
  This shows that the first hypothesis in Lemma~\ref{lem:RemainderSimplified} is satisfied,
  and it remains to show that the matrices ${Y = (Y_{i,j})_{i,j \in J}}$
  and ${Z = (Z_{i,j})_{i,j \in J}}$ of Lemma~\ref{lem:RemainderSimplified} are of Schur-type.
  For this, note that $|\det (B_j^t C_j)|^{-1} = |\det (A_j^t \delta A_j^{-t})|^{-1} = \delta^{-d}$
  and $\|C_j^t A_i\| = \delta \, \|A_j^{-1} A_i\| \leq \|A_j^{-1} A_i\|$, since $\delta \leq 1$.
  Therefore,
  \begin{align*}
    \max \{ \|C_j^t A_i\|, \|C_j^t A_i\|^{d+1} \}
     \leq \delta \|A_j^{-1} A_i\| \cdot \max \{ 1, \| A_j^{-1} A_i \|^d \}
      \leq \delta \max \{ 1, \| A_j^{-1} A_i \|^{d+1} \}
  \end{align*}
  for all $i,j \in I$.
  It is now readily verified that $Y_{i,j} \leq C' \cdot \delta \cdot \widetilde{Y}_{i,j}$
  and $Z_{i,j} \leq C' \cdot \delta \cdot \widetilde{Y}_{i,j}$ for $i,j \in J$,
  where $C'$ is as in Lemma~\ref{lem:FactorizationLemma}.
  Hence,
  \(
    \| Y \|_{\schur} \| Z \|_{\schur}
    \leq (C')^2 \cdot \delta^2 \cdot \| \widetilde{Y} \|_{\schur}^2
  \).
  Therefore, applying Lemma~\ref{lem:RemainderSimplified} completes the proof.
\end{proof}

The factor $\max\{1, |A_i^{-1} (b_i - b_j)|\}$ that appears in defining $K_{i,j}$ in
Proposition~\ref{prop:RemainderSimplifiedForStructuredSystems} can be inconvenient.
In particular, it does not appear in \cite{StructuredBanachFrames},
which makes it difficult to translate existing concrete examples from \cite{StructuredBanachFrames}
readily to the present setting.
For this reason, we supply the following.

\begin{lemma}\label{lem:TranslationRemoval}
  The matrix entries $\widetilde{Y}_{i,j}$ introduced in
  Proposition~\ref{prop:RemainderSimplifiedForStructuredSystems}
  satisfy $0 \leq \widetilde{Y}_{i,j} \leq (1 + R_\CalQ)^{d+1} \cdot \widehat{Y}_{i,j}$, where
  \[
    \widehat{Y}_{i,j}
    := L_{i,j}
       \cdot \int_{Q_i'}
               (1 + |S_j^{-1} (S_i \xi)|)^{2d+2}
               \max_{|\alpha| \leq d+1}
                 |(\partial^\alpha \widehat{g})(S_j^{-1} (S_i \xi))|
             \, d \xi
  \]
  and
  \(
    L_{i,j} := \max \big\{ \frac{w_i}{w_j}, \frac{w_j}{w_i} \big\}
               \big(
                 \max \{ 1, \|A_i^{-1} A_j\|^2 \}
                 \, \max \{ 1, \|A_j^{-1} A_i\|^3 \}
               \big)^{d+1}
  \)
  for $i,j \in J$.
\end{lemma}

\begin{proof}
  Since
  \(
    S_j^{-1} (S_i \xi)
    = A_j^{-1} (A_i \xi + b_i - b_j)
  \)
  for all $\xi \in \RHat^d$, it follows that
  \begin{align*}
    |A_i^{-1} (b_i - b_j)|
    & = |A_i^{-1} A_j A_j^{-1} (b_i - b_j)|
      \leq \|A_i^{-1} A_j\|
           \cdot \big(
                   |A_j^{-1} A_i \xi + A_j^{-1} (b_i - b_j)| + |A_j^{-1} A_i \xi|
                 \big) \\
    & \leq \|A_i^{-1} A_j\|
           \cdot \big(
                   |S_j^{-1} (S_i \xi)| + R_\CalQ \|A_j^{-1} A_i\|
                 \big) \\
    & \leq (1 + R_\CalQ)
           \cdot \max\{ 1, \|A_i^{-1} A_j\| \}
           \cdot \max \{ 1, \|A_j^{-1} A_i\| \}
           \cdot (1 + |S_j^{-1} (S_i \xi)|)
  \end{align*}
  for $\xi \in Q_i'$.
  Using this, the estimate $\widetilde{Y}_{i,j} \leq (1 + R_\CalQ)^{d+1} \cdot \widehat{Y}_{i,j}$
  follows directly from the definitions.
\end{proof}

\subsection{Invertibility of the frame operator}

The next result summarizes our criteria for the invertibility of the frame operator
obtained in this section.

\begin{theorem}\label{thm:MainSummary}
  Let $\CalQ = \big(S_j (Q_j ')\big)_{j \in J} = \big(A_j (Q_j ') + b_j)_{j \in J}$ be an
  affinely generated cover of an open set $\CalO \subset \RHat^d$ of full measure.
  Let $\Phi = (\varphi_j)_{j \in J}$ be a regular partition of unity subordinate to $\CalQ$,
  and let $w = (w_j)_{j \in J}$ be $\CalQ$-moderate.
  Suppose that
  \begin{itemize}
    \item[(i)]  The system $\StandardGSI$ is such that
                $g_j := |\det A_j|^{1/2} \cdot M_{b_j} [g \circ A_j^t]$
                and $C_j := \delta \cdot A_j^{-t}$ for some $\delta > 0$
                and some $g \in L^1(\R^d) \cap L^\infty(\R^d)$
                with $\widehat{g} \in C^\infty (\RHat^d)$;

    \item[(ii)] There is an $A' > 0$ such that $A' \leq \sum_{j \in J} |\widehat{g}(S_j^{-1} \xi)|^2$
                for almost all $\xi \in \CalO$;

    \item[(iii)] The matrix $\widehat{Y} = (\widehat{Y}_{i,j})_{i,j \in J}$ is of Schur-type,
                 where $\widehat{Y}_{i,j}$ as in Lemma~\ref{lem:TranslationRemoval};
                 \vspace{0.1cm}

    \item[(iv)] The term $M_0$ defined in Proposition~\ref{prop:MainTermForStructuredSystems}
                is finite.

  \end{itemize}
  Then the system $\StandardGSI$ is $(w,w,\Phi)$-adapted, and for $p,q \in [1,\infty]$,
   the frame operator
  ${S : \DecompSp(\CalQ,L^p,\ell_w^q) \to \DecompSp(\CalQ,L^p,\ell_w^q)}$
  associated to $\StandardGSI$ is well-defined and bounded.

  Finally, for given $p,q \in [1,\infty]$, let
  \(
    C_{d,\CalQ,w} := \max \big\{
                            [ \sup_{j \in J} \lambda(Q_j ') ]^{-\frac{3}{d+2}} ,
                            [\kappa_d K_{\CalQ,w}]^{1/(d+2)}
                          \big\}
  \),
  where
  \[
    \kappa_d := (2d)^{(d+1)^2}
                \big( 8d \big)^{2d+2}
                12^{5d + 5}
                \cdot (d+1)^{8d + 10}
                \cdot \frac{72 \cdot (d+1)^{5/2} \cdot 2^{d+2}}{\pi^{3d}}
                \cdot
                \bigg(
                  \frac{\frac{0.8}{e} (d+1)^2}{\ln(2+d)}
                \bigg)^{d+1}
  \]
  and
  \(
    K_{\CalQ,w}
    := \|\Gamma_\CalQ\|_{\ell_w^q \to \ell_w^q}^3
       N_{\CalQ}^2
       \max \{1, C_\Phi^2 \}
       (1 + R_\CalQ)^{3d + 4}
       \max_{|\alpha| \leq d+1} C_{\CalQ,\Phi,\alpha}^3 .
  \)
  Then, if $\delta > 0$ is chosen such that
  \begin{equation}
    C_{d,\CalQ,w}
    \cdot M_0^{\frac{d+1}{d+2}}
    \cdot \big( \|\widehat{Y}\|_{\schur}^2 \big)^{\frac{1}{d+2}}
    \cdot \frac{\delta}{A'}
    < 1 ,
    \label{eq:MainCondition}
  \end{equation}
  then the frame operator is also boundedly invertible as an operator on
  $\DecompSp(\CalQ,L^p,\ell_w^q)$.
\end{theorem}

\begin{proof}
  We proceed in two steps.
\\\\
  \textbf{Step 1.} Suppose that $\delta \leq 1$.
  Since $A' \leq \sum_{j \in J} |\widehat{g}(S_j^{-1} \xi)|^2$
  for almost all $\xi \in \CalO$, and since $M_0$ is finite, an application of
  Proposition~\ref{prop:MainTermForStructuredSystems} shows that
  $t_0$ is continuous on $\CalO$ and tame and that
  \(
    T_0 := \Phi_{t_0} : \DecompSp(\CalQ,L^p,\ell_w^q) \to \DecompSp(\CalQ,L^p,\ell_w^q),
  \)
  with $\Phi_{t_0}$ as defined in Proposition~\ref{prop:FourierMultiplierBound},
  is well-defined, bounded, and boundedly invertible, with
  \[
    \|T_0^{-1}\|_{\DecompSp(\CalQ,L^p,\ell_w^q) \to \DecompSp(\CalQ,L^p,\ell_w^q)}
    \leq C^{(1)} \cdot M_0^{d+1} \cdot (A')^{-(d+2)} \cdot \delta ^d
  \]
  for arbitrary $p,q \in [1,\infty]$.
  Here,
  \(
    C^{(1)}
    := (2d)^{(d+1)^2} \, C_{d} \, N_{\CalQ}^2 C_\Phi
       \cdot \max_{|\alpha| \leq d+1} C_{\CalQ,\Phi,\alpha}
  \),
  with $C_d$ as in Equation~\eqref{eq:MainTermSimplifiedSpecialConstant}.

  Lemma~\ref{lem:TranslationRemoval} shows that
  $\|\widetilde{Y}\|_{\schur} \leq (1 + R_\CalQ)^{d+1} \, \|\widehat{Y}\|_{\schur} < \infty$,
  with $\widetilde{Y}$ as in Proposition~\ref{prop:RemainderSimplifiedForStructuredSystems}.
  Therefore, Proposition~\ref{prop:RemainderSimplifiedForStructuredSystems} shows that the system
  $\StandardGSI$ is $(w,w,\Phi)$-adapted, and hence the frame operator
  $S : \DecompSp(\CalQ, L^p, \ell_w^q) \to \DecompSp(\CalQ, L^p, \ell_w^q)$ is well-defined
  and bounded for all $p,q \in [1,\infty]$ by Corollary~\ref{cor:FrameOperatorSpecialContinuity}.

  Lastly, it follows by
  Proposition~\ref{prop:RemainderSimplifiedForStructuredSystems}
  and Corollary~\ref{cor:RemainderTermR0}
  that the frame operator $S$ can be written as $S = T_0 + R_0$, where
  \[
    \|R_0\|_{\DecompSp(\CalQ,L^p,\ell_w^q) \to \DecompSp(\CalQ,L^p,\ell_w^q)}
    \leq C^{(2)} \cdot \delta^2 \cdot \|\widetilde{Y}\|_{\schur}^2
    \leq C^{(2)} \, (1 + R_\CalQ)^{2d+2} \cdot \delta^2 \cdot \|\widehat{Y}\|_{\schur}^2,
  \]
  where $C^{(2)} := C_0 C_{p,q} (C')^4 \|\Gamma_\CalQ\|_{\ell_w^q \to \ell_w^q}$,
  with $C_0$ as in \eqref{eq:RemainderSimplifiedSpecialConstant}
  and $C'$ as in Lemma~\ref{lem:FactorizationLemma},
  and with ${C_{p,q} := \max \{1, C_\Phi \} \cdot \|\Gamma_{\CalQ}\|_{\ell_w^q \to \ell_w^q}^2}$.
  Here, we used the easily verifiable estimate $\|\Gamma_{\CalQ}\|_{\ell_w^q \to \ell_w^q} \geq 1$.

  Therefore, for arbitrary $p,q \in [1,\infty]$, a combination of the above estimates gives
  \begin{align*}
    \|T_0^{-1}\|_{\op} \cdot \|R_0\|_{\op}
    & \leq C^{(1)} C^{(2)} (1 + R_\CalQ)^{2d+2}
           \cdot \delta^{2+d} \cdot \|\widehat{Y}\|_{\schur}^2
           \cdot M_0^{d+1} \cdot (A')^{-(d+2)} \\
    & =    \Big[
             \big(
               C^{(1)} C^{(2)} (1 + R_\CalQ)^{2d+2}
             \big)^{1/(d+2)}
             \cdot M_0^{\frac{d+1}{d+2}}
             \cdot (\|\widehat{Y}\|_{\schur}^2)^{\frac{1}{d+2}}
             \cdot \frac{\delta}{A'}
           \Big]^{d+2} \\
    & \leq \Big[
             C_{d,\CalQ, w}
             \cdot M_0^{\frac{d+1}{d+2}}
             \cdot (\|\widehat{Y}\|_{\schur}^2)^{\frac{1}{d+2}}
             \cdot \frac{\delta}{A'}
           \Big]^{d+2}
      <    1.
  \end{align*}
  Therefore, Lemma~\ref{lem:FrameOperatorInvertibility} implies that the frame operator
  $S = T_0 + R_0 : \mathcal{D}(\CalQ, L^p, \ell^q_w) \to  \mathcal{D}(\CalQ, L^p, \ell^q_w)$
  is boundedly invertible, as claimed.
\\\\
  \textbf{Step 2.}
  In this step it will be shown that \eqref{eq:MainCondition} already entails $\delta \leq 1$.
  To this end, first note that
  \(
    A'
    \leq \sum_{j \in J}
           |\widehat{g}(S_j^{-1} \eta)|^2
    \leq \big(
           \sum_{j \in J}
             |\widehat{g}(S_j^{-1} \eta)|
         \big)^2
  \),
  and hence
  $\sum_{j \in J} |\widehat{g}(S_j^{-1} \eta)| \geq \sqrt{A'}$ for almost every $\eta \in \CalO.$
  Thus, for any fixed $i \in J$,
  \begin{align*}
    \|\widehat{Y}\|_{\schur}
    & \geq \sum_{j \in J}
             \widehat{Y}_{i,j}
      \geq   \int_{Q_i'}
             \sum_{j \in J}
               |\widehat{g} (S_j^{-1} (S_i \xi))|
           \, d \xi
     \geq \int_{Q_i'} \sqrt{A'} \, d \xi
      =    \sqrt{A'} \cdot \lambda(Q_i ').
  \end{align*}
  Next, by applying Jensen's inequality, we see that the constant $M_0$ introduced in
  Proposition~\ref{prop:MainTermForStructuredSystems} satisfies, for each $i \in J$, the estimate
  \begin{align*}
    M_0
    & \geq \sum_{j \in J}
             \Big(
               \lambda(Q_i ')
               \int_{Q_i'}
                 |\widehat{g}(S_j^{-1} (S_i \xi))|^{2(d+1)}
               \, \frac{d \xi}{\lambda(Q_i ')}
             \Big)^{1/(d+1)} \\
             &
      \geq [\lambda(Q_i ')]^{1/(d+1) - 1}
           \sum_{j \in J}
             \int_{Q_i'}
               |\widehat{g}(S_j^{-1} (S_i \xi))|^{2}
             \, d \xi \\
    & =    [\lambda(Q_i ')]^{1/(d+1) - 1}
           \int_{Q_i'}
             \sum_{j \in J}
               |\widehat{g}(S_j^{-1} (S_i \xi))|^{2}
           \, d \xi \\
&      \geq [\lambda(Q_i ')]^{1/(d+1) - 1} \cdot A' \cdot \lambda(Q_i ') \\
&      =    A' \cdot [\lambda(Q_i ')]^{1/(d+1)}.
  \end{align*}
  Overall, we see that
  \[
    \kappa :=   M_0^{\frac{d+1}{d+2}} \cdot \big( \|\widehat{Y}\|_{\schur}^2 \big)^{\frac{1}{d+2}}
           \geq A' \cdot \sup_{i \in J} [\lambda(Q_i ')]^{\frac{3}{d+2}}
           \geq C_{d,\CalQ,w}^{-1} \, A'
  \]
  and hence $C_{d,\CalQ,w} \cdot \kappa \cdot \frac{\delta}{A'} \geq \delta$.
  Thus, if $\delta$ satisfies Equation~\eqref{eq:MainCondition}, then $\delta < 1$.
\end{proof}

\subsection{Proof of Theorem~\ref{thm:theorem_intro}}
\label{sec_proof_of_thm:theorem_intro}
Theorem \ref{thm:theorem_intro}, announced in the introduction,
is just a reformulation of Theorem~\ref{thm:MainSummary},
with the following identifications of notation:
$\newA=A'$; $\newB=B'$; $M_1 = \|\widehat{Y}\|_{\schur}$.
\hfill$\square$

\subsection{Banach frames and atomic decompositions}
We now remark that, under the assumptions of Theorem~\ref{thm:MainSummary},
the system $(\Translation{\delta A_j^{-t}k} g_j)_{j \in J, k \in \Z^d}$
forms a Banach frame and an atomic decomposition (\cite{GroechenigDescribingFunctions})
for the Besov-type spaces $\StandardDecompSp$, and, moreover,
\emph{the corresponding dual family is given by the canonical dual frame}.

\begin{corollary}\label{cor:BanachFramesAtomicDecompositions}
  Suppose that the assumptions of Theorem~\ref{thm:MainSummary} are satisfied,
  including the assumption \eqref{eq:MainCondition}.
  Then the system $(\Translation{\delta A_j^{-t} k} g_j)_{j \in J, k \in \Z^d}$
  forms a Banach frame and an atomic decomposition
  for all of the spaces $\StandardDecompSp$, $p,q \in [1,\infty]$, with associated coefficient
  space $Y_w^{p,q}$ as in Definition~\ref{def:CoefficientSpace}.
  Precisely, the analysis and synthesis maps
  \begin{align*}
    & \analysis : \StandardDecompSp \to Y_w^{p,q},
                  f \mapsto \big(
                              \langle f \mid \Translation{\delta A_j^{-t} k} g_j \rangle_{\Phi}
                            \big)_{j \in J, k \in \Z^d} \\[0.2cm]
    \text{and} \quad
    & \synthesis : Y_w^{p,q} \to \StandardDecompSp,
                 (c_j^{(k)})_{j \in J, k \in \Z^d}
                 \mapsto \sum_{j \in J}
                           \sum_{k \in \Z^d}
                             c_j^{(k)} \, \Translation{\delta A_j^{-t} k} g_j
  \end{align*}
  are well-defined and bounded, and satisfy
  \[
    (S^{-1} \circ \synthesis) \circ \analysis = \identity_{\StandardDecompSp}
  \quad
\text{and}
  \quad
    \synthesis \circ (\analysis \circ S^{-1}) = \identity_{\StandardDecompSp}
  .\]
\end{corollary}

\begin{proof}
  Theorem~\ref{thm:MainSummary} shows that
  $(\Translation{\delta A_j^{-t} k} g_j)_{j \in J, k \in \Z^d}$ is $(w,w,\Phi)$-adapted.
  Thus, the boundedness of $\analysis, \synthesis$
  follows from Proposition~\ref{prop:AnalysisSynthesisOperatorGeneral}.
  The remaining statements follow from the invertibility of ${S = \synthesis \circ \analysis}$
  proven in Theorem~\ref{thm:MainSummary}.
\end{proof}

\subsection{An example}
We conclude with an example verifying the hypotheses of Theorem~\ref{thm:MainSummary}
for Besov-type spaces associated with covers that have a geometry which is in a certain sense
intermediate between the geometry of the uniform and the dyadic covers.
These covers are an instance of the \emph{non-homogeneous isotropic covers}
from \cite[Section 2.5]{TriebelFourierAnalysisAndFunctionSpaces} and \cite[Section 2.1]{tr78};
the corresponding spaces are also known as \emph{$\alpha$-modulation spaces}
\cite{GroebnerAlphaModulationSpaces}.
For similar calculations of other concrete examples,
we refer to \cite{StructuredBanachFrames}.

For fixed $\alpha \in [0,1)$, the \emph{$\alpha$-modulation space} with parameters
$p,q \in [1,\infty]$ and $s \in \R$ is defined as
$M_{p,q}^{s,\alpha} (\RR^d) := \DecompSp (\CalQ^{(\alpha)}, L^p, \ell_{w^{(s,\alpha)}}^q)$,
where the cover $\CalQ^{(\alpha)}$ of $\RHat^d$ is given by
\[
  \CalQ^{(\alpha)}
  := \big( A_j^{(\alpha)} Q + b_j^{(\alpha)} \big)_{j \in \ZZ^d \setminus \{0\}}
,\]
where $
  A_j^{(\alpha)} := |j|^{\alpha_0} \, \identity_{\RR^d}, \;
  b_j^{(\alpha)} := |j|^{\alpha_0} \, j,
$ and$
  Q = B_r (0) ,
$
with $\alpha_0 := \tfrac{\alpha}{1-\alpha}$ and  $r \geq r_0 = r_0 (d,\alpha)$.
Under this assumption on $r$, one can show that $\CalQ^{(\alpha)}$ is indeed
an affinely generated cover of $\RHat^d$; see
\cite[Theorem 2.6]{BorupNielsenAlphaModulationSpaces} and \cite[Lemma 7.3]{StructuredBanachFrames}.
Finally, the weight $w^{(s,\alpha)}$ is given by
\(
  w^{(s,\alpha)}_j
  = |j|^{s / (1 - \alpha)}
\)
for $j \in \ZZ^d \setminus \{0\}$.
In the following, we will simply write $\CalQ$, $A_j$, and $b_j$
for $\CalQ^{(\alpha)}$, $A_j^{(\alpha)}$, and $b_j^{(\alpha)}$
and fix some $r \geq r_0(d,\alpha)$.

Fix $s_0 \geq 0$. In the following, we will only consider ``smoothness parameters'' $s \in [-s_0, s_0]$.
Take ${g \in L^1(\RR^d) \cap L^2(\R^d)}$ such that $\widehat{g} \in C^\infty(\RHat^d)$,
and assume that there are $c,C > 0$ and $N > 0$ such that
\begin{equation}
  |\widehat{g} (\xi)| \geq c \quad \forall \, |\xi| \leq r
  \qquad \text{and} \qquad
  \max_{|\alpha| \leq d+1}
    |\partial^\alpha \widehat{g} (\xi)|
  \leq C \cdot (1 + |\xi|)^{-N}
  \quad \forall \, \xi \in \RHat^d .
  \label{eq:AlphaModulationDecayAssumption}
\end{equation}

We will determine conditions on $N$
(depending on $d, \alpha, s_0$) which ensure that the prerequisites
of Theorem~\ref{thm:MainSummary} are satisfied.
In fact, it will turn out that it is enough if $N > 4d + 3 + \tau$
where $\tau := \frac{4 \alpha d + 3 \alpha + s_0}{1-\alpha} \in [0,\infty)$.

To show this, note because of $Q_i ' = B_r (0)$ for all $i \in \ZZ^d \setminus \{0\}$ that
\[
  S_j^{-1} (S_i Q_i ')
  = B_{R_{i,j}} (\xi_{i,j})
  \quad \text{where} \quad
  R_{i,j} = \big( |i| / |j| \big)^{\alpha_0} \cdot r
  \quad \text{and} \quad
  \xi_{i,j} = \big( |i| / |j| \big)^{\alpha_0} \cdot i - j.
\]
Thus, applying the change of variables $\eta = S_j^{-1} (S_i \xi)$,
combined with the estimate \eqref{eq:AlphaModulationDecayAssumption},
yields
\begin{align*}
  Z_{i,j}
  & := \int_{Q_i'}
         (1 + |S_j^{-1} (S_i \xi)|)^{2d+2} \,
         \max_{|\alpha| \leq d+1}
           |[\partial^\alpha \widehat{g}] \big( S_j^{-1} (S_i \xi) \big)|
       \, d \xi \\
  & = \bigg( \frac{|j|}{|i|} \bigg)^{d \alpha_0} \!\!
      \int_{B_{R_{i,j}} (\xi_{i,j})} \!\!
        (1 + |\eta|)^{2d + 2} \,
        \max_{|\alpha| \leq d+1}
          |[\partial^\alpha \widehat{g}] (\eta)|
      \, d \eta \\
      &
    \leq C \, \bigg( \frac{|j|}{|i|} \bigg)^{d \alpha_0} \!\!
         \int_{B_{R_{i,j}} (\xi_{i,j})} \!\!
           (1 + |\eta|)^{2d + 2 - N}
         \, d \eta.
\end{align*}
A similar computation shows
\begin{align*}
  W_{i,j} & := \bigg(
       \int_{Q_i '}
         \max_{|\alpha| \leq d+1}
           \big|
             (\partial^\alpha \widehat{g}) \big( S_j^{-1} (S_i \xi) \big)
           \big|^{2(d+1)}
       \, d \xi
     \bigg)^{\frac{1}{d+1}} \\
     &
  \leq C^2 \,
       \bigg(
         \bigg(
           \frac{|j|}{|i|}
         \bigg)^{d \alpha_0} \!\!
         \int_{B_{R_{i,j}} (\xi_{i,j})} \!\!
           (1 + |\eta|)^{-2N(d+1)}
         \, d \eta
       \bigg)^{\frac{1}{d+1}} \!\! .
\end{align*}
Using the notations
\[
  \Lambda_{i,j}^{[M,\tau]}
  := \bigg(
       \int_{B_{R_{i,j}} (\xi_{i,j})}
         (1 + |\eta|)^{-M}
       \, d \eta
     \bigg)^{\tau}
  \qquad \text{and} \qquad
  \Xi_{i,j}^{[k,M,\tau]} := \bigg( \frac{|j|}{|i|} \bigg)^{k} \cdot \Lambda_{i,j}^{[M,\tau]}
\]
for $i,j \in \ZZ^d \setminus \{0\}$ and $k, M \in \RR$, $\tau \in (0,\infty)$,
we have thus shown
\begin{equation}
  Z_{i,j} \leq C \cdot \Xi_{i,j}^{[d\alpha_0, N - 2d - 2, 1]}
  \quad \text{and} \quad
  W_{i,j} \leq C^2 \cdot \Xi_{i,j}^{\left[ \frac{d \alpha_0}{d+1}, 2N(d+1), \frac{1}{d+1} \right]} .
  \label{eq:AlphaModulationEstimateUsingLambda}
\end{equation}
This is useful, since
\cite[Equation (7.13)]{StructuredBanachFrames}
shows for $M \geq d + 1$ that
\begin{equation}
  \Xi_{i,j}^{[k,M,\tau]}
  \leq C' \cdot (1 + |j - i|)^{|k| + \tau (d + 1 - M)}
  \quad \forall \, i,j \in \ZZ^d \setminus \{0\} ,
  \label{eq:AlphaModulationNiceXiEstimate}
\end{equation}
where $C' = C'(\alpha, d, M, r, \tau, |k|)$.

Now, using that $w_{j}^{(s,\alpha)} = |j|^{s / (1 - \alpha)}$ and
$A_j = |j|^{\alpha_0} \, \identity$, a straightforward computation shows that
the quantity $L_{i,j}$ introduced in Lemma~\ref{lem:TranslationRemoval} satisfies
\[
  L_{i,j}
  = \begin{cases}
      \big( |j| / |i| \big)^{2(d+1)\alpha_0 + \frac{|s|}{1-\alpha}}
      & \text{if } |i| \leq |j|, \\
      \big( |j| / |i| \big)^{-3(d+1)\alpha_0 - \frac{|s|}{1-\alpha}}
      & \text{if } |i| >    |j|
    \end{cases}
  \leq \max \Big\{
              \big( |j| / |i| \big)^{\sigma},
              \big( |j| / |i| \big)^{-\sigma}
            \Big\} ,
\]
where we introduced $\sigma := \frac{3 \alpha (d + 1) + s_0}{1-\alpha} \in [0,\infty)$.
In combination with Equations~\eqref{eq:AlphaModulationEstimateUsingLambda}
and \eqref{eq:AlphaModulationNiceXiEstimate},
we thus see that the matrix elements $\widehat{Y}_{i,j}$ introduced in
Lemma~\ref{lem:TranslationRemoval} satisfy
\begin{align*}
  0
 & \leq \widehat{Y}_{i,j}
   =    L_{i,j} \, Z_{i,j}
    \leq C \cdot \max \big\{
                        \big( |j| / |i| \big)^{\sigma},
                        \big( |j| / |i| \big)^{-\sigma}
                      \big\}
           \cdot \Xi_{i,j}^{[d \alpha_0, N-2d-2,1]} \\
  & =    C \cdot \max \big\{
                        \Xi_{i,j}^{[\sigma + d \alpha_0, N-2d-2,1]},
                        \Xi_{i,j}^{[d \alpha_0 - \sigma, N-2d-2, 1]}
                      \big\} \\
                      &
    \leq C \cdot C_1 \cdot (1 + |j - i|)^{\sigma + d \alpha_0 + d+1 - (N-2d-2)} \\
  & =    C \cdot C_1 \cdot (1 + |j - i|)^{\sigma + d \alpha_0 + 3(d+1) - N} ,
\end{align*}
where $C_1 = C_1 (d, \alpha, N, r, s_0)$.
From this, it is easy to see that $\|\widehat{Y}\|_{\schur} \leq C \cdot C_2 < \infty$,
provided that $N > 4d + 3 + \sigma + d \alpha_0 = 4d + 3 + \tau$, where $C_2 = C_2 (d,\alpha,N,r,s_0)$.
We have thus verified condition (iii) of Theorem~\ref{thm:MainSummary}.

Next, we show that $M_0 < \infty$ for $M_0$ as defined in
Proposition~\ref{prop:MainTermForStructuredSystems}.
The same arguments as for estimating $\widehat{Y}_{i,j}$ give
\begin{align*}
  V_{i,j}
  & := \max \big\{ 1, \|A_j^{-1} A_i\|^{d+1} \big\} W_{i,j} \\
  &
    \leq C^2 \, \max \Big\{
                       \Xi_{i,j}^{\left[ \frac{d \alpha_0}{d+1}, 2N(d+1), \frac{1}{d+1} \right]},
                       \Xi_{i,j}^{\left[
                                    \alpha_0 \left( \frac{d}{d+1} - (d+1) \right), 2N(d+1), 1/(d+1)
                                  \right]}
                     \Big\} \\
  & \leq C^2 \cdot C_3 \cdot (1 + |j - i|)^{\alpha_0 \frac{d^2 + d + 1}{d+1}
                                            + \frac{1}{d+1} (d+1 - 2N(d+1))} \\
&    \leq C^2 \cdot C_3 \cdot (1 + |j - i|)^{1 + \alpha_0 \cdot (d+1) - 2N},
\end{align*}
where $C_3 = C_3(\alpha, d, N, r)$.
From this, we see that the constant $M_0$ introduced in
Proposition~\ref{prop:MainTermForStructuredSystems} satisfies
$M_0 = \|V\|_{\schur} \leq C^2 C_4 < \infty$ for a constant
$C_4 = C_4(\alpha, d, N, r)$, as soon as $N > \frac{1 + d}{2} (1 + \alpha_0)$,
which is implied by $N > 4d + 3 + \sigma + d \alpha_0$.
Thus, condition (iv) of Theorem~\ref{thm:MainSummary} is satisfied.

Lastly, we verify condition (ii) of Theorem~\ref{thm:MainSummary},
that is, $\sum_{j \in \ZZ^d \setminus \{0\}} |\widehat{g} (S_j^{-1} \xi)|^2 \geq A'$
for all $\xi \in \RHat^d$, where $A' := c^2$, with $c > 0$ as in
Equation~\eqref{eq:AlphaModulationDecayAssumption}.
To see this, note that Equation~\eqref{eq:AlphaModulationDecayAssumption}
implies $|\widehat{g}|^2 \geq c^2 \, \Indicator_{Q}$, where we recall $Q = B_r (0)$.
Hence, $|\widehat{g}(S_j^{-1} \xi)|^2 \geq c^2 \Indicator_{Q_j}$, since $Q_j = S_j Q$.
Finally, since $\CalQ^{(\alpha)} = (Q_j)_{j \in \ZZ^d \setminus \{0\}}$ is a cover
of $\RHat^d$, we see $\sum_{j \in \ZZ^d \setminus \{0\}} |\widehat{g}(S_j^{-1} \xi)|^2 \geq c^2 = A'$,
as claimed.

\appendix
\section{Estimation of the \texorpdfstring{$\Fourier L^1$}{𝓕L¹} norm}
\label{sec:TechnicalResults}

\subsection{Sobolev embeddings}

In this appendix we give an explicit bound for the constant implied in the estimate
$\|\Fourier^{-1} f\|_{L^1} \lesssim \max_{|\alpha| \leq d+1} \|\partial^\alpha f\|_{L^1}$.
Similar, but more qualitative results in the non-commutative context
can be found in \cite{MR2590916,MR3219221}.

\begin{lemma}\label{lem:BoxPolarCoordinateSubstitute}
  Let $d \in \NN$ and $\alpha, c > 0$.
  Define
  \(
    g : \RR^d \to     (0,\infty),
        x     \mapsto \big( \max \{c, \|x\|_{\ell^\infty} \} \big)^{-\alpha}.
  \)
  Then $\int_{\RR^d} g(x) \, dx < \infty$ if and only if $\alpha > d$,
  and in this case
  \[
    \int_{\RR^d} g(x) \, dx
    = \frac{2^d}{1 - \frac{d}{\alpha}} \cdot c^{d-\alpha} \, .
  \]
\end{lemma}

\begin{proof}
  Let $\mu$ denote the Lebesgue measure on $\RR^d$.
  We will use \cite[Proposition~6.24]{FollandRA}, which shows for
  measurable $f : \RR^d \to \CC$ that
  \[
    \int_{\RR^d} | f | \, d\mu
    = \int_0^\infty \lambda_f(\beta) \, d\beta \, ,
  \]
  where
  \(
    \lambda_f (\beta)
    := \mu \big( \{x \in \RR^d \,:\, | f(x) | > \beta \} \big) \, .
  \)
  To compute the distribution function $\lambda_g$,  first note that
  $g(x) \leq c^{-\alpha}$ for all $x \in \RR^d$, and thus $\lambda_g(\beta) = 0$
  for $\beta \geq c^{-\alpha}$.
  For $0 < \beta < c^{-\alpha}$, note that $g(x) > \beta$ is equivalent to
  $\| x \|_{\ell^{\infty}} < \beta^{-1/\alpha}$,
  whence to $x \in B^{\|\cdot\|^{\ell^{\infty}}}_{\beta^{-1/\alpha}} (0)$.
  Therefore, for any $\beta \in (0, c^{-\alpha})$, we compute
  \(
    \lambda_g (\beta)
    = \mu \big( B_{\beta^{-1/\alpha}}^{\| \mybullet \|_{\ell^\infty}} (0) \big)
    = (2 \cdot \beta^{-1/\alpha})^d \, ,
  \)
  and thus
  \[
    \int_{\RR^d} g(x) \, dx
    = \int_0^\infty \lambda_g (\beta) \, d\beta
    = 2^d \cdot \int_0^{c^{-\alpha}} \beta^{-d/\alpha} \, d \beta \, ,
  \]
  which is finite if and only if $d/\alpha < 1$.
  In the latter case, a direct calculation shows that
  \[
    \int_{\RR^d} g(x) \, dx
    = 2^d \cdot \frac{\beta^{1 - \alpha^{-1} d}}{1 - \frac{d}{\alpha}}
                \bigg|_{\beta=0}^{c^{-\alpha}}
    = \frac{2^d}{1 - \frac{d}{\alpha}} \cdot (c^{-\alpha})^{1 - \alpha^{-1} d}
    = \frac{2^d}{1 - \frac{d}{\alpha}} \cdot c^{d - \alpha} \, ,
  \]
  yielding the desired result.
\end{proof}

The following result provides the announced estimate.
For this, we use the usual \emph{Sobolev space}
\[
  W^{k,1} (\RR^d)
  := \left\{
       f \in L^p (\RR^d)
       \; : \;
       \partial^{\alpha} f \in L^p (\RR^d)
       \,\, \forall \, \alpha \in \NN_0^d \text{ with } |\alpha| \leq k
     \right\},
\]
with norm $\|f\|_{W^{k,1}} := \sum_{|\alpha| \leq k} \| \partial^\alpha f \|_{L^1}$.

\begin{lemma}\label{lem:FourierDecayViaSmoothness}
  Suppose $f \in W^{d+1,1}(\RHat^d)$.
  Then $\Fourier^{-1} f \in L^1(\RR^d)$ with
  \[
    \| \Fourier^{-1} f \|_{L^1}
    \leq \frac{d+1}{\pi^d}
         \cdot \max_{\theta \in \mathbb{I}} \| \partial^\theta f \|_{L^1} ,
  \]
 where $\mathbb{I} := \{0\}\cup\{(d+1) e_\ell \with \ell\in\underline{d} \} \subset \NN_0^d$,
  with $(e_k)_{k =1}^d$ denoting the standard basis of $\RR^d$.
\end{lemma}

\begin{proof}
  Since $\Schwartz(\RHat^d) \subset W^{d+1,1}(\RHat^d)$ is dense
  (see e.g.\ \cite[E10.8]{AltFunctionalAnalysis}),
  and since ${\Fourier^{-1} f_n \to \Fourier^{-1} f}$ uniformly if
  $f_n \to f$ in $W^{d+1,1}(\RHat^d) \hookrightarrow L^1(\RHat^d)$,
  it suffices---in view of Fatou's lemma---to prove the estimate for ${f \in \Schwartz(\RHat^d)}$.
  In this case, elementary properties of the Fourier transform yield
  for all $\alpha \in \NN_0^d$ and $ x \in \RR^d$ the estimate
  \[
    | x^\alpha \cdot \Fourier^{-1}f (x) |
    = (2\pi)^{-| \alpha |} \cdot | \Fourier^{-1} (\partial^\alpha f) (x) |
    \leq (2\pi)^{-| \alpha |} \cdot \| \partial^\alpha f \|_{L^1}
    .
  \]
  Next, using the auxiliary function
  \({
    g : \RR^d \to (0,\infty), \;
    x \mapsto (\max \{(2\pi)^{-1}, \|x\|_{\ell^\infty} \})^{-(d+1)}
  }\),
  it follows that
  \begin{align}
    | \Fourier^{-1} f(x) |
    & = g(x)
      \cdot \max \{ (2\pi)^{-(d+1)} , \|x\|_{\ell^\infty}^{d+1}\}
      \cdot | \Fourier^{-1} f(x) | \nonumber\\
    & = g(x) \cdot \max \left\{
                           (2\pi)^{-(d+1)} \cdot | \Fourier^{-1} f(x) |, \quad
                           \max_{\ell \in \underline{d}}
                             \big| x_\ell^{d+1} \cdot \Fourier^{-1} f(x) \big|
                        \right\} \nonumber\\
    & \leq g(x) \cdot \max \left\{
                             (2\pi)^{-(d+1)} \cdot \| f \|_{L^1}, \quad
                             \max_{\ell \in \underline{d}}
                               \big[
                                 (2\pi)^{-(d+1)}
                                 \cdot \| \partial_\ell^{d+1} f \|_{L^1}
                               \big]
                           \right\} \nonumber\\
    & \leq g(x) \cdot (2\pi)^{-(d+1)} \cdot
           \max_{\theta \in \mathbb{I}} \| \partial^\theta f \|_{L^1} .
  \label{eq:FourierDecayViaSmoothnessFundamental}
  \end{align}
 Hence, it remains to compute the integral $\int_{\RR^d} g(x) \, dx$.
 For this, note that an application of Lemma~\ref{lem:BoxPolarCoordinateSubstitute}
 (with $c = (2 \pi)^{-1}$ and $\alpha = d+1$) gives
 \(
   \int g(x) \, dx
   = \frac{2^d}{1 - \alpha^{-1} d} \cdot c^{d - \alpha}
   = 2^{d+1}\pi \cdot (d+1)
 \),
 and thus
 \[
   \| \Fourier^{-1} f \|_{L^1}
   \leq 2^{d+1} \pi \cdot (d+1) \cdot (2\pi)^{-(d+1)}
        \cdot \max_{\theta \in \mathbb{I}} \| \partial^\theta f \|_{L^1}
   = \frac{d+1}{\pi^d}
     \cdot \max_{\theta \in \mathbb{I}} \| \partial^\theta f \|_{L^1} \, ,
 \]
 which completes the proof.
\end{proof}

\subsection{The chain rule}

Lemma~\ref{lem:FourierDecayViaSmoothness} allows to estimate the $\Fourier L^1$ norm of $f$ in terms
of the $L^1$ norms of certain derivatives of $f$.
In many cases, we will have $f = g \circ A$, where we have good control
over the derivatives of $g$.
In such cases, the following lemma will be helpful.

\begin{lemma}\label{lem:LinearChainRule}
  (\cite[Lemma~2.6]{DecompositionIntoSobolev})

  Let $d,k \in \NN$, $A \in \RR^{d \times d}$, and $f \in C^k(\RR^d)$
  be arbitrary.
  Let $(e_1,\dots,e_d)$ denote the standard basis of $\RR^d$,
  let $i_1,\dots,i_k \in \underline{d}$, and define
  $\alpha := \sum_{m=1}^k e_{i_m} \in \NN_0^d$.

  Then $|\alpha| = k$, and
  \begin{equation}
    \partial^\alpha (f \circ A)(x)
    = \sum_{\ell_1,\dots,\ell_k \in \underline{d}}
        [
         A_{\ell_1,i_1} \cdots A_{\ell_k,i_k}
         \cdot (\partial_{\ell_1} \cdots \partial_{\ell_k} f) (A \, x)
        ]
    \qquad \forall \, x \in \RR^d \, .
    \label{eq:LinearChainRule}
  \end{equation}
\end{lemma}

\subsection{The norm of a reciprocal}

\begin{lemma}\label{lem:DerivativesOfReciprocal}
  Let $m \in \NN$ and let $U \subset \RR$ be open.
  Suppose that $f \in C^m (U)$ never vanishes on $U$.
  Let $A > 0$, $K \geq 0$, and $x_0 \in U$ be such that
  \[
    |f (x_0)| \geq A^{-1}
    \quad \text{and} \quad
    |f^{(\ell)}(x_0)| \leq K
    \quad \forall \, 1 \leq \ell \leq m \, .
  \]
 Then the reciprocal $F := 1 / f$ of $f$ satisfies
 \[
    \bigg|\frac{d^\ell}{d x^\ell}\Big|_{x = x_0} F (x)\bigg|
    \leq C_m \cdot A \cdot \max \big\{ AK, \, (A K)^\ell \big\}
 \]
 for all $1 \leq \ell \leq m$, where the constant $C_m$ satisfies,
 for all $1 \leq \ell \leq m$,
 \[
   1 \leq C_m
     \leq 3 \sqrt{m}
            \cdot \Big(
                    \frac{\frac{0.8}{e} \cdot m^2}{\ln(1+m)}
                  \Big)^m \, .
 \]
\end{lemma}

\begin{proof}
  Setting $g : \RR \setminus \{0\} \to \RR, t \mapsto t^{-1}$,
  we have $F = g \circ f$.
  Therefore, the ``set partition version'' of \emph{Faa di Bruno's formula},
  see for instance \cite[p. 219]{FaaDiBruno}, shows for $1 \leq \ell \leq m$
  that
  \[
    F^{(\ell)} (x_0)
    = \sum_{\pi \in P_\ell}
      \Big[
        g^{(|\pi|)} (f(x_0))
        \prod_{B \in \pi} f^{(|B|)} (x_0)
      \Big] \, ,
  \]
  where $P_\ell \subset 2^{2^{\underline{\ell}}}$ denotes the sets
  of all \emph{partitions} of the set $\underline{\ell} := \{1,\dots,\ell\}$.
  Phrased differently, the set $P_\ell$ contains exactly those subsets
  $\pi \subset 2^{\underline{\ell}}$ of the power set $2^{\underline{\ell}}$
  for which $\underline{\ell} = \biguplus \pi$ and $B \neq \emptyset$ for all
  $B \in \pi$.
  For each $\pi \in P_\ell$, we denote by $|\pi|$ the number of blocks of the
  partition determined by $\pi$; that is, $|\pi|$ is the number of elements
  of $\pi$. Likewise, for a block $B \in \pi$, we denote by $|B|$ the size
  of the block, that is, the number of elements of $B$.

  An induction argument shows that
  $g^{(k)}(t) = (-1)^k \cdot k! \cdot t^{-(k+1)}$ for all $k \in \NN_0$.
  Therefore, for arbitrary $\pi \in P_\ell$, it follows that
  $|g^{(|\pi|)} (f(x_0))|
   = |\pi|! \cdot |f(x_0)|^{-(1 + |\pi|)}
   \leq \ell! \cdot A^{1 + |\pi|}$, since any $\pi \in P_\ell$ satisfies
  $\ell = \sum_{B \in \pi} |B| \geq \sum_{B \in \pi} 1 = |\pi|$.
  Similarly, it follows that
  \[
    \prod_{B \in \pi}
      \big| f^{(|B|)} (x_0) \big|
    \leq \prod_{B \in \pi} K
    =    K^{|\pi|}
  \]
 for all $\pi \in P_\ell$.
 Combining these observations shows that
  \[
    |F^{(\ell)}(x_0)|
    \leq \sum_{\pi \in P_\ell}
           \big( \ell! \cdot A \cdot (AK)^{|\pi|} \big)
    \leq A \cdot \max \big\{AK, (AK)^\ell \big\} \cdot \ell! \cdot |P_\ell| \, ,
  \]
  where we used again that $1 \leq |\pi| \leq \ell$ for $\pi \in P_\ell$.
  Since $\ell! \leq m!$ and $|P_\ell| \leq |P_m|$
  for $\ell \leq m$,
  it suffices to show that $C_m := m! \cdot |P_m|$
  satisfies the bound stated in the lemma. Here,
  the cardinalities $|P_m|$ are the so-called \emph{Bell numbers}.
  For these, \cite[Theorem~2.1]{BellNumberBound} provides the bound
  $|P_m| \leq \big(\frac{0.8 \cdot m}{\ln(1+m)}\big)^m$.
  Furthermore, the version of Stirling's formula derived in
  \cite{RobbinsStirlingsFormula} shows that
  \[
    m! \leq \sqrt{2\pi} \cdot e^{1/12} \cdot (m/e)^m \cdot \sqrt{m}
       \leq 3 \cdot (m/e)^m \cdot \sqrt{m} \, .
  \]
  Combining these estimates gives the desired result.
\end{proof}

\section{Proof of Proposition~\ref{prop:Density}}
\label{sub:DensityResultProof}
(i) Let $f \in \DenseSpace (\RR^d)$ and
  set $K := \supp \widehat{f} \subset \CalO$.
  For $i \in I$, the set
  ${U_i := \varphi_i^{-1}(\CC \setminus \{0\}) }$ is open.
  Moreover, since $\sum_{i \in I} \varphi_i \equiv 1$ on $\CalO$,
  it follows that $\CalO = \bigcup_{i \in I} U_i$.
  By compactness of $K$, there exists a finite subset $I_K \subset I$
  satisfying $K \subset \bigcup_{\ell \in I_K} U_\ell \subset \bigcup_{\ell \in I_K} Q_\ell$.
  Therefore, for any $i \in I$ satisfying $Q_i \cap K \neq \emptyset$, necessarily
  $\emptyset \neq Q_i \cap K \subset Q_i \cap \bigcup_{\ell \in I_K} Q_\ell$,
  and hence $i \in I_K^\ast := \bigcup_{\ell \in I_K} \ell^\ast$, which is a finite subset of $I$.
  By contraposition,  we have $Q_i \cap K = \emptyset$, and hence
  $\varphi_i \cdot \widehat{f} \equiv 0$, for all $i \in I \setminus I_K^\ast$.

  Next, for each $i \in I_K^\ast$, clearly
  $\varphi_i \cdot \widehat{f} \in C_c^\infty(\CalO) $,
  and thus $\|\Fourier^{-1} (\varphi_i \cdot \widehat{f}\, ) \|_{L^p} < \infty$.
  Therefore, setting
  \(
    M := \max_{i \in I_K^\ast}
           \|\Fourier^{-1} (\varphi_i \cdot \widehat{f} \, ) \|_{L^p}
      <  \infty
  \)
  gives
  \[
    \|f\|_{\DecompSp(\CalQ,L^p,\ell_w^q)}
    = \Big\|
        \big(
          \|
            \Fourier^{-1} (\varphi_i \cdot \widehat{f} \, )
          \|_{L^p}
        \big)_{i \in I}
      \Big\|_{\ell_w^q}
    \leq \big\|
           (
             M
           )_{i \in I_K^\ast}
         \big\|_{\ell_w^q}
    < \infty ,
  \]
  which shows that $f \in \DecompSp(\CalQ,L^p,\ell_w^q)$.
\\\\
(ii) Let $p, q \in [1,\infty)$.
  Recall the notation $C_\Phi = \sup_{i \in I} \|\Fourier^{-1} \varphi_i\|_{L^1}$
  from Definition~\ref{def:BAPU}.
  Let $f \in \DecompSp(\CalQ, L^p, \ell_w^q)$ and $\eps > 0$ be arbitrary.
  Note ${c = (c_i)_{i \in I} \in \ell^q_w (I)}$, where
  $c_i := \|\Fourier^{-1} (\varphi_i \cdot \widehat{f} \,)\|_{L^p}$
  for $i \in I$.
  Since $\|c\|_{\ell_w^q}\vphantom{\sum_{j_i}} = \|f\|_{\DecompSp(\CalQ,L^p,\ell_w^q)} < \infty$
  and since $q < \infty$, there exists a finite set $I_0 = I_0 (\eps,f) \subset I$
  such that the sequence
  $\vphantom{\sum^L}\widetilde{c} := c \cdot \Indicator_{I \setminus I_0}$ satisfies
  \[
    \|\, \widetilde{c} \,\|_{\ell_w^q}
    < \Big(
        C_\Phi \cdot \|\Gamma_\CalQ\|_{\ell_w^q \to \ell_w^q}
      \Big)^{-2}
      \cdot \frac{\eps}{2} \, .
  \]
  For each $i \in I_0^\ast := \bigcup_{\ell \in I_0} \ell^\ast$, let
  $c_i^\ast := (\Gamma_\CalQ \, c)_i = \sum_{\ell \in i^\ast} c_\ell$ and choose
  some $h_i \in \Schwartz(\RR^d)$ such that
  $
    \big\|
      \Fourier^{-1} (\varphi_i^\ast \cdot \widehat{f} \, )
      - h_i
    \big\|_{L^p}
    \leq \delta \cdot c_i^\ast,
  $
  where $\delta := (\eps / 2)
                   \cdot \big(
                           C_{\Phi}
                           \cdot \| \Gamma_{\CalQ} \|_{\ell_w^q \to \ell_w^q}
                         \big)^{-2}
                   \cdot (1+\|c\|_{\ell^q_w})^{-1}$.
  This is possible since we have $p < \infty$, and since if $c_i^\ast = 0$, then
  \(
    \|\Fourier^{-1} (\varphi_i^\ast \cdot \widehat{f} \, )\|_{L^p}
    \leq \sum_{\ell \in i^\ast}
           \| \Fourier^{-1} (\varphi_\ell \widehat{f} \,) \|_{L^p}
    =    c_i^\ast
    = 0
  \).

  Define $g_i := \widehat{h_i} \in \Schwartz(\RHat^d)$ for $i \in I_0^\ast$,
  and $g_i := 0$ for $i \in I \setminus I_0^\ast$.
  We claim that
  \begin{equation}
    \big\|
      \Fourier^{-1} (\varphi_i^\ast \cdot \widehat{f} \, )
      - \Fourier^{-1} g_i
    \big\|_{L^p}
    \leq \Big(
           \Gamma_\CalQ \, \widetilde{c}
           + \delta \cdot \Gamma_\CalQ \, c
         \Big)_{i}
    \label{eq:DensityMainEstimate}
  \end{equation}
  for all $i \in I$.
  To show this, distinguish the two cases $i \in I_0^\ast$ and
  $i \in I \setminus I_0^\ast$.
  In the first case,
  \[
    \big\|
      \Fourier^{-1} (\varphi_i^\ast \cdot \widehat{f} \, ) - \Fourier^{-1} g_i
    \big\|_{L^p}
    = \big\|
        \Fourier^{-1} (\varphi_i^\ast \cdot \widehat{f} \, ) - h_i
      \big\|_{L^p}
    \leq \delta \cdot c_i^\ast = \delta \cdot (\Gamma_\CalQ \, c)_i
  \]
  by choice of $h_i$.
  Since, furthermore, $(\Gamma_\CalQ \, \widetilde{c} \, )_i \geq 0$, the estimate
  \eqref{eq:DensityMainEstimate} holds in the first case.
  For the second case, we have $g_i = 0$.
  Furthermore, $i \notin I_0^\ast$ and thus $\ell \notin I_0$
  for all $\ell \in i^\ast$.
  Therefore,
  \[
    \big\|
      \Fourier^{-1} (\varphi_i^\ast \cdot \widehat{f} \, ) - \Fourier^{-1} g_i
    \big\|_{L^p}
    = \big\|
        \Fourier^{-1} (\varphi_i^\ast \cdot \widehat{f} \, )
      \big\|_{L^p}
    \leq \sum_{\ell \in i^\ast}
         \big[
           \Indicator_{I \setminus I_0} (\ell)
           \cdot \big\|\Fourier^{-1} (\varphi_\ell \cdot \widehat{f} \, ) \big\|_{L^p}
         \big]
    = \big(\, \Gamma_\CalQ \, \widetilde{c} \,\big)_{i} \, .
  \]
  As in the first case, we thus see that estimate
  \eqref{eq:DensityMainEstimate} holds.

  Define $g := \Fourier^{-1} \big( \sum_{i \in I} \varphi_i \cdot g_i \big)$.
  Then $g \in \DenseSpace (\RR^d) $ since $g_i = 0$ for all but finitely many
  $i \in I$.
  Next, note that $\varphi_i \, \varphi_i^\ast = \varphi_i$, and hence
  \[
    \varphi_\ell \cdot \widehat{f - g}
    = \varphi_\ell \cdot \Big(
                           \sum_{i \in I}
                             \varphi_i \cdot \widehat{f}
                           - \sum_{i \in I}
                               \varphi_i \cdot g_i
                         \Big)
    = \varphi_\ell \cdot \sum_{i \in \ell^\ast}
                         \big[
                           \varphi_i \cdot (\varphi_i^\ast \, \widehat{f} - g_i)
                         \big] \, .
  \]
  Using Young's inequality, we thus get
  \begin{align*}
    \| \Fourier^{-1} (\varphi_\ell \cdot \widehat{f-g}) \|_{L^p}
    & \leq \|\Fourier^{-1} \varphi_\ell\|_{L^1}
           \cdot \sum_{i \in \ell^\ast}
                 \bigg(
                   \|\Fourier^{-1} \varphi_i\|_{L^1}
                   \cdot \|
                           \Fourier^{-1} (\varphi_i^\ast \, \widehat{f} \, )
                           - \Fourier^{-1} g_i
                         \|_{L^p}
                 \bigg) \\
    & \leq C_\Phi^2 \cdot \sum_{i \in \ell^\ast}
                          \big(
                            \Gamma_\CalQ \, \widetilde{c}
                            + \delta \cdot \Gamma_{\CalQ} \, c
                          \big)_i
      =   C_\Phi^2 \cdot \Big[
                           \Gamma_\CalQ
                           \big(
                            \Gamma_\CalQ \, \widetilde{c}
                            + \delta \cdot \Gamma_\CalQ \, c
                           \big)
                         \Big]_\ell \, ,
  \end{align*}
  where the last inequality follows by \eqref{eq:DensityMainEstimate}.
  This finally implies
  \[
    \|f - g\|_{\DecompSp(\CalQ,L^p,\ell_w^q)}
    \leq C_\Phi^2 \cdot \|\Gamma_\CalQ \|
                  \cdot \big(
                          \|\Gamma_\CalQ \, \widetilde{c}\|_{\ell_w^q}
                          + \delta \cdot \|\Gamma_\CalQ \, c\|_{\ell_w^q}
                        \big)
    \leq (C_\Phi \, \| \Gamma_\CalQ \| )^2
         \cdot (\|\widetilde{c}\|_{\ell_w^q} + \delta \cdot \|c\|_{\ell_w^q})
    \leq \eps ,
  \]
  which completes the proof of (ii).
\\\\
(iii) Since $\CalQ$ is a decomposition cover, the index set $I$ is countably infinite.
Indeed, the sets $\big( \varphi_i^{-1} (\CC \setminus \{0\}) \big)_{i \in I}$ form an open cover
of $\CalO$.
Since $\CalO$ is second countable, there is a countable $I_0 \subset I$ such that
\(
  \CalO
  \subset \bigcup_{i \in I_0}
            \varphi_i^{-1} (\CC \setminus \{0\})
  \subset \bigcup_{i \in I_0}
            Q_i
\).
Finally, for $i \in I$, we have
\(
  \emptyset
  \neq Q_i
  \subset \CalO
  \subset \bigcup_{\ell \in I_0}
            Q_\ell
\),
and hence $i \in \ell^\ast$ for some $\ell \in I_0$.
In other words, $I \subset \bigcup_{\ell \in I_0} \ell^\ast$ is countable as a countable union
of finite sets.
Finally, if $I$ was finite, then $\sum_{i \in I} \varphi_i \in C_c (\CalO)$,
in contradiction to $\CalO$ being open and to $\sum_{i \in I} \varphi_i \equiv 1$ on $\CalO$.
Thus, we can write $I = \{i_n \colon n \in \NN\}$ for pairwise distinct $(i_n)_{n \in \NN}$.

For each $i \in I$, we have $f_i := \Fourier^{-1} (\varphi_i \, \widehat{f} \,) \in L^p (\RR^d)$
with $\supp \widehat{f_i} \subset \supp \varphi_i \subset U_i$ for the open set
$U_i := (\varphi_i^\ast)^{-1} (\CC \setminus \{0\}) \subset Q_i^\ast \subset \CalO$,
since $\varphi_i^\ast \varphi_i = \varphi_i$.
Now, for each fixed $i \in I$, \cite[Lemma~3.2]{DecompositionEmbedding} yields a sequence
$(f_i^{(n)})_{n \in \NN}$ of Schwartz functions such that $|f_i^{(n)}| \leq |f_i|$ and
$f_i^{(n)} \xrightarrow[n\to\infty]{} f_i$ pointwise, and such that
$\supp \widehat{f_i^{(n)}} \subset B_{1/n} (\supp \varphi_i)$, where
$B_{1/n}(\supp \varphi_i) := \{\xi \in \RHat^d  :  \text{dist}(\xi, \supp \varphi_i) \leq n^{-1} \}$.
By choosing $N_i \in \NN$ with $B_{1/N_i} (\supp \varphi_i) \subset U_i$,
and by replacing $f_i^{(1)}, \dots, f_i^{(N_i)}$ by $f_i^{(N_i)}$,
we get $\supp \widehat{f_i^{(n)}} \subset U_i \subset Q_i^\ast \subset \CalO$
for all $i \in I$ and $n \in \NN$.

Note that we have $f_i^{(n)} \xrightarrow[n \to \infty]{\Schwartz'(\RR^d)} f_i$.
Indeed, if $p < \infty$, then this follows from $f_i^{(n)} \xrightarrow[n\to\infty]{L^p} f_i$,
which is a consequence of the dominated convergence theorem since $|f_i^{(n)}| \leq |f_i| \in L^p$
and $f_i^{(n)} \to f_i$ pointwise.
If $p = \infty$ and $h \in \Schwartz (\RR^d)$, then $f_i^{(n)} \cdot h \to f_i \cdot h$
pointwise, and we have the estimate
$|f_i^{(n)} \cdot h| \leq |f_i \cdot h| \leq \|f_i\|_{L^\infty} \cdot |h| \in L^1$, whence
\(
  \langle f_i^{(n)}, h \rangle_{\Schwartz', \Schwartz}
  \to \langle f_i, h \rangle_{\Schwartz',\Schwartz}
\)
by dominated convergence.

Now, define $g_N := \sum_{n=1}^N f_{i_n}^{(N)} \in \DenseSpace(\RR^d)$.
We first verify that $g_N \to f$ with convergence in $Z' (\CalO)$.
To see this, let $\psi \in Z(\CalO)$ be arbitrary.
Then $\Fourier^{-1} \psi \in C_c^\infty (\CalO)$,
so that $K := \supp \Fourier^{-1} \psi \subset \CalO$ is compact.
Precisely as in the proof of Part (i), we thus see that there is a finite set $I_K \subset I$
such that $Q_i \cap K = \emptyset$ for all $i \in I \setminus I_K$.
Therefore, $U_i \cap K \subset Q_i^\ast \cap K = \emptyset$,
and hence $\widehat{f_i^{(n)}} \cdot \Fourier^{-1} \psi \equiv 0$,
for all $i \in I \setminus I_K^\ast$.
Now, choose $N_0 = N_0 (K) \in \NN$ such that $I_K^\ast \subset \{ i_1,\dots,i_{N_0} \}$.
If $N \geq N_0$, we then have
\begin{align*}
  \langle g_N , \psi \rangle_{Z', Z}
    = \langle
        \widehat{g_N} ,
        \Fourier^{-1} \psi
      \rangle_{\CalD'(\CalO), \CalD}
   = \sum_{n=1}^N
        \langle
          \widehat{f_{i_n}^{(N)}},
          \Fourier^{-1} \psi
        \rangle_{\CalD'(\CalO), \CalD}
   = \sum_{i \in I_K^\ast}
        \langle
          \widehat{f_i^{(N)}}, \Fourier^{-1} \psi
        \rangle_{\CalD'(\CalO), \CalD},
\end{align*}
where the last equality follows since $\{i_1,\dots,i_N\} \supset I_K^\ast$ and
$\widehat{f_i^{(N)}} \cdot \Fourier^{-1} \psi \equiv 0$ for $i \in I \setminus I_K^\ast$.
Next, using that $f_i^{(N)} \to f_i $ in $\Schwartz'$ and noting that
\(
  \Fourier^{-1} \psi
  = \sum_{i \in I} \varphi_i \, \Fourier^{-1} \psi
  = \sum_{i \in I_K^\ast} \varphi_i \, \Fourier^{-1} \psi
\),
we see that
\begin{align*}
\langle g_N , \psi \rangle_{Z', Z}
   \xrightarrow[N\to\infty]{}
     & \sum_{i \in I_K^\ast}
        \langle
          \widehat{f_i} ,
          \Fourier^{-1} \psi
        \rangle_{\CalD'(\CalO), \CalD} \\
  & = \sum_{i \in I_K^\ast}
        \langle
          \varphi_i \widehat{f},
          \Fourier^{-1} \psi
        \rangle_{\CalD'(\CalO), \CalD} \\
  & = \langle \widehat{f}, \Fourier^{-1} \psi \rangle_{\CalD'(\CalO), \CalD} \\
   & = \langle f, \psi \rangle_{Z', Z} \, .
\end{align*}
Thus, $g_N \xrightarrow[N\to\infty]{} f$ with convergence in $Z'(\CalO)$.

Finally, we construct a sequence $F = (F_i)_{i \in I} \in \ell_w^q (I; L^p)$
such that each $g_N$ is $(F,\Phi)$-dominated.
To this end, set $F_i := \sum_{\ell \in i^{\ast \ast}} |\widecheck{\varphi_i}| \ast |f_\ell|$,
where $f_\ell := \Fourier^{-1} (\varphi_\ell \cdot \widehat{f})$.
Note because of $\supp \widehat{f_{i_n}^{(N)}} \subset Q_{i_n}^\ast$ that
$\varphi_i \cdot \widehat{f_{i_n}^{(N)}} \not\equiv 0$ can only hold for $i_n \in i^{\ast\ast}$.
Therefore, since $|f_i^{(m)}| \leq |f_i|$, we get
\begin{align*}
  \big|
    \Fourier^{-1} (\varphi_i \, \widehat{g_N} \,)
  \big|
  & = \bigg|
        \Fourier^{-1}
        \Big(
          \varphi_i
          \cdot \sum_{n \in \underline{N} \; : \; i_n \in i^{\ast \ast}}
                  \widehat{f_{i_n}^{(N)}}
        \Big)
      \bigg|
    \leq \sum_{n \in \underline{N} \; : \; i_n \in i^{\ast \ast}}
           \big|
             \Fourier^{-1} (\varphi_i \cdot \widehat{f_{i_n}^{(N)}})
           \big| \\
&    \leq \sum_{\ell \in i^{\ast \ast}}
           |\widecheck{\varphi_i}| \ast |f_\ell|
    =    F_i \, .
\end{align*}
Finally, setting $c = (c_i)_{i \in I}$ with
$c_i := \|\Fourier^{-1} (\varphi_i \, \widehat{f} \, )\|_{L^p}$, we see
because of $i^{\ast \ast} = \bigcup_{j \in i^\ast} j^\ast$ that
\[
  \|F_i\|_{L^p}
  \leq \sum_{\ell \in i^{\ast \ast}}
         \big\| |\Fourier^{-1} \varphi_i| \ast |f_\ell| \big\|_{L^p}
  \leq \sum_{j \in i^\ast}
         \sum_{\ell \in j^{\ast}}
           \| \Fourier^{-1} \varphi_i \|_{L^1} \cdot \| f_\ell \|_{L^p}
  \leq C_\Phi \cdot (\Gamma_{\CalQ} \, \Gamma_{\CalQ} \, c)_i \, .
\]
Thus, $F \in \ell_w^q(I; L^p)$ with
$\|F\|_{\ell_w^q (I; L^p)}
 \leq C_\Phi \, \|\Gamma_{\CalQ}\|_{\ell_w^q \to \ell_w^q}^2
      \cdot \|f\|_{\DecompSp(\CalQ,L^p,\ell_w^q)}$,
since $\|f\|_{\DecompSp(\CalQ,L^p,\ell_w^q)} = \|c\|_{\ell_w^q}$.
\hfill$\square$

\section{Proof of Proposition~\ref{prop:FourierMultiplierBound}}
\label{sec:FourierMultiplierProof}

Before proving Proposition~\ref{prop:FourierMultiplierBound},
we first collect a few properties of the ``generalized multiplication operation'' $\odot$
introduced in Definition~\ref{def:SpecialMultiplication}.

\begin{lemma}\label{lem:SpecialMultiplicationProperties}
  Let $p \in [1,\infty]$.
  For $f,g \in \Fourier L^1(\R^d)$ and $h \in \Fourier L^p(\R^d)$, the following properties hold:
  \begin{enumerate}
    \item[(i)] $f \odot (g \odot h) = (f \odot g) \odot h$.

    \item[(ii)] If $f \in \Schwartz(\RHat^d)$, then $f \odot h = f \cdot h$.

    \item[(iii)] If $p \in [1,2]$, then $f \odot h = f \cdot h$.

    \item[(iv)] We have $\supp (f \odot h) \subset \supp f \cap \supp h$,
          where the support is understood in the sense of tempered distributions.
  \end{enumerate}
\end{lemma}

\begin{proof}
  (i) Note that $\widecheck{f},\widecheck{g} \in L^1(\R^d)$ and $\widecheck{h} \in L^p(\R^d)$.
  Thus, Young's inequality shows for almost all $x \in \R^d$ that
  $(|\widecheck{f}| \ast (|\widecheck{g}| \ast |\widecheck{h}|))(x) < \infty$.
  For each such $x$, a standard calculation using Fubini's theorem shows
  $((\widecheck{f} \ast \widecheck{g}) \ast \widecheck{h})(x) = (\widecheck{f} \ast (\widecheck{g} \ast \widecheck{h}))(x)$.
  Hence, both sides are identical as tempered distributions.
  Thus, $(f \odot g) \odot h = f \odot (g \odot h)$.
\\\\
  (ii) This was already observed in Remark~\ref{rem:SpecialMultiplication}.
\\\\
 (iii)
  It is well-known that if $p \in [1,2]$, then
  $\widehat{\varphi \ast \psi} = \widehat{\varphi} \cdot \widehat{\psi}$
  for $\varphi \in L^1(\R^d)$ and $\psi \in L^p(\R^d)$.
  Indeed, for $\varphi,\psi \in \Schwartz(\R^d)$, the identity is clear;
  furthermore, it follows from the Hausdorff-Young inequality that as elements of $L^{p'}(\R^d)$,
  both sides of the identity depend continuously on $\varphi \in L^1(\R^d)$ and $\psi \in L^p(\R^d)$.
  Therefore, $f \odot h = \Fourier[\widecheck{f} \ast \widecheck{h}] =  f \cdot h$.
\\\\
  (iv) Let $\varphi \in C_c^\infty(\RHat^d)$ with
  $\supp \varphi \subset \RHat^d \setminus \supp f$.
  There is $\psi \in C_c^\infty (\RHat^d)$ with
  $\varphi = \varphi \cdot \psi$ and $\supp \psi \subset \RHat^d \setminus \supp f$.
  Furthermore, by combining Properties (i) and (ii), we see that
  \[
    \psi \cdot (f \odot h)
    = \psi \odot (f \odot h)
    = (\psi \odot f) \odot h
    = (\psi \cdot f) \odot h
    = 0.
  \]
  Because of $\varphi = \psi \cdot \varphi$, this entails
  \(
    \langle f \odot h, \varphi \rangle_{\Schwartz',\Schwartz}
    = \langle \psi \cdot (f \odot h) , \varphi \rangle_{\Schwartz',\Schwartz}
    = 0
  \).
  Since this holds for every $\varphi \in C_c^\infty (\RHat^d)$ with
  $\supp \varphi \subset \RHat^d \setminus \supp f$, we see $\supp (f \odot h) \subset \supp f$.
  The argument for $\supp (f \odot h) \subset \supp h$ is similar.
\end{proof}

With this preparation, we can now provide the proof of Proposition~\ref{prop:FourierMultiplierBound}.

\begin{proof}[Proof of Proposition \ref{prop:FourierMultiplierBound}]
  Before proving the claims, we show that $\Phi_h$ is well-defined,
  with unconditional convergence in $Z'(\CalO)$ of the defining series.
  For brevity, let
  $\psi_i := \Fourier^{-1}[(\varphi_i^\ast h) \odot (\varphi_i \widehat{f})] \in \Schwartz'(\R^d)$.
  This is well-defined since \eqref{eq:TameCondition}
  implies $\varphi_i \, h \in \Fourier L^1$, and
  $\varphi_i^\ast \, h = \sum_{\ell \in i^\ast} \varphi_\ell \, h \in \Fourier L^1(\R^d)$.
  \smallskip{}

  Since $\Fourier : Z'(\CalO) \to \CalD'(\CalO)$ is an isomorphism, it is enough to show that
  the series $\sum_{i \in I} \Fourier \psi_i$ converges unconditionally in $\CalD'(\CalO)$.
  To see this, note that $\supp \widehat{\psi_i} \subset \supp \varphi_i \subset \overline{Q_i}$
  for all $i \in I$, by Property (iv) of Lemma~\ref{lem:SpecialMultiplicationProperties}.
  Therefore, $\sum_{i \in I} \Fourier \psi_i$ converges unconditionally in
  $\CalD'(\CalO)$ as a locally finite%
  \footnote{Here, we use that if $\xi_0 \in \CalO$ is arbitrary, then $\xi_0 \in Q_\ell$
  for some $\ell \in I$ and hence $\varphi_\ell^\ast (\xi_0) = 1$.
  Thus, ${U := \{ \xi \in \CalO \colon |\varphi_\ell^\ast (\xi)| > 1/2 \}} \subset Q_\ell^\ast$
  is an open neighborhood of $\xi_0$; finally, if $U \cap \overline{Q_i} \neq \emptyset$,
  then also $U \cap Q_i \neq \emptyset$ and hence
  $i \in \ell^{\ast \ast} = \bigcup_{j \in \ell^\ast} j^\ast$,
  proving that the family $( \, \overline{Q_i} \, )_{i \in I}$ is locally finite on $\CalO$.}
  sum of (tempered) distributions.
\\\\
  (ii)  As above, let $\psi_i^{(n)} := \Fourier^{-1}[(\varphi_i^\ast h) \odot (\varphi_i \widehat{f_n})]$.
  Note that $\widehat{f_n} \to \widehat{f}$ in $\CalD'(\CalO)$, since $f_n \to f$ in $Z'(\CalO)$.
  Thus, setting $e_x : \RHat^d \to \CC, \xi \mapsto e^{2 \pi i \langle x,\xi \rangle}$
  for $x \in \R^d$, an application of \cite[Theorem 7.23]{RudinFunctionalAnalysis}
  shows that
  \[
    \Fourier^{-1} (\varphi_i \widehat{f} \, ) (x)
    = (\varphi_i \widehat{f} \, )(e_x)
    = \widehat{f} (\varphi_i e_x)
    = \lim_{n \to \infty} \widehat{f_n} (\varphi_i e_x)
    = \lim_{n \to \infty} \Fourier^{-1} (\varphi_i \widehat{f_n}) (x)
  \]
  for all $i \in I \text{ and } x \in \R^d$. Therefore, using that
  \(
    \langle F \ast G, \varphi \rangle_{\Schwartz',\Schwartz}
    = \int_{\R^d} G(x) \cdot (\varphi \ast \widetilde{F})(x) \, dx
  \)
  with $\widetilde{F}(x) = F(-x)$ for $ F \in L^1, G \in L^p$,
  and the estimate $|\Fourier^{-1}(\varphi_i \widehat{f_n})| \leq F_i \in L^p (\R^d)$,
  we get by the dominated convergence theorem
  \begin{equation}
    \begin{split}
      \langle \psi_i^{(n)}, \varphi \rangle_{\Schwartz',\Schwartz}
      & = \langle
            \Fourier^{-1}(\varphi_i^\ast h) \ast \Fourier^{-1} (\varphi_i \widehat{f_n}),
            \varphi
          \rangle_{\Schwartz',\Schwartz}
        = \int_{\R^d}
            \Fourier^{-1}(\varphi_i \widehat{f_n})(x)
            \cdot (\varphi \ast \widehat{\varphi_i^\ast h})(x)
          \, d x \\
      & \xrightarrow[n\to\infty]{}
        \int_{\R^d}
          \Fourier^{-1}(\varphi_i \widehat{f})(x)
          \cdot (\varphi \ast \widehat{\varphi_i^\ast h})(x)
        \, d x
        = \langle \psi_i, \varphi \rangle_{\Schwartz',\Schwartz}
    \end{split}
    \label{eq:FourierMultiplierDominatednessMainIdentity}
  \end{equation}
  for all $\varphi \in \Schwartz(\R^d)$ and $i \in I$.
  Here, we used that
  $\varphi \ast \widehat{\varphi_i^\ast h} \in L^1(\R^d) \cap L^\infty (\R^d) \subset L^{p'} (\R^d)$.

  \smallskip{}

  Now, let $\varphi \in Z(\CalO)$ be arbitrary, so that $\Fourier^{-1} \varphi \in C_c^\infty(\CalO)$.
  Then there is a finite set $I_\varphi \subset I$ such that
  $\supp \Fourier^{-1} \varphi \subset \overline{Q_i}^c$ for all $i \in I \setminus I_\varphi$.
  Since $\supp \Fourier \psi_i \subset \overline{Q_i}$
  and $\supp \Fourier \psi_i^{(n)} \subset \overline{Q_i}$,
  this implies
  \(
    \langle \psi_i , \varphi \rangle_{Z',Z}
    = \langle \Fourier \psi_i, \Fourier^{-1} \varphi \rangle_{\CalD', C_c^\infty}
    = 0
  \)
  for all $i \in I \setminus I_\varphi$. The same holds for $\psi_i$ replaced by $\psi_i^{(n)}$.
  Thus,
  \[
    \langle \Phi_h f_n, \varphi \rangle_{Z',Z}
    = \sum_{i \in I_\varphi}
        \langle \psi_i^{(n)}, \varphi \rangle_{\Schwartz', \Schwartz}
    \xrightarrow[n\to\infty]{} \sum_{i \in I_\varphi}
                                 \langle \psi_i, \varphi \rangle_{\Schwartz', \Schwartz}
    = \langle \Phi_h f, \varphi \rangle_{Z',Z} .
  \]
  This shows that $\Phi_h f_n \to \Phi_h f$ with convergence in $Z'(\CalO)$.

  Finally, we see for $\ell \in I$ directly by definition of $\psi_i^{(n)}$ and by
  definition of the ``extended multiplication'' $\odot$ that
  \begin{align*}
    \Fourier^{-1}(\varphi_\ell \, \widehat{\psi_i^{(n)}})
    & = \Fourier^{-1} \varphi_\ell
        \ast \Fourier^{-1}(\varphi_i^\ast h)
        \ast \Fourier^{-1}(\varphi_i \, \widehat{f_n})
      = \Fourier^{-1}(\varphi_i^\ast h)
        \ast \Fourier^{-1} (\varphi_i \, \varphi_\ell \, \widehat{f_n}) \\
    & = \Fourier^{-1}(\varphi_i^\ast h)
        \ast \Fourier^{-1} (\varphi_i)
        \ast \Fourier^{-1}(\varphi_\ell \widehat{f_n}).
  \end{align*}
  This shows that
  \(
    \vphantom{\sum_j}
    \Fourier^{-1} (\varphi_\ell \, \widehat{\psi_i^{(n)}})
    = \Fourier^{-1}(\varphi_i^\ast \, h)
      \ast \Fourier^{-1}(\varphi_i \, \varphi_\ell \, \widehat{f_n})
    = 0
  \)
  if $\ell \in I \setminus i^\ast$, since then $\varphi_i \, \varphi_\ell \equiv 0$.
  Therefore, since $|\Fourier^{-1}(\varphi_\ell \, \widehat{f_n} \,)| \leq F_\ell$, we see
  \begin{align*}
    |\Fourier^{-1}(\varphi_\ell \cdot \widehat{\Phi_h f_n})|
    & \leq \sum_{i \in \ell^\ast}
             |\Fourier^{-1}(\varphi_\ell \cdot \widehat{\psi_i^{(n)}})|
      \leq \sum_{i \in \ell^\ast}
             |\Fourier^{-1}(\varphi_i^\ast h)|
             \ast |\Fourier^{-1}(\varphi_i)|
             \ast |\Fourier^{-1}(\varphi_\ell \widehat{f_n})| \\
    & \leq \sum_{i \in \ell^\ast}
             |\Fourier^{-1}(\varphi_i^\ast h)| \ast |\Fourier^{-1}(\varphi_i)| \ast F_\ell
      =: G_\ell .
  \end{align*}
  In view of Young's inequality, we see
  \[
    \|G_\ell\|_{L^p}
    \leq \sum_{i \in \ell^\ast}
           \| \Fourier^{-1}(\varphi_i^\ast h)\|_{L^1} \,
           \|\Fourier^{-1}(\varphi_i)\|_{L^1} \,
           \|F_\ell\|_{L^p}
    \leq N_\CalQ^2 C_{\Phi} C_h \cdot \|F_\ell\|_{L^p},
  \]
  and hence $\|G\|_{\ell_w^q(I;L^p)} \leq N_\CalQ^2 C_{\Phi} C_h \cdot \|F\|_{\ell_w^q(I;L^p)} < \infty$,
  so that indeed each $\Phi_h f_n$ is $(G,\Phi)$-dominated.
\\\\  (i) By applying Property (ii) to the constant sequence given by $f_n = f$ for all $n \in \N$
  and with $F_i := |\Fourier^{-1}(\varphi_i \widehat{f})|$, we see that
  $\Phi_h f$ is $(G,\Phi)$-dominated for a function $G \in \ell_w^q(I;L^p)$ satisfying
  \(
    \|G\|_{\ell_w^q(I;L^p)}
    \leq N_\CalQ^2 C_{\Phi} C_h \cdot \|F\|_{\ell_w^q(I;L^p)}
    =    N_\CalQ^2 C_{\Phi} C_h \cdot \|f\|_{\DecompSp(\CalQ,L^p,\ell_w^q)}
  \).
  This proves the claim.
\\\\
  (iii) If $\widehat{f} \in C_c(\CalO)$, then $\varphi_i \, \widehat{f} \in C_c(\CalO) \subset L^2(\RHat^d)$,
  so that
  \(
    (\varphi_i^\ast h) \odot (\varphi_i \widehat{f})
    = (\varphi_i^\ast h) \cdot (\varphi_i \widehat{f})
    = \varphi_i \cdot h \widehat{f}
  \);
  see Lemma~\ref{lem:SpecialMultiplicationProperties}(iii).
  Since $h \widehat{f} \in C_c(\CalO)$, it follows
  \(
    h \widehat{f} = \sum_{i \in I} [ \varphi_i \cdot h \widehat{f} ]
  \),
  where only finitely many terms do not vanish.
  Hence, by definition of $\Phi_h f$,
  \[
    \Phi_h f
    = \sum_{i \in I}
        \Fourier^{-1} [(\varphi_i^\ast h) \odot (\varphi_i \widehat{f})]
    = \Fourier^{-1}
      \Big[
        \sum_{i \in I}
          \varphi_i \cdot h \widehat{f} \,
      \Big]
    = \Fourier^{-1} (h \cdot \widehat{f}).
  \]
\\\\
  (iv) We have
  \begin{align*}
    \big\| \Fourier^{-1} \big( \varphi_i \cdot (g \cdot h) \big) \big\|_{L^1}
    &= \| \Fourier^{-1} (\varphi_i g \cdot \varphi_i^\ast h) \|_{L^1}
    \leq \sum_{\ell \in i^\ast}
           \|\Fourier^{-1}(\varphi_i \, g)\|_{L^1}
           \cdot \|\Fourier^{-1}(\varphi_\ell \, h)\|_{L^1} \\
  &  \leq N_\CalQ \, C_g \, C_h < \infty,
  \end{align*}
  so that $g \cdot h$ is tame.
  Part (iii) shows for $f \in \DenseSpace (\R^d)$ that $\Phi_g f = \Fourier^{-1}(g \widehat{f})$,
  which in particular implies $\Fourier [\Phi_g f] \in C_c (\CalO)$.
  Thus, by Part (iii) again,
  \(
    \Phi_h \Phi_g f
    = \Fourier^{-1} [h \cdot \Fourier[\Phi_g f]]
    = \Fourier^{-1} (h g \widehat{f})
    = \Phi_{g h} f
  \).
  Finally, for arbitrary $f \in \DecompSp(\CalQ,L^p,\ell_w^q)$, Proposition~\ref{prop:Density}
  yields a sequence $(f_n)_{n \in \N} \subset \DenseSpace(\RR^d)$ which is $(F,\Phi)$-dominated
  for some $F \in \ell_w^q (I;L^p)$ and such that $f_n \to f$ in $Z'(\CalO)$.
  By Part (ii), this implies $\Phi_{g h} f_n \to \Phi_{g h} f$ and $\Phi_g f_n \to \Phi_g f$
  in $Z'(\CalO)$. Furthermore, there is $G \in \ell_w^q(I; L^p)$ such that
  each $\Phi_{g} f_n$ is $(G,\Phi)$-dominated.
  Thus, a final application of Part (ii) implies
  \[
    \Phi_{g h} f
    = \lim_{n \to \infty} \Phi_{g h} f_n
    = \lim_{n \to \infty} \Phi_h [\Phi_g f_n]
    = \Phi_h [\Phi_g f],
  \]
  which completes the proof.
\end{proof}

\section{Other auxiliary results}
\label{sec:miscellaneous}

\subsection{An estimate for
            \texorpdfstring{the series $\sum_{k \in \ZZ^d} (1 + |\eta + A k|)^{-(d+1)}$}
                           {a certain series}}
\label{sub:series_estimate}

\begin{lemma}\label{lem:series_estimate}
 For $\eta \in \RR^d$ and $A \in \GL(d, \RR)$,
 \[
   \sum_{k \in \ZZ^d} (1 + |\eta + A k|)^{-(d+1)}
   \leq (d+1) \cdot 2^{1+2d} \cdot \max \{1, \|A^{-1} \|^{d+1} \} \, .
 \]
\end{lemma}

\begin{proof}
 First, note that the function
 \(
   \Theta : \RR^d \to [0,\infty], \;
            x \mapsto \sum_{k \in \ZZ^d} (1 + |x+k|)^{-(d+1)}
 \)
 is $\ZZ^d$-periodic, and hence
 $\|\Theta\|_{\sup} = \|\Theta|_{[0,1)^d}\|_{\sup}$.
 For $x \in [0,1)^d$, we have
 \(
   \|k\|_{\infty}
   \leq 1 + \|x+k\|_{\infty}
   \leq 1 + |x+k| \, ,
 \)
 and thus $1 + \|k\|_{\infty} \leq 2 (1 + |x+k|)$.
 Therefore, $\Theta(x) \leq 2^{d+1} \cdot \sum_{k \in \ZZ^d}
                                             (1 + \|k\|_{\infty})^{-(d+1)}$.
 In order to estimate this last term, we rewrite it using \cite[Proposition~6.24]{FollandRA}
 as
 \[
   \sum_{k \in \mathbb{Z}^d} (1 + \| k \|_{\infty} )^{-(d+1)}
   = \int_0^\infty
       | \{ k \in \ZZ^d \; \colon \; (1+\|k\|_{\infty})^{-(d+1)} > \lambda \}|
     \; d\lambda .
 \]
 Let $f : \ZZ^d \to (0,1] , k \mapsto (1 + \|k\|_{\infty})^{-(d+1)}$.
 For $\lambda \geq 1$, clearly $\{k \in \ZZ^d \,:\, f(k) > \lambda\} = \emptyset$.
 In contrast, for $\lambda \in (0,1)$,
 \begin{align*}
   \{ k \in \ZZ^d \,:\, f(k) > \lambda \}
   & \subset \,\, \{
                    k \in \ZZ^d
                    \,:\,
                    \|k\|_{\infty} \leq \lambda^{-1/(d+1)} - 1
                  \} \\
   & \subset \,\, \Big\{
                    k \in \ZZ^d
                    \,:\,
                    \forall \, n \in \underline{d}:
                      k_n \in \big\{
                                - \lfloor \lambda^{-1/(d+1)} - 1 \rfloor,
                                \dots,
                                  \lfloor \lambda^{-1/(d+1)} - 1 \rfloor
                              \big\}
                  \Big\}
   \, ,
 \end{align*}
 and thus
 \(
   |\{k \in \ZZ^d \,:\, f(k) > \lambda \}|
   \leq \big( 1 + 2 \lfloor \lambda^{-1/(d+1)} - 1 \rfloor \big)^d
   \leq 2^d \cdot \lambda^{-d/(d+1)}
 \),
 which  implies
 \[
   \Theta (x)
   \leq 2^{d+1} \sum_{k \in \ZZ^d}
                  (1 + \|k\|_{\infty})^{-(d+1)}
   \leq 2^{1+2d} \int_0^1 \! \lambda^{-\frac{d}{d+1}} \, d\lambda
   =    (d+1) \cdot 2^{1+2d}
 \]
 for all $x \in [0,1)^d$, whence $\Theta(x) \leq (d+1) \cdot 2^{1 + 2d}$ for all $x \in \RR^d$.

 Now, let $A \in \GL(d,\R)$ be arbitrary.
 Then
 \[
   1 + |k + A^{-1} \eta|
   \leq 1 + \|A^{-1}\| \cdot |A (k + A^{-1} \eta)|
   \leq \max \big\{1, \|A^{-1}\| \big\}
        \cdot \big( 1 + |A (k + A^{-1} \eta)| \big) \, ,
 \]
 and hence
 $\big( 1 + |\eta + Ak| \big)^{-(d+1)}
  = \big( 1 + |A(k + A^{-1} \eta)| \big)^{-(d+1)}
  \leq \max \big\{ 1 , \|A^{-1}\|^{d+1} \}
       \cdot \big( 1 + |k + A^{-1} \eta| \big)^{-(d+1)}$.
 Overall, we see for arbitrary $\eta \in \RR^d$ and $A \in \GL(d, \R)$ that
 \begin{equation}
   \begin{split}
     \sum_{k \in \ZZ^d} \big( 1 \!+\! |\eta + A k| \big)^{-(d+1)}
     & \leq \max \big\{ 1, \|A^{-1}\|^{d+1} \big\}
            \cdot \sum_{k \in \ZZ^d}
                    \big( 1 + |k \!+\! A^{-1} \eta| \big)^{-(d+1)} \\
     & = \max \big\{ 1, \|A^{-1}\|^{d+1} \big\}
         \cdot \Theta(A^{-1} \eta) \\
         &
       \leq (d \!+\! 1) \cdot 2^{1+2d} \cdot \max \big\{ 1, \|A^{-1}\|^{d+1} \big\} ,
   \end{split}
   \label{eq:GeneralDilatedSeriesEstimate}
 \end{equation}
 finishing the proof.
 \end{proof}

As a corollary, we get the following estimate for the series where we sum over
$k \in \ZZ^d \setminus \{0\}$ instead of $k \in \ZZ^d$.

\begin{corollary}\label{cor:nonzero_series_estimate}
  For $\eta \in \RR^d$ and $A \in \GL (d,\RR)$, we have
  \[
    \sum_{k \in \ZZ^d \setminus \{0\}}
      (1 + |\eta + A k|)^{-(d+1)}
    \leq (d+1) \cdot 2^{3 + 4d}
         \cdot (1 + |\eta|)
         \cdot \max \big\{ \|A^{-1}\|, \|A^{-1}\|^{d+1} \big\} .
  \]
\end{corollary}

\begin{proof}
  We distinguish two cases.

  First, suppose $|A^{-1} \eta| \leq \tfrac{1}{3}$.
  Then, noting that $|k| \geq 1$ for all $k \in \ZZ^d \setminus \{0\}$, we get the estimate
  \(
    |k + A^{-1} \eta|
    \geq |k| - |A^{-1} \eta|
    \geq \tfrac{|k|}{2} + \tfrac{1}{2} - |A^{-1} \eta|
    \geq \tfrac{|k|}{2}
    \geq \tfrac{1 + |k|}{4}.
  \)
  Next, note that $|x| = |A^{-1} A x| \leq \|A^{-1}\| \, |A x|$,
  and hence $|A x| \geq \|A^{-1}\|^{-1} \, |x|$ for all $x \in \RR^d$.
  This implies
  \[
    1 + |\eta + A k|
    \geq |A k + \eta|
    \geq \|A^{-1}\|^{-1} \cdot |k + A^{-1} \eta|
    \geq \frac{\|A^{-1}\|^{-1}}{4} \cdot (1 + |k|).
  \]
  Now, Lemma~\ref{lem:series_estimate} shows that
  \begin{align*}
    \sum_{k \in \ZZ^d \setminus \{0\}}
      (1 + |\eta + A k|)^{-(d+1)}
    & \leq 4^{d+1} \|A^{-1}\|^{d+1}
           \cdot \!\! \sum_{k \in \ZZ^d \setminus \{0\}}
                        (1 + |k|)^{-(d+1)} \\
&      \leq (d+1) \cdot 2^{3 + 4 d} \cdot \|A^{-1}\|^{d+1} \\
    & \leq (d+1) \cdot 2^{3 + 4 d}
           \cdot (1 + |\eta|)
           \cdot \max \big\{ \|A^{-1}\|, \|A^{-1}\|^{d+1} \big\}.
  \end{align*}

  For the other case, suppose $|A^{-1} \eta| > \tfrac{1}{3}$.
  Then
  $(1 + |\eta|) \, \|A^{-1}\| \geq \|A^{-1}\| \cdot |\eta| \geq |A^{-1} \eta| > \tfrac{1}{3}$,
  and
  \begin{align*}
    \max \{1, \|A^{-1}\|^{d+1} \}
  &  \leq \max \big\{ 3 (1 + |\eta|) \, \|A^{-1}\|, \|A^{-1}\|^{d+1} \big\} \\
  &  \leq 4 \, (1 + |\eta|) \cdot \max \big\{ \|A^{-1}\|, \|A^{-1}\|^{d+1} \big\} .
  \end{align*}
  Now, an application of Lemma~\ref{lem:series_estimate} shows that
  \begin{align*}
    \smash{\sum_{k \in \ZZ^d \setminus \{0\}}}
      (1 + |\eta + A k|)^{-(d+1)}
    & \leq (d+1) \, 2^{1 + 2d} \cdot \max \big\{1, \|A^{-1}\|^{d+1} \big\} \\
    & \leq (d+1) \, 2^{3 + 2d}
           \cdot (1 + |\eta|)
           \cdot \max \big\{ \|A^{-1}\|, \|A^{-1}\|^{d+1} \big\} .
  \end{align*}
  Together with the first case, this shows that the claimed estimate always holds.
\end{proof}

\subsection{Proof of Lemma~\ref{lem:FactorizationLemma}}
\label{sub:ProofFactorizationLemma}

For brevity, set $\llangle \xi \rrangle := 1 + |\xi|^2$ for $\xi \in \RHat^d$.
With this notation, \cite[Lemma~6.8]{StructuredBanachFrames}
shows for arbitrary $\theta \in \RR$ and $\alpha \in \NN_0^d$ that there
is a polynomial $P_{\theta,\alpha} \in \RR[\xi_1,\dots,\xi_d]$ such that,
for all $\xi \in \RHat^d$,
\begin{equation}
  \partial^\alpha \llangle \xi \rrangle^\theta
  = \llangle \xi \rrangle^{\theta - |\alpha|} \cdot P_{\theta,\alpha}(\xi)
  \qquad \text{and} \qquad
  |P_{\theta,\alpha}(\xi)|
  \leq C_{\theta,\alpha} \cdot (1 + |\xi|)^{|\alpha|} ,
  \label{eq:ChineseBracketDerivativeEstimate}
\end{equation}
where $C_{\theta,\alpha} = |\alpha|! \cdot [2 (1 + d + |\theta|)]^{|\alpha|}$.
Since $(1 + |\xi|)^k \leq 2^k \cdot \llangle \xi \rrangle^{k/2}$ for all $k \geq 0$,
it follows that
\begin{equation}
  (1 + |\xi|)^{|\alpha|} \cdot \llangle \xi \rrangle^{\theta - |\alpha|}
  \leq 2^{|\alpha|} \cdot \llangle \xi \rrangle^{\theta - |\alpha|/2}
  \leq 2^{|\alpha|} \cdot \llangle \xi \rrangle^{\theta}
  \label{eq:ElementaryChineseBracketEstimate}
\end{equation}
for all $\xi \in \RHat^d$, $\theta \in \RR$ and $\alpha \in \NN_0^d$.
Next, for $\theta = -\frac{1}{2}(d+1)$ and any $\alpha \in \NN_0^d$
with $|\alpha| \leq d+1$,
\[
  C_{-(d+1)/2,\alpha}
  = |\alpha|! \cdot \Big[
                      2 \Big( 1 + d + \Big|-\frac{d+1}{2}\Big| \Big)
                    \Big]^{|\alpha|}
  \leq (d+1)! \cdot \big[ 3 \cdot (d+1) \big]^{|\alpha|}
  \leq \big( 3 \cdot (d+1)^2 \big)^{d+1} \, .
\]
Combining Equations~\eqref{eq:ChineseBracketDerivativeEstimate}
and \eqref{eq:ElementaryChineseBracketEstimate} with the elementary estimate
$1 + |\xi| \leq 2 \llangle \xi \rrangle^{1/2}$, we see that
\begin{align*}
  \max_{|\alpha| \leq d+1}
    \big| \partial^\alpha h_2 (\xi) \big|
  &= \max_{|\alpha| \leq d+1}
      \big| \partial^\alpha \llangle \xi \rrangle^{-(d+1)/2} \big| \\
  & \leq \big( 3  (d+1)^2 \big)^{d+1}
          \max_{|\alpha| \leq d+1}
                 (1 + |\xi|)^{|\alpha|} \,
                 \llangle \xi \rrangle^{-\frac{d+1}{2} - |\alpha|} \\
  & \leq \big( 6  (d+1)^2 \big)^{d+1}  \llangle \xi \rrangle^{-\frac{d+1}{2}}
    \leq C' \cdot (1 + |\xi|)^{-(d+1)} \, .
\end{align*}
For the estimate concerning $h_1$, note that
since $C_{\theta,\alpha} = C_{-\theta,\alpha}$,  we also have
$C_{(d+1)/2, \beta} \leq \big(3 \cdot (d+1)^2 \big)^{d+1}$
for all $\beta \in \NN_0^d$ with $|\beta| \leq d+1$.
Hence, using the Leibniz rule and
Equations~\eqref{eq:ChineseBracketDerivativeEstimate} and
\eqref{eq:ElementaryChineseBracketEstimate},
it follows for arbitrary $\xi \in \RHat^d$ that
\begin{align*}
  \max_{|\alpha| \leq d + 1}
    \big| \partial^\alpha h_1 (\xi) \big|
  & \leq \sum_{\beta \leq \alpha}
           \binom{\alpha}{\beta}
           \, \big| \partial^\beta \llangle \xi \rrangle^{(d+1)/2} \big|
           \cdot \big| \partial^{\alpha - \beta} g (\xi) \big| \\
  & \leq \varrho(\xi) \cdot (1 + |\xi|)^{-(d+1)}
           \sum_{\beta \leq \alpha}
                  \binom{\alpha}{\beta}
                  \, C_{(d+1)/2, \beta}
                  \,\, (1 + |\xi|)^{|\beta|}
                  \,\, \llangle \xi \rrangle^{\frac{d+1}{2} - |\beta|} \\
  & \leq \big( 6 \cdot (d+1)^2 \big)^{d+1}
           \varrho(\xi)
           (1 + |\xi|)^{-(d+1)} \, \llangle \xi \rrangle^{(d+1)/2}
           \sum_{\beta \leq \alpha}
                  \binom{\alpha}{\beta} \\
&    \leq C'  \varrho(\xi) \, ,
\end{align*}
which completes the proof.\hfill $\square$

\end{document}